\documentclass[reqno,11pt]{amsart}
\usepackage{amssymb, amsmath,latexsym,amsfonts,amsbsy, amsthm}
\usepackage{graphicx}
\usepackage{color}

\newcommand{\non}{\nonumber}

\def\eqdefa{\buildrel\hbox{\footnotesize def}\over =}

\makeatletter

\newcommand{\Rmnum}[1]{\expandafter\@slowromancap\romannumeral #1@}

\newtheorem{athm}{\bf \t}[section]
\newenvironment{thm} [1] {\def\t{#1}\begin{athm} \bf \rm} {\end{athm}}
\newcommand{\bthm}{\begin{thm}}
\newcommand{\ethm}{\end{thm}}

\newtheorem{theorem}{Theorem}[section]
\newtheorem{lemma}{Lemma}[section]
\newtheorem{remark}{Remark}[section]
\newtheorem{corollary}{Corollary}[section]

\newtheorem{proposition}{Proposition}[section]

\newcommand{\beq}{\begin{equation}}
\newcommand{\eeq}{\end{equation}}
\newcommand{\ben}{\begin{eqnarray}}
\newcommand{\een}{\end{eqnarray}}
\newcommand{\beno}{\begin{eqnarray*}}
\newcommand{\eeno}{\end{eqnarray*}}
\newcommand{\bali}{\begin{aligned}}
\newcommand{\eali}{\end{aligned}}

\numberwithin{equation}{section}

\newcommand{\al}{\alpha}
\newcommand{\be}{\beta}

\newcommand{\ve}{\varepsilon}

\newcommand{\f}{\frac}
\newcommand{\na}{\nabla}

\newcommand{\ud}{\mathrm{d}}

\newcommand{\vv}{\mathbf{v}}

\newcommand{\xx}{\mathbf{x}}

\newcommand{\nn}{\mathbf{n}}

\newcommand{\hh}{\mathbf{h}}
\newcommand{\mm}{\mathbf{m}}
\newcommand{\pp}{\mathbf{p}}

\newcommand{\A}{\mathbf{A}}
\newcommand{\B}{\mathbf{B}}

\newcommand{\GG}{\mathbf{G}}
\newcommand{\HH}{\mathbf{H}}

\newcommand{\NN}{\mathbf{N}}

\newcommand{\QQ}{\mathbf{Q}}
\newcommand{\DD}{\mathbf{D}}
\newcommand{\FF}{\mathbf{F}}
\newcommand{\JJ}{\mathbf{J}}
\newcommand{\II}{\mathbf{I}}

\newcommand{\CF}{\mathcal{F}}

\newcommand{\CU}{\mathcal{U}}

\newcommand{\CP}{\mathcal{P}}
\newcommand{\CH}{\mathcal{H}}
\newcommand{\CL}{\mathcal{L}}
\newcommand{\CQ}{\mathcal{Q}}
\newcommand{\CM}{\mathcal{M}}

\newcommand{\CJ}{\mathcal{J}}

\newcommand{\Ff}{\mathfrak{F}}
\newcommand{\Ef}{\mathfrak{E}}

\newcommand{\BS}{{\mathbb{S}^2}}
\newcommand{\BR}{{\mathbb{R}^3}}
\newcommand{\BOm}{\mathbf{\Omega}}

\newcommand{\pa}{\partial}

\newcommand{\QI}{\mathbb{Q}_{\nn}^{\mathrm{in}}}
\newcommand{\QO}{\mathbb{Q}_{\nn}^{\mathrm{out}}}
\newcommand{\QP}{\mathbb{Q}_{phy}}
\newcommand{\QPd}{\mathbb{Q}_{phy,\delta}}

\begin{document}

\title[Well-posedness and small Deborah limit of a $Q$-tensor model]
{Local well-posedness and small Deborah limit of a molecule-based $Q$-tensor system}

\author{Sirui Li}
\address{School of  Mathematical Sciences and LMAM, Peking University, Beijing 100871, China}
\email{srli@pku.edu.cn}

\author{Wei Wang}
\address{Beijing International Center for Mathematical Research, Peking University, Beijing 100871, China}
\email{wangw07@pku.edu.cn}

\author{Pingwen Zhang}
\address{School of  Mathematical Sciences and LMAM, Peking University, Beijing 100871, China}
\email{pzhang@pku.edu.cn}

\maketitle
\begin{abstract}
In this paper, we consider a hydrodynamic $Q$-tensor system for nematic liquid crystal flow,
which is derived from Doi-Onsager molecular theory by the Bingham closure.
We first prove the existence and uniqueness of local strong solution. Furthermore, by taking Deborah number goes to zero and using the Hilbert expansion method, we present a rigorous derivation
from the molecule-based $Q$-tensor theory to the Ericksen-Leslie theory.
\end{abstract}

\tableofcontents

\section{Introduction}
Liquid crystals are a state of matter whose properties are
intermediate between those of the conventional isotropic liquid and the crystalline solid.
The nematic, composed by rod-like molecules, is the simplest liquid crystal phase,
exhibiting long-range orientational order but no positional order. We refer \cite{DG} for a comprehensive elaboration
of the physics of liquid crystals.
There are three main theories to model the nematic liquid crystals: the Doi-Onsager theory,
the Landau-de Gennes theory and the Ericksen-Leslie
theory. The first is microscopic theory derived from viewpoints of statistical mechanics, and the later two are macroscopic theories
based on continuum mechanics. 

{\bf Notations and conventions.}
The Einstein convention will be assumed throughout the paper.
We introduce the following notations for the space of symmetric traceless tensors
\begin{align}
\mathbb{Q}&~\eqdefa\big\{Q\in \mathbb{R}^{3\times3}:~Q_{ij}=Q_{ji},  Q_{ii}=0\big\},\\
\QP&~\eqdefa\big\{Q\in\mathbb{Q}:~\text{the eigenvalues of }Q \in (-\frac13, \frac23) \big\}.
\end{align}
The space $\mathbb{Q}$ is endowed with the inner product
$ \langle Q_1,Q_2\rangle\eqdefa Q_1:Q_2=Q_{1ij}Q_{2ij}.$
The set $\mathbb{Q}$ is a five-dimensional linear subspace of $\mathbb{R}^{3\times3}.$
We define the matrix norm on $\mathbb{Q}$ as
$|Q|\eqdefa\sqrt{tr Q^2}=\sqrt{Q_{ij}Q_{ij}}$.
In terms of this norm, the Sobolev space is defined as
$$H^k(\mathbb{R}^3)\eqdefa
\Big\{f:
~ \int_{\mathbb{R}^3}\sum_{|\alpha'|\leq k}|\partial^{\alpha'}f(\xx)|^2
\ud\xx<\infty\Big\}$$
with $k$ being a non-negative integer and $\alpha'$ being a multi-index.
For two tensors $A,B\in\mathbb{Q}$ we denote $(A\cdot B)_{ij}=A_{ik}B_{kj}$
and $A:B=A_{ij}B_{ij}$. We denote $(M:Q)_{ij}=M_{ijkl}Q_{kl}$ where $M$ is the fourth-order tensor and $Q\in\mathbb{Q}$.
In addition, $\nn_1\otimes\nn_2\otimes\cdots\otimes\nn_k$ denotes the tensor product of $k$ vectors $\nn_1$, $\nn_2$,$\cdots$, $\nn_k$,
and we usually omit the symbol $\otimes$ for simplicity.
We use $f_{,i}$ to denote $\partial_if$ for simplicity and $\II$ to denote the $3\times3$ order identity tensor.

\subsection{The Ericksen-Leslie theory}
The hydrodynamic theory of liquid crystals, established by  Ericksen \cite{E-61} and Leslie \cite{Les} in the 1960's,
is a system coupling the time evolution equation of the
fluid velocity $\vv=\vv(t,\xx)$ with the director equation describing the motion of the director
field $\nn=\nn(t,\xx)\in \BS$. The general Ericksen-Leslie system takes the form
\begin{align}
&\vv_t+\vv\cdot\nabla\vv=-\nabla{p}+\nabla\cdot\sigma,\label{eq:EL-v}\\
&\na\cdot\vv=0,\\
&\nn\times\big(\hh-\gamma_1\NN-\gamma_2\DD\cdot\nn\big)=0,\label{eq:EL-n}
\end{align}
where $\vv$ is the velocity of the fluid and $p$ is the pressure.
The stress $\sigma$ is modeled by the phenomenological constitutive relation
\beno
\sigma=\sigma^L+\sigma^E,
\eeno
where $\sigma^L$ is the viscous (Leslie) stress
\begin{eqnarray}\label{eq:Leslie stress}
\sigma^L=\alpha_1(\nn\nn:\DD)\nn\nn+\alpha_2\nn\NN+\alpha_3\NN\nn+\alpha_4\DD
+\alpha_5\nn\nn\cdot\DD+\alpha_6\DD\cdot\nn\nn \end{eqnarray}
with $\kappa=(\nabla\vv)^T,~\DD=\frac{1}{2}(\kappa+\kappa^T)$ and
\beno
\NN=\nn_t+\vv\cdot\nabla\nn-\BOm\cdot\nn,\qquad\BOm=\frac12(\kappa^T-\kappa).
\eeno
The six constants $\al_1, \cdots, \al_6$ are saied to be the Leslie coefficients.
Moreover, $\sigma^E$ is the elastic (Ericksen) stress given by
\begin{eqnarray}\label{eq:Ericksen}
\sigma_{ij}^E=-\frac{\partial{E_F}}{\partial n_{k,j}}n_{k,i},
\end{eqnarray}
where $E_F=E_F(\nn,\nabla\nn)$ is the Oseen-Frank energy with the form
\begin{align}\label{energy-OF}
\!\!E_F=\f {k_1} 2(\na\cdot\nn)^2+\f {k_2} 2(\nn{\cdot}(\na\times\nn))^2
+\f {k_3} 2|\nn{\times}(\na\times \nn)|^2
+\frac{k_2+k_4}2\big(\textrm{tr}(\na\nn)^2-(\na\cdot\nn)^2\big).
\end{align}
Here $k_1, k_2, k_3, k_4$ are the elastic constant.
The molecular field $\hh$ is given by
\beno
&&\hh=-\frac{\delta{E_F}}{\delta{\nn}}=
\nabla\cdot\frac{\partial{E_F}}{\partial(\nabla\nn)}-\frac{\partial{E_F}}{\partial\nn}.
\eeno
The Leslie coefficients and material dependent coefficients $\gamma_1, \gamma_2$ satisfy the following relations
\begin{eqnarray}
&\alpha_2+\alpha_3=\alpha_6-\alpha_5,\label{Leslie relation}\\
&\gamma_1=\alpha_3-\alpha_2,\quad \gamma_2=\alpha_6-\alpha_5,\label{Leslie-coeff}
\end{eqnarray}
where (\ref{Leslie relation}) is called Parodi's relation derived from the Onsager reciprocal relation \cite{Parodi}. These two relations will
ensure that the system (\ref{eq:EL-v})--(\ref{eq:EL-n}) has a basic energy law:
\begin{align}
-\frac{\ud}{\ud{t}}\Big(\int_{\BR}\frac{1}{2}|\vv|^2\ud\xx+E_F\Big)
=&\int_{\BR}\Big((\alpha_1+\frac{\gamma_2^2}{\gamma_1})(\DD:\nn\nn)^2
+\alpha_4|\DD|^2\qquad\nonumber\\
&\quad+\big(\alpha_5+\alpha_6-\frac{\gamma_2^2}{\gamma_1}\big)|\DD\cdot\nn|^2
+\frac{1}{\gamma_1}|\nn\times\hh|^2\Big)\ud\xx.\quad\label{EL_energy_law}
\end{align}

For the well-posedness results of the Ericksen-Leslie system, we refer to \cite{LL, WZZ2, WXL} and the references therein. In particular,
under a natural physical condition on the Leslie coefficients, \cite{WZZ2}
proved the well-posedness of the system, and the global existence of weak solution in two-dimensional case
was shown in \cite{HLW, WW}.

\subsection{The $Q$-tensor theory}
The most general continuum theory for the nematic liquid crystals is the celebrated Landau-de Gennes theory which can describe
uniaxial and biaxial liquid phases.
In this phenomenological theory, the detailed nature of molecular interactions and molecular structures is ignored, and
the state of the nematic liquid crystals is described by a macroscopic tensor value order parameter $Q(\xx)$,
which is a symmetric and traceless $3\times 3$ matrix, i.e. $Q\in\mathbb{Q}$. Physically, it can be interpreted as the second-order traceless moment of
the orientational distribution function $f$, that is,
\begin{align}\label{relation:Q-f}
Q(\xx,\mm)=\int_\BS(\mm\mm-\frac{1}{3}\II)f(\xx,\mm)\ud\mm.
\end{align}
Under this interpretation, the so-called physical constraint is that the eigenvalues of $Q$ should satisfy
\begin{align}\label{phy-constraint}
\lambda_i(Q)\in(-\frac13,\frac23), ~\text{for}~ 1\le i\le 3,
\end{align}
namely, $Q\in\QP$.

The nematic liquid crystal is called isotropic at $\xx$ when $Q(\xx)=0$. When $Q(\xx)$ has two equal non-zero eigenvalues,
it is called uniaxial and $Q(\xx)$ can be written as
\beno
Q(\xx)=s\big(\nn\nn-\f13\II\big),\quad s\in\mathbb{R},~\nn(\xx)\in \BS.
\eeno
When $Q(\xx)$ has three distinct eigenvalues, it is called biaxial and $Q(\xx)$ can be written as
\beno
Q(\xx)=s\big(\nn\nn-\f13\II\big)+r(\nn'\nn'-\frac13\II),\quad \nn,\,\nn'\in \BS,\quad \nn\cdot\nn'=0,\quad s,~r\in\mathbb{R}.
\eeno

The classic Landau-de Gennes energy functional, being a nonlinear functional of $Q$ and its spatial derivatives,
takes the following general form
\begin{align}
\mathcal{F}_{LG}(Q,\nabla Q)=&\int_{\BR}\Big\{ \underbrace{-\frac{a}2tr(Q^2)
-\frac{b}{3}tr(Q^3)+\frac{c}{4}(tr(Q^2)^2}_{\text{bulk energy}}\nonumber\\
&+\underbrace{\frac{1}{2}\Big(L_1|\nabla Q|^2+L_2Q_{ij,j}Q_{ik,k}
+L_3Q_{ij,k}Q_{ik,j}+L_4Q_{ij}Q_{kl,i}Q_{kl,j}\Big)}_{\text{elastic energy}} \Big\}\ud\xx,\label{eq:Landau-energy}
\end{align}
where $a, b, c$ are material-dependent and temperature-dependent non-negative constants and $L_i(i=1,2,3,4)$
are material dependent elastic constants. We refer to \cite{DG, MN} for more details.
The energy (\ref{eq:Landau-energy}) can not ensure $Q$ to satisfy the
natural physical constraint (\ref{phy-constraint}).
For this reason, based on the mean-field Maier-Saupe energy, Ball-Majumdar \cite{BM} proposed an energy functional,
which will diverge if $Q\not\in\QP.$ There are many works to study the equilibrium solutions of the classic
Landau-de Gennes model, for example,  one may see \cite{BM, MZ} and the references therein.

So far, there are two types of dynamic $Q$-tensor theories to describe the flow of nematic liquid crystal.
The first type models are obtained by variational methods under physical considerations, such as Beris-Edwards model \cite{BE}
and Qian-Sheng's  model \cite{QS}. Let $\mathcal{F}(Q,\nabla Q)$ be the total free energy, and define
$$\mu_{Q}=\frac{\delta{\mathcal{F}(Q,\nabla Q)}}{\delta Q}.$$
The dynamical $Q$-tensor model of this types can be written in
the following general form:
\begin{align}
\frac{\pa{{Q}}}{\pa{t}}+\vv\cdot\nabla{Q}&=D^{rot}(\mu_Q)+F(Q,\DD)
+\BOm\cdot Q-Q\cdot\BOm,\label{eq:Q-general-intro1}\\
\frac{\pa{\vv}}{\pa{t}}+\vv\cdot\nabla\vv&=-\nabla{p}+\nabla\cdot\big(\sigma^{dis}+\sigma^{s}+\sigma^a+\sigma^{d}\big),
\label{eq:Q-general-intro2}\\
\nabla\cdot\vv&=0,\label{eq:Q-general-intro3}
\end{align}
where $\vv$ is the fluid velocity, $D^{rot}(\mu_Q)$ is the rotational diffusion term,
$F(Q,\DD)$ and $\BOm\cdot Q-Q\cdot\BOm$ are induced by the deformation part and
and rotation part of the velocity gradient respectively.
In addition, $\sigma^d$ is the distortion stress,
$\sigma^a$ is the anti-symmetric part of orientational-induced stress,
$\sigma^{s}=\gamma F(Q, \mu_Q)$ which conjugates to $F(Q, \DD)$ ($\gamma$ is a constant), is the symmetric stress induced by the orientation of molecules,  and $\sigma^{dis}$ is an additional dissipation stress.

In  Beris-Edwards's model and Qian-Sheng's model, module some constants, $\sigma^a$ and $\sigma^d$ are the same, i.e.,
\begin{align}
\sigma_{ij}^d=\frac{\partial \CF}{\partial (Q_{kl,j})}Q_{kl,i},\quad
\sigma^a=Q\cdot\mu_Q-\mu_Q\cdot Q.
\end{align}
In Beris-Edwards's model, the other terms are given by
\begin{align*}
&D^{rot}_{BE}=-\Gamma \mu_Q,\quad \sigma_{BE}^{dis}=\eta\DD,\quad \sigma_{BE}^s=F_{BE}(Q,\mu_Q),\\
&F_{BE}(Q,A)=\xi\Big((Q+\frac13\II)\cdot A+A\cdot(Q+\frac13\II)-2(Q+\frac13\II)(A:Q)\Big).
\end{align*}
In Qian-Sheng's model, they are given by
\begin{align*}
&D^{rot}_{QS}=-\Gamma \mu_Q,\quad
\sigma_{QS}^s=-\frac12\frac{\mu_2^2}{\mu_1}\mu_{Q},\quad
F_{QS}(Q,\DD)=-\frac12\frac{\mu_2}{\mu_1}\DD,\\
&\sigma^{dis}_{QS}=\beta'_1 Q(Q:A)+\beta'_2 \DD+ \beta'_3(Q\cdot\DD+\DD\cdot Q).
\end{align*}
When taking $\mathcal{F}(Q,\nabla Q)=\mathcal{F}_{LG}(Q,\nabla Q)$, for the well-posedness results of the Beris-Edwards's model on whole space and bounded domain, we refer to \cite{PZ1,PZ2,HD} and \cite{ADL1, ADL2}.

The second type is derived from the molecular kinetic theory by closure approximations.
In such models, the evolution of $Q$ is derived from the evolution of probability density function $f$ by relation (\ref{relation:Q-f}). However, one have to approximate the higher order moment such as
\begin{align}
\int_\BS\mm\mm\mm\mm f(\xx,\mm)\ud\mm
\end{align}
by using $Q$. This process is called closure approximation. There are various
kinds of closure approximation and then they lead to different models in $Q$-tensor form, which are summarized in \cite{Feng, FLS}. However, these models do not obey energy dissipation law. In \cite{WZZ3}, based on Doi's kinetic theory, the authors proposed a $Q$-tensor model with energy dissipation law by using the Bingham closure. In this paper, we are mainly concerned this model.
Before introducing it, we first give a brief description of the Bingham closure.

For a given configuration distribution function $f(\mm)$ satisfying
$$\int_{S^2}f(\mm)\ud\mm=1,~\int_{S^2}(\mm\mm-\frac13\II)f(\mm)\ud\mm=Q,$$
the Bingham closure is to use the quasi-equilibrium distribution (also called the Bingham distribution)
$$f_Q=\frac{1}{Z_Q}\exp(B_Q:\mm\mm),\qquad Z_Q=\int_{S^2}\exp(B_Q:\mm\mm)\ud\mm,$$
to approximate $f$. Here, $B_Q\in\mathbb{Q}$ depends on $Q$ and is determined by the following relation
$$\int_{S^2}(\mm\mm-\frac13\II)f_Q\ud\mm=Q.$$
By Proposition \ref{bingham_exi-unqi}, $B_Q$ can be uniquely determined for $Q\in\QP$. Then, the fourth-order moment and the sixth-order moment of $f$ are approximated by
$$M^{(4)}_Q=\int_{S^2}\mm\mm\mm\mm f_Q\ud\mm,\quad\text{and}\quad M^{(6)}_Q=\int_{S^2}\mm\mm\mm\mm\mm\mm f_Q\ud\mm.$$

Now we introduce the dynamic $Q$-tensor model presented in \cite{WZZ3}.
For given free energy functional $\mathcal{F}(Q,\nabla Q)$, define
\begin{align*}
 \mu_Q=\frac{\delta \CF(Q,\nabla Q)}{\delta Q}.
\end{align*}
We introduce the following two operators
\begin{align*}
 \mathcal{M}_Q(A)=&\frac13 A+Q\cdot A-A:M^{(4)}_Q,\\
 \mathcal{N}_Q(A)_{\alpha\beta}=&\partial_i\Big\{[\gamma_{\perp}(M^{(4)}_{\alpha\beta kl}\delta_{ij}
 -\frac13\delta_{\alpha\beta}Q_{kl}\delta_{ij})+(\gamma_{\parallel}-\gamma_{\perp})(
 M^{(6)}_{\alpha\beta klij}-\frac13 \delta_{\alpha\beta}M^{(4)}_{klij})]\partial A_{kl}\Big\}.
\end{align*}
Based on the Doi-Onsager's molecular theory, making use of the aforementioned  Bingham closure approximation, the new $Q$-tensor model is given as following \cite{WZZ3}:
\begin{align}
\frac{\pa{{Q}}}{\pa{t}}+\vv\cdot\nabla{Q}&=\frac{\ve}{De}\mathcal{N}_Q(\mu_Q)
  -\frac{2}{De}\Big(\CM_Q(\mu_Q)+\CM_Q^T(\mu_Q)\Big)+\CM_Q(\nabla\vv)+\CM_Q^T(\nabla\vv),\label{eq:Q-D}\\\nonumber
\frac{\pa{\vv}}{\pa{t}}+\vv\cdot\nabla\vv&=-\nabla{p}
+\frac{\gamma}{Re}\Delta\vv+\frac{1-\gamma}{2Re}\nabla\cdot(\DD:M_Q^{(4)})\\
&\quad\quad\quad+\frac{1-\gamma}{DeRe}\Big(2\nabla\cdot\CM_Q(\mu_Q)+\mu_Q:\nabla Q\Big),\label{eq:v-D}\\
\nabla\cdot \vv&=~0, \label{eq:v-0}
\end{align}
where $De$ and $Re$ are called Deborah number and Reynolds number respectively,
and $\gamma\in(0,1)$ is a constant.
The small parameter $\sqrt{\ve}$ characterizes the typical interaction distance, which is usually at the scale of molecule length.
The term $\mathcal{N}_Q(\mu_Q)$ represents the translational diffusion. An important feature of this model is that  (\ref{eq:Q-D})-(\ref{eq:v-0}) obeys the following basic energy dissipative law (see \cite{WZZ3})
\begin{align}\label{newQenergdissip}
 &\frac{d}{dt}\Big(\frac12\int_{\mathbb{R}^3}|\vv|^2\ud\xx+\frac{1-\gamma}{ReDe}\mathcal{F}(Q,\nabla Q)\Big)=
 -\int_{\mathbb{R}^3}\Big(\frac{\gamma}{Re}|\nabla \vv|^2+\frac{1-\gamma}{2Re}\DD:M^{(4)}_Q:\DD\nonumber\\
 &\qquad\qquad-\frac{\ve(1-\gamma)}{ReDe^2}\mu_Q:\mathcal{N}(\mu_Q)
 +\frac{4(1-\gamma)}{ReDe^2}\mu_Q:\mathcal{M}_{Q}(\mu_Q)\Big)\ud\xx.
\end{align}

In \cite{WZZ3}, the energy functional is also derived from Onsager's molecular theory.
\begin{align*}
 \mathcal{F}(Q,\nabla Q)&=\mathcal{F}_b(Q)+\mathcal{F}_e(Q,\nabla Q),
\end{align*}
where the bulk energy $\mathcal{F}_b(Q)$ and the elastic distortion energy $\mathcal{F}_e$ are respectively given by
\begin{align*}
\mathcal{F}_b(Q)=&~L_0\int \Big(-\ln{Z}_Q+Q:B_Q-
\frac{\alpha}{2}|Q|^2\Big)\ud\xx,\\
\mathcal{F}_e(Q,\nabla Q)=&~\frac\ve2\int\bigg\{
L_1|\nabla Q|^2+L_2\Big(\partial_i(Q_{ik})\partial_j(Q_{jk})+\partial_i(Q_{jk})\partial_j(Q_{ik})\Big)+L_3|\nabla Q^{(4)}|^2\nonumber\\
&+L_4\Big(\partial_i(Q^{(4)}_{iklm})\partial_j
(Q^{(4)}_{jklm})+\partial_i(Q^{(4)}_{jklm})\partial_j(Q^{(4)}_{iklm})\Big)+L_5\partial_i(Q^{(4)}_{ijkl})
\partial_j(Q_{kl})\bigg\}\ud\xx,\nonumber
\end{align*}
where $Q^{(4)}=Q^{(4)}(Q)$ is the fourth order symmetric traceless moment of the Bingham distribution $f_Q$. Namely,
\begin{align}
Q_{ijkl}^{(4)}:=&\int_{\BS}\Big\{{m}_{i}{m}_{j}{m}_{k}{m}_{l}
-\frac{1}{7}\Big(m_{i}m_{j}\delta_{kl}+
m_{k}m_{l}\delta_{ij}+m_{i}m_{k}\delta_{jl}
+m_{j}m_{l}\delta_{ik}\nonumber\\
&\qquad+m_{i}m_{l}\delta_{jk}+
m_{j}m_{k}\delta_{il}\Big)+\frac{1}{35}\Big(\delta_{ij}\delta_{kl}
+\delta_{ik}\delta_{jl}+\delta_{il}\delta_{jk}\Big)\Big\}f_Q\ud\mm.
\end{align}
The difference between $Q^{(4)}$ and $M_Q^{(4)}$ is that $Q^{(4)}$ is traceless, i.e. $Q_{ijkk}^{(4)}=0$, while $M_Q^{(4)}$ is not.
The bulk energy $\mathcal{F}_b$ is equivalent to the penalized energy derived by Ball-Majumdar in \cite{BM}.
Thus, the order parameter tensor $Q$ should satisfy the physical constraint (\ref{phy-constraint}).

The parameters appearing in the system (\ref{eq:Q-D})-(\ref{eq:v-0})
have clear physical significance but not are phenomenological. In \cite{WZZ3}, the coefficients
$L_i(i=0,1,\cdots,5)$ are also explicitly calculated in terms of physical molecular parameters.
The parameter $\ve$ appears in the elastic energy $\CF_e$ due to the fact that the ratios
between the coefficients of $\CF_e$ and the ones in $\CF_b$  are at the order of square of molecule length.
Another important feature of the molecule-based $Q$-tensor system
(\ref{eq:Q-D})-(\ref{eq:v-0}) is that the translational and rotational diffusions are still maintained.

\subsection{Motivations and main results}
The connection between different level of liquid crystal theories is a problem of both physical and mathematical importance.
Based on a formal asymptotical expansion, Kuzzu-Doi \cite{KD} and E-Zhang \cite{EZ} derived the Ericksen-Leslie
equation from the Doi-Onsager equations by taking small Deborah number limit for spacial homogeneous case and inhomogeneous case respectively.
Wang-Zhang-Zhang rigorously justified this limit in \cite{WZZ1} before the first singularity time of
the Ericksen-Leslie system. In \cite{WZZ4}, they also presented a rigorous derivation from Beris-Edwards model to Ericksen-Leslie model.
In \cite{WZZ3}, it is proposed a systematic study on the modeling for liquid crystals in both static and dynamic cases.
They derived a $Q$-tensor model from Onsager's molecular theory and Doi's kinetic theory, which is introduced in the previous subsection,
and also derived Oseen-Frank model and Ericksen-Leslie model.

The main aim of this paper is to prove the local well-posedness for strong solution of the molecule-based $Q$-tensor model,
and also to show that the strong solution will converges to the solution of Ericksen-Leslie system under the limit of Deborah number $De\to0$.

In this paper, to avoid some tedious technical difficulties, we will only consider the case
when the translational diffusion $\mathcal{N}_Q(\mu_Q)=0$ and the coefficients $L_0=1,L_3=L_4=L_5=0$. Then
\begin{align}
\mu_{Q}&~=\frac{\delta\CF_b(Q)}{\delta{Q}}+\frac{\delta\CF_e(\nabla Q)}{\delta{Q}},\\
\frac{\delta\CF_b(Q)}{\delta{Q}}&~=B_Q-\alpha Q,\\
\Big(\frac{\delta\CF_e(Q)}{\delta{Q}}\Big)_{ij}&~= -\ve\Big(L_1\Delta Q_{ij}+ L_2(Q_{ik,jk}+Q_{jk,ik})\Big)=:\ve\mathcal{L}(Q).
\end{align}
Then, the corresponding molecule-based $Q$-tensor system becomes :
\begin{align}
 \frac{\partial Q}{\partial t}+\vv\cdot\nabla Q=&-\frac{2}{De}\Big(\mathcal{M}_Q(B_Q-\alpha Q+\ve\CL(Q))
 +\mathcal{M}^{T}_Q(B_Q-\alpha Q+\ve\CL(Q))\Big)\nonumber\\
  &+\mathcal{M}_Q(\nabla\vv)+\mathcal{M}^T_Q(\nabla\vv),\label{eqnewQ1}\\
\frac{\partial \vv}{\partial t}+\vv\cdot\nabla \vv=&-\nabla p +\frac{\gamma}{Re}\Delta\vv+\frac{1-\gamma}{2Re}
       \nabla\cdot(\DD:M^{(4)}_Q)\nonumber\\
       &+\frac{1-\gamma}{DeRe}\nabla\cdot\Big(2\mathcal{M}_Q\big(B_Q-\alpha Q+\ve\CL(Q)\big)+\ve\sigma^d(Q,Q)\Big),\label{eqnewQ2}\\
 \nabla\cdot \vv=&~0, \label{eqnewQ3}
\end{align}
where $\sigma^d$ is defined by
\beno
\sigma^d_{ji}(Q,\widetilde Q)\eqdefa-\frac{\partial \mathcal{F}_e}{\partial Q_{kl,j}}\widetilde Q_{kl,i}
=-\big(L_1Q_{kl,j}\widetilde Q_{kl,i}+L_2Q_{km,m}\widetilde Q_{kj,i}+L_2Q_{kj,l}\widetilde Q_{kl,i}\big).
\eeno
It not hard to see that $\nabla\cdot\sigma^d(Q,Q)$ differs from $\mu_Q:\nabla Q$ with only pressure terms.

When $\alpha>\alpha^*$, the bulk energy function $\CF_b$ has stable uniaxial
critical points $Q=S_2(\nn\nn-\frac13\II)$ for any $\nn\in\BS$, which correspond
to nematic phase. Here, $S_2=S_2(\alpha)$ is a increasing function of $\alpha$ for $\alpha>\alpha^*$,
see the precise definition in (\ref{eqn:S}). Throughout this paper, we always assume $\alpha>\alpha^*$
and $L_1>0, L_1+2L_2>0$. Thus, it is known from Lemma 2.2 in \cite{WZZ4} that
\begin{align*}
\int\CL(Q):Q \ud\xx\ge c_0\int |\nabla Q|^2\ud\xx,
\end{align*}
for some constant $c_0>0$.

We first state the following the local well-posedness result.
\begin{theorem}\label{thm:main1}
Let $s\geq 2$ be an integer. $\nn^{\ast}\in\BS$ is a constant vector and $Q^{\ast}=S_2(\nn^{\ast}\nn^{\ast}-\frac13\II).$
If the initial data satisfies
\begin{align}
\vv_I(\xx)\in H^{s}(\mathbb{R}^3),\quad  Q_I(\xx)-Q^{\ast}\in H^{s+1}(\mathbb{R}^3),
\end{align}
with
\begin{align}
Q_I(\xx)\in\QPd:=\Big\{Q\in\mathbb{Q}: \text{all the eigenvalues of $Q$ belong to }[-\frac13+\delta, \frac23-\delta]\Big\},
\end{align}
for all $\xx\in\BR$, then there exists $T>0$ and a unique solution $(\vv,Q)$ of the $Q$-tensor system (\ref{eqnewQ1})-(\ref{eqnewQ3}) on $[0,T]$,
such that $\vv(0,\xx)=\vv_I(\xx),Q(0,\xx)=Q_I(\xx)$, and
\begin{align}
 \vv(t,\xx)\in&~ C([0,T];H^{s}(\mathbb{R}^3))\cap L^2(0,T;H^{s+1}(\mathbb{R}^3)),\\
 Q(t,\xx)-Q^{\ast}\in&~ C([0,T];H^{s+1}(\mathbb{R}^3)),
\end{align}
and $Q(t,\xx)\in \mathbb{Q}_{phy,\delta/2}$.
\end{theorem}

Next, we consider the small Deborah number limit $De\to 0$. To obtain the full Ericksen-Leslie system,
we have to take $De=O(\ve)$ as in \cite{WZZ3}. For simplicity, we choose $De=\ve$.
Then the system can be written as:
\begin{align}
 \frac{\partial Q^{\ve}}{\partial t}+\vv^{\ve}\cdot\nabla Q^{\ve}
 =&-\frac{2}{\varepsilon}\Big(\mathcal{M}_{Q^{\ve}}(B_{Q^{\ve}}-\alpha Q^{\ve}+\ve\CL(Q^{\ve}))
 +\mathcal{M}^{T}_{Q^{\ve}}(B_{Q^{\ve}}-\alpha Q^{\ve}+\ve\CL(Q^{\ve}))\Big)\nonumber\\
  &+\mathcal{M}_{Q^{\ve}}(\nabla\vv^{\ve})+\mathcal{M}^T_{Q^{\ve}}(\nabla\vv^{\ve}),\label{eqepsiQ1}\\
\frac{\partial \vv^{\ve}}{\partial t}+\vv^{\ve}\cdot\nabla \vv^{\ve}
=&-\nabla p^\ve +\frac{\gamma}{Re}\Delta\vv^{\ve}+\frac{1-\gamma}{2Re}
       \nabla\cdot(\DD:M^{(4)}_{Q^{\ve}})\nonumber\\
       &+\frac{1-\gamma}{\varepsilon Re}\nabla\cdot\Big(2\mathcal{M}_{Q^{\ve}}\big(B_{Q^{\ve}}-\alpha Q^{\ve}+\ve\CL(Q^{\ve})\big)+\ve\sigma^d(Q^{\ve},Q^{\ve})\Big),\label{eqepsiQ2}\\
 \nabla\cdot \vv^{\ve}=&~0. \label{eqepsiQ3}
\end{align}
We define the coefficient in Ericksen-Leslie theory as:
\begin{align}\label{leslie1-intro}
&\alpha_1=-\frac{S_4}{2},\qquad
\alpha_2=-\frac{S_2}{2}(1+\frac{1}{\zeta}),\qquad
\alpha_3=-\frac{S_2}{2}(1-\frac{1}{\zeta}),\nonumber\\
&\alpha_4=\frac{4}{15}-\frac{5}{21}S_2-\frac{1}{35}S_4,\quad
\alpha_5=\frac{1}{7}S_4+\frac{6}{7}S_2,\quad
\alpha_6=\frac{1}{7}S_4-\frac{1}{7}S_2,
\end{align}
and
\begin{align}\label{leslie2-intro}
\gamma_1=\frac{1}{\frac1{3S_2}+\frac2{3S_2^2}-\frac2{S_2^2\alpha}},\quad\gamma_2=-S_2,
\quad\zeta\eqdefa-\frac{\gamma_2}{\gamma_1}=\frac13+\frac2{3S_2}-\frac{2}{S_2\alpha},
\end{align}
and the elastic constants in Oseen-Frank energy are given by
\begin{align}\label{OF-LD-relation-intro}
k_1=k_3=2(L_1+L_2)S_2^2,\quad k_2=2L_1S_2^2,\quad k_4=L_2S_2^2.
\end{align}
Here $S_4=S_4(\alpha)$ is also a constant related to $\alpha$, see the definition in (\ref{eqn:S}).

For a given direction field $\nn(t,\xx)$, we define
\begin{align}
\CP^{out}(\QQ)=&~\QQ-(\nn\nn\cdot\QQ+\QQ\cdot\nn\nn)-2(\QQ:\nn\nn)\nn\nn, \label{def:Pout-intro}\\
\mathcal{H}_{\nn}(Q)=&~\psi_1(\nn\nn-\frac13\II)(\nn\nn:Q)
+\psi_2\big(-Q+\nn\nn\cdot Q+Q\cdot\nn\nn-\frac23\II(\nn\nn:Q)\big),\label{def:Hn-intro}
\end{align}
where the $\psi_1$ and $\psi_2$ are constants depending on $\alpha$.
$\CH_\nn(Q)$ is the linearized operator of $B_Q-\alpha Q$ around the local critical point $S_2(\nn\nn-\frac13\II)$.
The detailed motivation of the above definitions will be explained in Section 4.

The second main result of this paper is stated as follows.
\begin{theorem}\label{thm:main2}
Let $(\nn(t,\xx), \vv(t,\xx))$ be a solution of the Ericksen-Leslie system (\ref{eq:EL-v})--(\ref{eq:EL-n}) on $[0,T]$
with the coefficients given by (\ref{leslie1-intro})-(\ref{OF-LD-relation-intro}), which satisfies
\beno
\vv\in C([0,T];H^{k}), \quad \nabla\nn\in C([0,T];H^{k})\quad \textrm{for}\quad k\ge 20.
\eeno
Let $Q_0(t,x)=S_2\big(\nn(t,\xx)\nn(t,\xx)-\II\big)$ and
the functions $\big(Q_1,Q_2,Q_3, \vv_1,\vv_2\big)$ are determined by Proposition \ref{prop:Hilbert}.
Assume that the initial data $(Q^{\ve}_I, \vv^\ve_I)$ takes the form
\begin{align*}
Q_I^\ve(\xx)=\sum^3_{k=0}\ve^kQ_{3}(0,\xx)+\ve^3Q_{I,R}^\ve(\xx),~~~
\vv_I^\ve(\xx)=\sum^3_{k=0}\ve^k\vv_{k}(0,\xx)+\ve^3\vv_{I,R}^\ve(\xx),
\end{align*}
where  $(Q_{I,R}^\ve, \vv_{I,R}^\ve)$ satisfies
\begin{align}
\|\vv_{I,R}^\ve\|_{H^2}+\|Q_{I,R}^\ve\|_{H^3}+\ve^{-1}\|\CP^{out}(Q^\ve_{I,R})\|_{L^2}\le E_0.\non
\end{align}
Then there exists $\ve_0>0$ and $E_1>0$ such that
for all $\ve<\ve_0$, the system (\ref{eqepsiQ1})--(\ref{eqepsiQ3}) has a unique solution
$(Q^\ve(t,\xx), \vv^\ve(t,\xx))$ on $[0,T]$ which has the expansion
\begin{align*}
Q^\ve(t,\xx)=\sum^3_{k=0}\ve^kQ_k(t,\xx)+\ve^3Q_R(t,\xx),~~~
\vv^\ve(t,\xx)=\sum^3_{k=0}\ve^k\vv_k(t,\xx)+\ve^3\vv_R(t,\xx),
\end{align*}
where $(Q_R,\vv_R)$ satisfies
\begin{align*}
\Ef(Q_R(t),\vv_R(t))\le E_1.
\end{align*}
Here $\Ef(Q,\vv)$ is defined by
\begin{align*}\nonumber
\Ef(Q,\vv)\eqdefa&~\frac12\int\Big(|\vv|^2+\CJ^{-1}_{\nn}(Q):Q+\frac{1-\gamma}{\ve Re}\CH^{\ve}_{\nn}(Q):Q\Big)
 +\ve^2\Big(|\nabla \vv|^2\\
 &+\frac{1-\gamma}{\ve Re}\CH^{\ve}_{\nn}(\nabla Q):\nabla Q\Big)
 +\ve^4\Big(|\Delta \vv|^2+\frac{1-\gamma}{\ve Re}\CH^{\ve}_{\nn}(\Delta Q):\Delta Q\Big)d\xx,\label{eq:energy functional}
\end{align*}
and $\CH^\ve_\nn(Q)=\CH_{\nn}(Q)+\ve\CL(Q)$.
\end{theorem}
\begin{remark}
It can be observed from \cite{EZ} that the Leslie coefficients of Ericksen-Leslie system derived from the Doi-Onsager system
have same forms as (\ref{leslie1-intro})-(\ref{leslie2-intro}) except for $\gamma_1$.
The only difference is due to the Bingham closure approximation.
\end{remark}

The remaining sections of this paper are organized as follows. In Section 2,
the important properties of the Bingham closure and the critical point are presented.
Section 3 is devoted to the proof for the existence of the local strong solution of the molecule-based $Q$-tensor system.
In Section 4, we present some important linearized operators which will be used in deriving the Ericksen-Leslie system
from the molecule-based $Q$-tensor system. In Section 5, by using the Hilbert expansion method, we present a rigorous derivation
from the molecule-based $Q$-tensor theory to the Ericksen-Leslie theory.

\section{The Bingham closure and the critical points}
This section is mainly concerned to the important properties of the Bingham closure and the critical points.

\subsection{The Bingham closure and Bingham map}
The Bingham closure plays an important role in the system (\ref{eq:Q-general-intro1})-(\ref{eq:Q-general-intro3}).
For this, one should find $B_Q\in\mathbb{Q}$ such that
\begin{equation}\label{relationQB00}
\int_{\BS}(\mm\mm-\frac13\II)\frac{\exp(B_Q:\mm\mm)}{\int_{\BS}\exp(B_Q:\mm'\mm')d\mm'}d\mm=Q,
\end{equation}
for a given $Q\in\QP$.
The following proposition tells us that $B_Q$ can be uniquely defined for any $Q\in\QP$.
 We call this map from $Q\in\QP$ to $B_Q\in\mathbb{Q}$ Bingham map.
\begin{proposition}[Existence and uniqueness of $B_Q$]\label{bingham_exi-unqi}
For a given $Q\in \QP$, there exists a unique $B_Q\in \mathbb{Q}$
such that (\ref{relationQB00}) holds.
\end{proposition}
\begin{proof}
A sketched proof is given in \cite{BM}. Here we give a detailed proof for completeness.

Define $\omega: \mathbb{Q}\to \mathbb{R}$ as:
\begin{align}
\omega(B)=\ln \int_\BS \mathrm{e}^{\mm\mm:B}\ud\mm.
\end{align}
Obviously, $\omega(B)$ depends only on its eigenvalues. From the fact that
\begin{align}
\int_\BS \mathrm{e}^{\mm\mm:B_1}\ud\mm\int_\BS \mathrm{e}^{\mm\mm:B_2}\ud\mm\ge \Big(\int_\BS \mathrm{e}^{\mm\mm:(B_1+B_2)/2}\ud\mm\Big)^2,
\end{align}
we know $\omega(B)$ is convex. Then we can define its {\it convex conjugate} by Legendre transformation:
$\omega^*(Q): X\to\mathbb{R}$ as
\begin{align}
\omega^*(Q)=\sup_{B\in\mathbb{Q}}\big(B:Q-\omega(B)\big)
\end{align}
with domain $X$ defined by
\begin{align}
X=\big\{Q: \sup_{B\in\mathbb{Q}}\big(B:Q-\omega(B)\big) <+\infty\big\}.
\end{align}
We will prove that $X=\QP$. For this, we need an elementary inequality:

{\it Claim:} Let $b_1\le b_2\le b_3$ and $q_1\le q_2\le q_3$ are the eigenvalues of $B$ and $Q$ respectively, then
$B:Q\le b_1q_1+b_2q_2+b_3q_3$.

To prove it, we can assume $B$ is diagonal without loss of generality. Suppose
$Q=q_1\nn_1\otimes\nn_1+q_2\nn_2\otimes\nn_2+q_3\nn_3\otimes\nn_3$ with $\nn_i\cdot\nn_j=\delta^i_j.$ Then
$B:Q=\sum_{i,j=1,2,3}b_iq_j n_{ji}^2$, where $\nn_i=(n_{i1}, n_{i2}, n_{i3})^T$. A direct computation shows that
\begin{align*}
b_1q_1+b_2q_2+b_3q_3-B:Q=&~(q_1-q_2)(b_2-b_3)n_{13}^2+(q_1-q_2)(b_1-b_2)(1-n_{11}^2)\\
&~+(q_2-q_3)(b_1-b_2)n_{31}^2+(q_2-q_3)(b_2-b_3)(1-n_{33}^2)\\
\ge &~0,
\end{align*}
which yields our claim.

For $Q\in\QP$ with eigenvalues $-\frac13< q_1\le q_2\le q_3<2/3$, and $B\in\mathbb{Q}$ with eigenvalues $\{b_1,b_2,b_3\}$, we can assume that $b_1\le b_2\le 0\le b_3$ or $b_1\le 0\le b_2\le b_3$.
Consider
$$A=\Big\{\mm: m_1^2-\frac13-q_1 <0, ~m_2^2-\frac13-q_2<0, ~m_3^2-\frac13-q_3>0 \Big\},$$
or for the later case
$$A=\Big\{\mm: m_1^2-\frac13-q_1 <0, ~m_2^2-\frac13-q_2>0, ~m_3^2-\frac13-q_3>0 \Big\}.$$
We know that the measure of $A$ is positive in each case.
Therefore,
\begin{align}
\exp&(\omega(B)-B:Q)\ge \exp(\omega(B)-q_1b_1-q_2b_2-q_3b_3)\nonumber\\
=&~\int_{\BS}\exp(b_1m_1^2+b_2m_2^2+b_3m_3^2-q_1b_1-q_2b_2-q_3b_3)\ud\mm\nonumber\\
=&~\int_{\BS}\exp\big(b_1(m_1^2-\frac13-q_1)+b_2(m_2^2-\frac13-q_2)+b_3(m_3^2-\frac13-q_3)\big)\ud\mm\nonumber\\
\ge & ~\int_A 1\ud\mm = \text{meas}(A).\nonumber
\end{align}
This implies that $B:Q-\omega(B)\le -\ln(\text{meas}(A))$ is bounded. Hence $Q\in X$, i.e. $\QP \subseteq X$.

On the other hand, if $q_1\le -\frac13$, then we take $b_1=2b\to -\infty, b_2=b_3=-b$, then
\begin{align}
\exp(B:Q-\omega(B))=&~\int_{\BS}\exp\big(b(-m_1^2+\frac13+q_1)\big)\ud\mm\nonumber\\
\ge &~ 4\pi e^{b(\frac13+q_1)} \to +\infty.
\end{align}
If $q_3>\frac23$, taking $b_3=2b\to+\infty, b_1=b_2=-b$, then we can also
obtain that $\exp(B:Q-\omega(B))$ is unbounded, which implies $X \subseteq\QP$. Therefore, $X =\QP$.

Therefore, for any $Q\in\QP$, there exists $B\in\mathbb{Q}$ such that
$$B:Q-\omega(B)=\sup_{B_1\in\mathbb{Q}}\big(B_1:Q-\omega(B_1)\big).$$
Thus
\begin{align}\label{relation:QB}
Q=(\nabla_B\omega)(B)=\frac{\int_{S^2}(\mm\mm-\frac13\II)\exp(B:\mm\mm)\ud\mm}
{\int_{S^2}\exp(B:\mm\mm)\ud\mm}.
\end{align}
We let $B_Q=B$, then the existence of $B_Q$ is proved. Since $\omega(B)$ is convex, we can deduce that
$(\nabla_{B}\omega)(B_1)\neq (\nabla_{B}\omega)(B_2)$ for $B_1\neq B_2$, which implies the uniqueness.
\end{proof}
The map from $\QP$ to $\mathbb{Q}$ which satisfies (\ref{relation:QB}) is a diffeomorphism, and so is its inverse.
We denote them by $B=B(Q): \QP\to\mathbb{Q}$ and $Q=Q(B): \mathbb{Q}\to\QP$ respectively.
For $\Lambda, \delta>0$, we introduce compact subsets of $\mathbb{Q}$ as
\begin{align}
\mathbb{Q}_{\Lambda}&~ =\{Q\in\mathbb{Q}: \text{all the eigenvalues of $Q$ belong to }[-\Lambda, \Lambda]\},\\
\QPd&~=\{Q\in\mathbb{Q}: \text{all the eigenvalues of $Q$ belong to }[-\frac13+\delta, \frac23-\delta]\}.
\end{align}
The next proposition tells us that $B(Q)$ maps a compact subset of $\QP$ to a compact subset of $\mathbb{Q}$.
\begin{proposition}\label{prop:closed}
For any $\delta>0$, there is a positive constant
$\Lambda=\Lambda(\delta)$ such that, for all $Q\in \QPd$, $B_Q\in\mathbb{Q}_{\Lambda}$.
\end{proposition}
\begin{proof}
We only have to consider the case $Q$ and $B$ are both diagonal. Assume $Q=\text{diag}\{q_1,q_2,q_3\}$
and $B=\text{diag}\{b_1,b_2,b_3\}$ with $b_1\ge b_2\ge b_3$. Let
\begin{align}
U=\Big\{\mm: m_3^2<\frac{\delta}{8},~m_2^2<\frac{\delta}{4}\Big\},\quad V=\Big\{\mm: m_3^2>\frac{\delta}{2}\Big\}.
\end{align}
Then $U\cap V=\varnothing$, and
\begin{align}
\int_U {e}^{(b_2-b_1)m_2^2+(b_3-b_1)(m_3^3-\frac\delta4)}\ud\mm &~\ge \int_U {e}^{(b_2-b_1)\frac{\delta}{8}
-(b_3-b_1)\frac{\delta}{8}}\ud\mm\ge\text{meas}(U),\\
\int_V {e}^{(b_2-b_1)m_2^2+(b_3-b_1)(m_3^3-\frac\delta4)}\ud\mm &~\le \int_V {e}^{(b_3-b_1)\frac{\delta}{4}}\ud\mm={e}^{(b_3-b_1)\frac{\delta}{4}}\text{meas}(V).
\end{align}
Therefore, we have
\begin{align*}
q_3+\frac13&~=\frac1{\int_\BS{e}^{b_1m_1^2+b_2m_2^2+b_3m_3^3}\ud\mm}\Big({\int_\BS m_3^2{e}^{b_1m_1^2+b_2m_2^2+b_3m_3^3}\ud\mm}\Big)\\
&~=\frac1{\int_\BS{e}^{(b_2-b_1)m_2^2+(b_3-b_1)(m_3^3-\frac\delta4)}\ud\mm}\Big({\int_{\BS\backslash V}+\int_V m_3^2{e}^{(b_2-b_1)m_2^2+(b_3-b_1)(m_3^3-\frac\delta4)}\ud\mm}\Big)\\
&~\le \frac{\delta}{2}+\frac1{\int_\BS{e}^{(b_2-b_1)m_2^2+(b_3-b_1)(m_3^3-\frac\delta4)}\ud\mm}
\Big(\int_V{e}^{(b_2-b_1)m_2^2+(b_3-b_1)(m_3^3-\frac\delta4)}\ud\mm\Big)\\
&~\le \frac{\delta}{2}+\frac{\text{meas}(V)}{\text{meas}(U)}{e}^{(b_3-b_1)\frac{\delta}{4}},
\end{align*}
which implies
\begin{align}
b_1-b_3\le \frac{4}{\delta}\ln\Big(\frac{2~\text{meas}(V)}{\delta~\text{meas}(U)}\Big)\triangleq \Lambda(\delta).
\end{align}
This concludes the proof of the proposition.
\end{proof}
\begin{proposition}\label{prop:positive}
The Jacobian matrix $\nabla_BQ(B)$ is positive definite for any $B\in\mathbb{Q}$. Consequently, $B(Q)$ is a smooth map
from $\QP$ to $\mathbb{Q}$.
\end{proposition}
\begin{proof}
It is straightforward to calculate that for any non-zero $E\in\mathbb{Q}$, it holds
\begin{align}
&\langle\nabla_BQ(B)E, E\rangle\nonumber\\
&=\frac{\int_\BS(\mm\mm:E)^2\exp(B:\mm\mm)\ud\mm}{\int_\BS\exp(B:\mm\mm)\ud\mm}
-\frac{\big(\int_\BS(\mm\mm:E)\exp(B:\mm\mm)\ud\mm\big)^2}
{\big(\int_\BS\exp(B:\mm\mm)\ud\mm\big)^2}\nonumber\\
&=\frac{\int_\BS\int_{\BS}\Big[(\mm\mm:E)\exp(B:\mm\mm)-(\mm'\mm':E)\exp(B:\mm'\mm')\Big]^2\ud\mm\ud\mm'}
{\big(\int_\BS\exp(B:\mm\mm)\ud\mm\big)^2}>0.\nonumber
\end{align}
Thus, the Jacobian $\nabla_BQ(B)$ is positive definite.
Together with the fact that $Q(B)$ is a smooth function of $B$, we know the inverse $B(Q)$ is also smooth.
\end{proof}
We give some estimates related to the Bingham map.
\begin{lemma}\label{prop:compo}
For any $\delta>0$, $k\in N^*$ and constant matrix $Q^*\in\QP$, there is a positive constant
$C=C({\delta},Q^*)$  such that if $Q(\xx)\in\QPd$, then
\begin{align*}
& \|B(Q)-B(Q^*)\|_{H^k}\leq C \|Q-Q^*\|_{H^k}.
\end{align*}
\end{lemma}
The above lemma is a direct consequence of Proposition \ref{prop:positive} and Lemma \ref{lem:composition} by using change of variables.
\begin{lemma}\label{prop:diff}
For any $\delta>0$, there is a positive constant
$C_{\delta}$ depending on $\delta$ such that if $Q_1, Q_2\in\QPd$
\begin{align*}
& |B(Q_1)-B(Q_2)|\leq C_{\delta}|Q_1-Q_2|.
\end{align*}
Thus
\begin{align*}
& |\partial_i B(Q)|\leq C_{\delta}|\partial_i Q|.
\end{align*}
Moreover, for $ k\in N^*$, there exists a constant $C=C(\delta,\|Q_1-Q^*\|_{H^k},\|Q_2-Q^*\|_{H^k})$ such that
\begin{align*}
& \|B(Q_1)-B(Q_2)\|_{H^k}\leq C(\delta,\|Q_1-Q^*\|_{H^k},\|Q_2-Q^*\|_{H^k})\|Q_1-Q_2\|_{H^k}.
\end{align*}
\end{lemma}
\begin{proof} The first assertion is a direct consequence of Proposition \ref{prop:positive}.
The second one can be induced by Proposition \ref{prop:positive} and Lemma \ref{lem:difference}.
\end{proof}

\begin{remark}\label{rmk:M4}
Since $M^{(4)}_Q$ is a smooth function of $B_Q=B(Q)$, it shares the same estimates with $B(Q)$.
\end{remark}

Now we give some properties for the operator $\CM_{Q}: \mathbb{R}^{3\times3}\to\mathbb{R}^{3\times3}$
\begin{align*}
\mathcal{M}_Q(A)=&\frac13 A+Q\cdot A-A:M^{(4)}_Q.
\end{align*}
Note that $\CM_{Q}$ is defined not only for the symmetric matrix, and $\CM_{Q}(A)$ is not necessarily symmetric even if $A$ is symmetric.
The following Lemma \ref{CM-lemma} gives some basic properties of $\CM_{Q}$, which proof can be found in \cite{WZZ3}.
\begin{lemma}\label{CM-lemma}
 $(i)$ For any $Q\in\QP$, it holds that
 \[
  \CM_{Q}(B_Q)=\frac32Q.
 \]
$(ii)$ $\CM_Q$ is self-adjoint on $\mathbb{R}^{3\times3}$, i.e., $\CM_Q(A):B=\CM_Q(B):A$ for $A,B\in\mathbb{R}^{3\times3}$;\\
$(iii)$ For any $Q\in\QP$ and $A\in\mathbb{R}^{3\times3}$, the operator $\CM_Q(A)$ is positive, i.e.,
\[
 \CM_{Q}(A):A\geq0.
\]
\end{lemma}

\begin{lemma}
For any $\delta>0$, there is a positive constant
$C_{\delta}$ depending on $\delta$ such that if $Q(\xx)\in\QPd, A\in\mathbb{R}^{3\times3}$,  it holds for any multiple index $a$,
\begin{align}\label{esti:M}
& \|\partial^a\CM_{Q}(A)-\CM_{Q}(\partial^a A)\|_{L^2}\le C_\delta(\|\nabla Q\|_{L^\infty}\|A\|_{H^{|a|-1}}+\|\nabla Q\|_{H^{|a|-1}}\|A\|_{L^\infty}).
\end{align}
Moreover, if $|a|\ge 2$, we have
\begin{align}\label{esti:M1}
& \|\partial^a\CM_{Q}(A)-\CM_{Q}(\partial^a A)\|_{L^2}\le C_\delta\|\nabla Q\|_{H^{|a|}}\|A\|_{H^{|a|-1}}.
\end{align}
\end{lemma}
\begin{proof}
With Lemma \ref{lem:commutator}, Lemma \ref{prop:diff} and Remark \ref{rmk:M4}, direct computation shows that
\begin{align*}
\|\partial^a(A:M^{(4)}_Q)-\partial^{a}A:M^{(4)}_Q\|_{L^2}
&~\le C(\|\nabla M^{(4)}_{Q}\|_{L^\infty}\|A\|_{H^{|a|-1}}+\|\nabla M^{(4)}_{Q}\|_{H^{|a|-1}}
\|A\|_{L^\infty})\nonumber\\ \nonumber
&~\le C_\delta(\|\nabla{B}\|_{L^\infty}\|A\|_{H^{|a|-1}}+\|\nabla{B}\|_{H^{|a|-1}}\|A\|_{L^\infty})\\
&~\le C_\delta(\|\nabla Q\|_{L^\infty}\|A\|_{H^{|a|-1}}+\|\nabla Q\|_{H^{|a|-1}}\|A\|_{L^\infty}).
\end{align*}
(\ref{esti:M1}) can be deduced by the same argument with Lemma \ref{lem:commutator}.
\end{proof}

\begin{lemma}\label{prop:diff-CM}
For any $\delta>0$ and $ k\in N^*$, there exist constants $C_1=C_1(\delta)$ and
$C_2=C_2(\delta,\|Q_1-Q^*\|_{H^k},\|Q_2-Q^*\|_{H^k})$ such that
\begin{align*}
 \|\CM_{Q_1}(A)-\CM_{Q_2}(A)\|_{H^k}\leq&~ C_1\|A\|_{H^k}\|Q_1-Q_2\|_{L^\infty}
+C_2\|A\|_{L^\infty}\|Q_1-Q_2\|_{H^k}.
\end{align*}
If $0\le k\le 2$, there exist constant $C=C(\delta,\|Q_1-Q^*\|_{H^2},\|Q_2-Q^*\|_{H^2})$ such that
\begin{align*}
\|\CM_{Q_1}(A)-\CM_{Q_2}(A)\|_{H^k}\leq&~ C \|A\|_{H^{2}}\|Q_1-Q_2\|_{H^k},\\
\|\CM_{Q_1}(A)-\CM_{Q_2}(A)\|_{H^k}\leq&~ C \|A\|_{H^{k}}\|Q_1-Q_2\|_{H^2}.
\end{align*}
\end{lemma}
\begin{proof}
From Lemma \ref{lem:product}, we have that
\begin{align*}
 \|\CM_{Q_1}(A)-\CM_{Q_2}(A)\|_{H^k}\leq& C(\|Q_1-Q_2\|_{L^\infty}\|A\|_{H^k}+\|Q_1-Q_2\|_{H^k}\|A\|_{L^\infty})\\
 &+C(\|M^{(4)}_{Q_1}-M^{(4)}_{Q_2}\|_{L^\infty}\|A\|_{H^k}+\|M^{(4)}_{Q_1}-M^{(4)}_{Q_2}\|_{H^k}\|A\|_{L^\infty}).
\end{align*}
Then the conclusion can be deduced from Lemma \ref{prop:diff} and Remark \ref{rmk:M4}.
\end{proof}

\subsection{The energy functional and critical points}
The bulk part of free energy density functional takes the following form
\begin{align*}
{f}_{bulk}(Q)\eqdefa -\ln Z_Q +B_Q:Q-\frac12\alpha |Q|^2.
\end{align*}
A direct calculation yields that
\begin{align*}
 \frac{\partial{f}_{bulk}(Q)}{\partial Q}=0~~\Rightarrow~~ B_Q-\alpha Q=0.
\end{align*}
We say that a tensor $Q_0$ is a critical point of the bulk free energy density functional ${f}_{bulk}(Q)$
if $Q_0$ satisfies $B_{Q_0}-\alpha Q_0=0$. The critical points are completely classified in \cite{LZZ,FS}.
\begin{proposition}\label{prop:critical}
Let $\eta$ be a solution of the equation
\begin{eqnarray}\label{criti1}
 \frac{3e^{\eta}}{\int^1_0e^{\eta x^2}dx}=3+2\eta+\frac{\eta^2}{\alpha}.
\end{eqnarray}
Then there holds
\begin{eqnarray}\label{criti2}
B_Q-\alpha Q=0~~\Longleftrightarrow~~ B_Q=\eta (\nn\nn-\frac13\II), ~\nn\in \mathbb{S}^2,
\end{eqnarray}
and there exists a critical number $\alpha^{\ast}>0$ such that

$(i)$ when $\alpha<\alpha^{\ast}, \eta=0$ is the only solution of $(\ref{criti1})$;

$(ii)$ when $\alpha=\alpha^{\ast}$, besides $\eta=0$ there is another solution $\eta=\eta^{\ast}$ of $(\ref{criti1})$;

$(iii)$ when $\alpha>\alpha^{\ast}$, besides $\eta=0$ there are other two solutions
$\eta_1>\eta^{\ast}>\eta_2$ of $(\ref{criti1})$.

\end{proposition}

\begin{figure}[htbp]
\centering
\includegraphics[height=50mm,width=60mm]{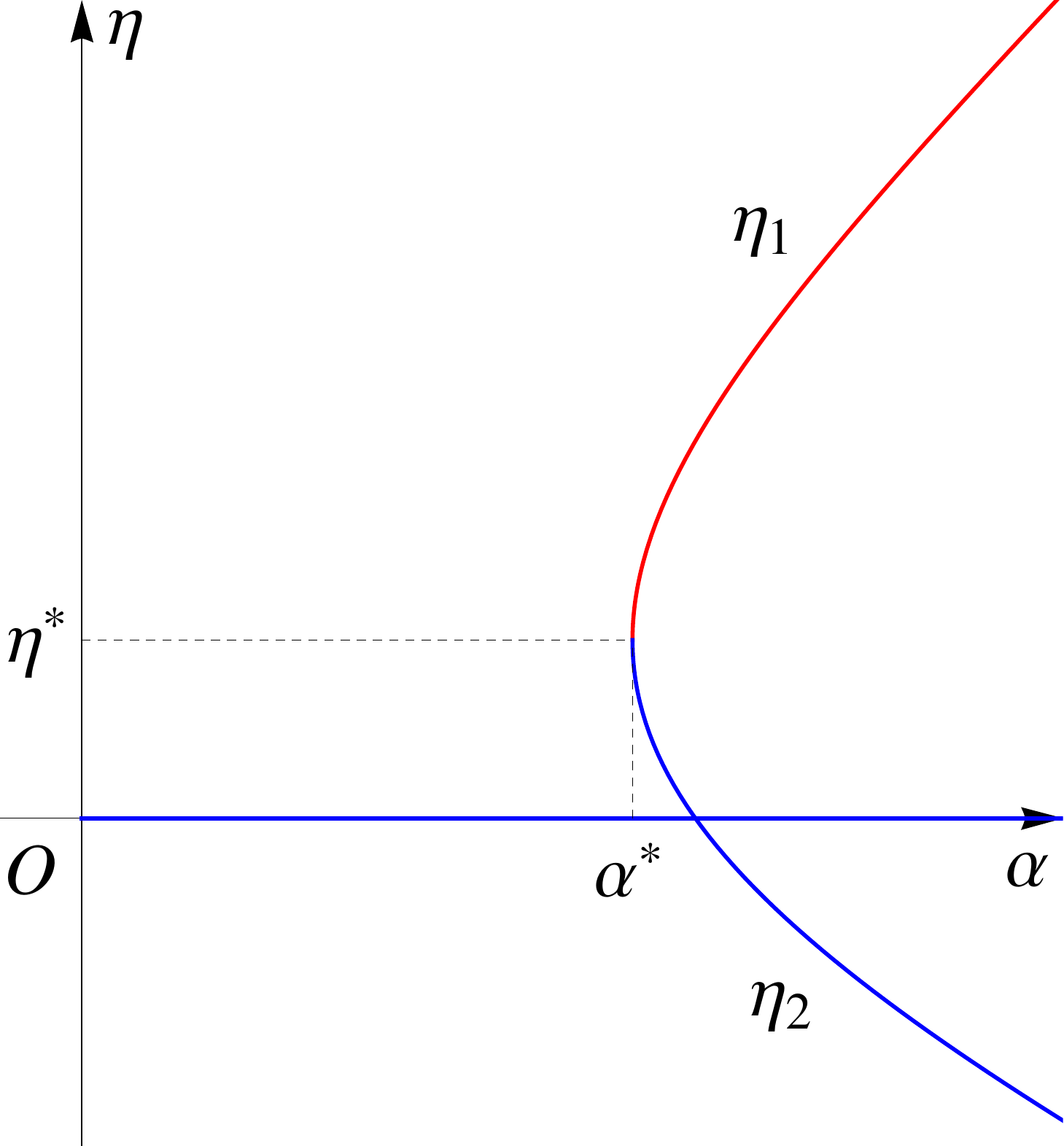}
\caption{The $\alpha-\eta$ curve of the critical point.} \label{fig:graph}
\end{figure}

In the sequel, we always choose $\alpha>\alpha^*$, and $\eta=\eta_1(\alpha)$
corresponding the stable equilibrium solution. We also introduce some important constants
used in this paper. All of them only depend on the parameter $\alpha$.

We define
\begin{align}
A_k=\int_{-1}^1x^ke^{\eta x^2}\ud x,\quad S_k=\frac{\int_{-1}^1P_k(x)e^{\eta x^2}\ud x}{\int_{-1}^1e^{\eta x^2}\ud x},
\end{align}
where $P_k(x)$ is the $k$-th order Legendre polynomial. Particularly,
\begin{align*}
P_2(x)=\frac12(3x^2-1),\quad P_4(x)=\frac{1}{8}(35x^4-30x^2+3).
\end{align*}
Then we have
\begin{align}\label{eqn:S}
S_2=\frac{3A_2-A_0}{2A_0},~~~S_4=\frac{35A_4-30A_2+3A_0}{8A_0}.
\end{align}
An important fact induced by Proposition \ref{prop:critical} is that
\begin{eqnarray}
B_Q-\alpha Q=0~~\Longleftrightarrow~~ Q=S_2 (\nn\nn-\frac13\II), ~\nn\in \mathbb{S}^2.
\end{eqnarray}
The relation
\begin{align}\label{eqn:alpha}
\alpha=\frac{A_0}{A_2-A_4},
\end{align}
and the inequalities
\begin{align}\label{AA024}
 3A^2_2+2A_0A_2-5A_0A_4>0,\quad 6A_2-5A_4-A_0>0,
\end{align}
will play important roles in Section 4. Their proofs can be found in \cite{WZZ1}, and we omit them here.

\section{Existence and uniqueness of the local strong solution for the dynamical $Q$-tensor systems}
This section is devoted to the proof for the existence of the local strong solution of the system (\ref{eqnewQ1})
-(\ref{eqnewQ3}). For $s\ge2$, we define the space:
\begin{align}\nonumber
X(\delta, T, C_0)=&\Big\{(Q,\vv): Q\in\mathbb{Q}_{phy,\delta/2},~\|Q-Q^{\ast}\|_{H^{s+1}}+\|\CL(Q)\|_{L_t^2(H_x^s)}\\
&\qquad\qquad\qquad\qquad+\|\vv\|_{H^s}+\|\nabla\vv\|_{L_t^2(H_x^s)}
\le C_0, ~\text{a.e. }t\in[0,T]\Big\}.
\end{align}
If $(Q,\vv)\in X$, then by Sobolev imbedding, we have
\begin{align*}
\|Q\|_{L^\infty}+\|\nabla Q\|_{L^\infty}+\|\vv\|_{L^\infty}\le C(C_0).
\end{align*}
The proof of Theorem \ref{thm:main1} is based on iterative argument and a closed energy estimate.

\subsection{Linearized system and iteration scheme}

First of all, we take
$$(Q^{(0)}(t,\xx), \vv^{(0)}(t,\xx))=(Q_I(\xx),\vv_I(\xx))\in X(\delta, T, C_0).$$
Assuming that $(Q^{(n)}, \vv^{(n)})\in X(\delta, T, C_0)$ has been constructed,
we construct $(Q^{(n+1)}, \vv^{(n+1)})$ by solving the following linearized system:
\begin{align}
 &\frac{\partial Q^{(n+1)}}{\partial t}+\vv^{(n)}\cdot \nabla Q^{(n+1)}=-\frac{2}{De}\Big(\mathcal{M}_{Q^{(n)}}
       (B_{Q^{(n)}}-\alpha Q^{(n)}+\ve\CL (Q^{(n+1)}))\nonumber\\
      &\qquad+\mathcal{M}^{T}_{Q^{(n)}}(B_{Q^{(n)}}-\alpha Q^{(n)}+\ve\CL (Q^{(n+1)}))\Big)
       +\mathcal{M}_{Q^{(n)}}(\nabla \vv^{(n+1)})+\mathcal{M}^{T}_{Q^{(n)}}(\nabla \vv^{(n+1)}),\label{iter-Q}\\
&\frac{\partial \vv^{(n+1)}}{\partial t}+\vv^{(n)}\cdot\nabla \vv^{(n+1)}=-\nabla p^{(n+1)}+\frac{\gamma}{Re}\Delta \vv^{(n+1)}+
        \frac{1-\gamma}{2Re}\nabla\cdot(\DD^{(n+1)}:M^{(4)}_{Q^{(n)}})\nonumber\\
    &\qquad+\frac{1-\gamma}{DeRe}\nabla\cdot\Big(2\mathcal{M}_{Q^{(n)}}
        (B_{Q^{(n)}}-\alpha Q^{(n)}+\ve\CL (Q^{(n+1)}))
    +\ve\sigma^d(Q^{(n)},Q^{(n+1)})\Big),\label{iter-v}\\
&\nabla\cdot \vv^{(n+1)}=~0,\label{iter-incom}
\end{align}
with initial data:
\begin{align}
(Q^{(n+1)}(0,\xx), \vv^{(n+1)}(0,\xx))=(Q_I(\xx),\vv_I(\xx)).
\end{align}
The existence of $(Q^{(n+1)}, \vv^{(n+1)})$ is ensured by the classical parabolic theory, see \cite{BCD} for example.
Now we prove that $(Q^{(n+1)}, \vv^{(n+1)})\in X$, for a suitably chosen $T>0$.

Define the energy functional
\begin{align*}
E_s(\vv,Q)=&\int_{\mathbb{R}^3}\Big(|Q-Q^{\ast}|^2
  +\frac{(1-\gamma)\ve}{2DeRe}(L_1|\nabla Q|^2+2L_2|Q_{ij,j}|^2)+\frac12|\vv|^2\\
&  \qquad+\frac{(1-\gamma)\ve}{2DeRe}\big(L_1|\nabla^{s+1} Q|^2+2L_2|\nabla^{s}Q_{ij,j}|^2\big)
+\frac12|\nabla^s \vv|^2\Big)\ud\xx,\\
 F_s(\vv,Q)=&\int_{\mathbb{R}^3}\Big(|\CL (Q)|^2
 +|\nabla^s\CL (Q)|^2+|\nabla\vv|^2+|\nabla^{s+1}\vv|^2\Big)\ud\xx.
\end{align*}
Obviously, we have
$$E_s\sim \|Q-Q^{\ast}\|_{L^2}^2+\|\nabla Q\|_{H^s}^2+\|\vv\|_{H^s}^2,
\quad F_s\sim \|\nabla Q\|_{H^{s+1}}^2+\|\nabla\vv\|_{H^s}^2.$$
Let $E_s^{(n)}=E_s(Q^{(n)}, \vv^{(n)})$.
We will prove the following closed energy estimates:
\begin{align}
 &\frac12\frac{d}{dt}E_s^{(n+1)}+\nu F_s^{(n+1)}\leq C(\delta, C_0,\nu)(1+E^{(n+1)}_s),
\end{align}
for some small $\nu>0$. The proof is split into three steps.

{\bf Step 1. $L^2$ energy estimate for $Q^{(n+1)}-Q^{\ast}$}

From Lemma \ref{prop:diff}, we have
\begin{align*}
\|B_{Q^{(n)}}-\alpha Q^{(n)}\|_{L^2}
=\|B_{Q^{(n)}}-\alpha Q^{(n)}-B_{Q^{\ast}}+\alpha Q^{\ast}\|_{L^2}\le C_{\delta}\|Q^{(n)}-Q^{\ast}\|_{L^2}\le C(\delta,C_0).
\end{align*}
Therefore, by making $L^2$ inner product to (\ref{iter-Q}) with $Q^{(n+1)}-Q^{\ast}$, we get
\begin{align}\label{Q-estimate}
&\frac12\frac{d}{dt}\|Q^{(n+1)}-Q^{\ast}\|^2_{L^2}=\big\langle\partial_t Q^{(n+1)},Q^{(n+1)}-Q^{\ast}\big\rangle\nonumber\\
&=-\frac{4\ve}{De}\big\langle\mathcal{M}_{Q^{(n)}}(\CL(Q^{(n+1)})),Q^{(n+1)}-Q^{\ast}\big\rangle
   +\frac{4}{De}\big\langle\mathcal{M}_{Q^{(n)}}(-B_{Q^{(n)}}+\alpha Q^{(n)}),Q^{(n+1)}-Q^{\ast}\big\rangle\nonumber\\
 &\quad+2\big\langle\mathcal{M}_{Q^{(n)}}(\nabla \vv^{(n+1)}),Q^{(n+1)}-Q^{\ast}\big\rangle\nonumber\\
&\leq\Big(C_{\delta}\|\CL (Q^{(n+1)})\|_{L^2}
+C(\delta,C_0)+C_{\delta}\|\nabla \vv^{(n+1)}\|_{L^2}\Big)\|Q^{(n+1)}-Q^{\ast}\|_{L^2}\nonumber\\
&\leq  C(\delta,C_0)\big((E_s^{(n+1)})^{1/2}+E_s^{(n+1)}\big).
\end{align}

{\bf Step 2. $L^2$ estimates for $(\nabla Q^{(n+1)}, \vv^{(n+1)})$}

In this step and the next step, a key point is that we will use the self-adjointness of $\CM_{Q^{{n}}}$.
By making $L^2$ inner product to (\ref{iter-Q}) with $\CL(Q^{(n+1)})$, we get
\begin{align}\label{nablaQ-estimate}
&\frac12\frac{d}{dt}\int_\BR\Big(L_1|\nabla Q|^2+2L_2|Q_{ij,j}|^2\Big)\ud\xx
=\big\langle\partial_t Q^{(n+1)},\CL (Q^{(n+1)})\big\rangle\nonumber\\
=&-\big\langle\vv^{(n)}\cdot \nabla Q^{(n+1)},\CL (Q^{(n+1)})\big\rangle
   -\frac{4\ve}{De} \big\langle\mathcal{M}_{Q^{(n)}}(\CL(Q^{(n+1)})),\CL (Q^{(n+1)})\big\rangle\nonumber\\
 &+\frac{4}{De}\big\langle\mathcal{M}_{Q^{(n)}}(-B_{Q^{(n)}}+\alpha Q^{(n)}),\CL (Q^{(n+1)})\big\rangle
+2\big\langle\mathcal{M}_{Q^{(n)}}(\nabla \vv^{(n+1)}),\CL (Q^{(n+1)})\big\rangle\nonumber\\
\leq&\|\vv^{(n)}\|_{L^{\infty}}\|\nabla Q^{(n+1)}\|_{L^2}\|\CL (Q^{(n+1)})\|_{L^2}
 -\nu\|\CL (Q^{(n+1)})\|^2_{L^2}\nonumber\\
&+C(\delta,C_0)\|\CL (Q^{(n+1)})\|_{L^2}
 +2\big\langle\mathcal{M}_{Q^{(n)}}(\nabla \vv^{(n+1)}),\CL (Q^{(n+1)})\big\rangle\nonumber\\
\leq& C(\delta, C_0,\nu)(1+E_s^{(n+1)})-\nu\|\CL (Q^{(n+1)})\|^2_{L^2}
  +2\big\langle\mathcal{M}_{Q^{(n)}}(\nabla \vv^{(n+1)}),\CL (Q^{(n+1)})\big\rangle.
\end{align}
From (\ref{iter-v}), we have
\begin{align}\label{v-estimate}
&\frac12\frac{d}{dt}\|\vv^{(n+1)}\|^2_{L^2}=\big\langle\partial_t \vv^{(n+1)},\vv^{(n+1)}\big\rangle\nonumber\\
=&\frac{\gamma}{Re}\big\langle\Delta \vv^{(n+1)},\vv^{(n+1)}\big\rangle
    +\frac{1-\gamma}{2Re}\big\langle\nabla\cdot(\DD^{(n+1)}:M^{(4)}_{Q^{(n)}}),\vv^{(n+1)}\big\rangle\nonumber\\
 &  +\frac{2(1-\gamma)}{DeRe}\big\langle \mathcal{M}_{Q^{(n)}}(-B_{Q^{(n)}}+\alpha Q^{(n)}),\nabla\vv^{(n+1)}\big\rangle\nonumber\\
 & -\frac{2(1-\gamma)\varepsilon}{DeRe}\big\langle\mathcal{M}_{Q^{(n)}}
        (\CL (Q^{(n+1)})),\nabla \vv^{(n+1)}\big\rangle
   -\frac{(1-\gamma)\varepsilon}{DeRe}\big\langle\sigma^d(Q^{(n)}, Q^{(n+1)}),\nabla\vv^{(n+1)}\big\rangle\nonumber\\
\leq&-\frac{\gamma}{Re}\|\nabla \vv^{(n+1)}\|^2_{L^2}
   -\frac{1-\gamma}{2Re}\big\langle\DD^{(n+1)}:M^{(4)}_{Q^{(n)}},\DD^{(n+1)}\big\rangle
   +C_\delta\|\nabla\vv^{(n+1)}\|_{L^2}\nonumber\\
   &-\frac{2(1-\gamma)\varepsilon}{DeRe}\big\langle\mathcal{M}_{Q^{(n)}}
        (\nabla \vv^{(n+1)}),\CL (Q^{(n+1)})\big\rangle
  +C \|\nabla Q^{(n)}\|_{H^2}\|\nabla Q^{(n+1)}\|_{L^2}\|\nabla\vv^{(n+1)}\|_{L^2}\nonumber\\
\leq&-\frac{\gamma}{2Re}\|\nabla \vv^{(n+1)}\|^2_{L^2}+C_\delta\|\nabla\vv^{(n+1)}\|_{L^2}
-\frac{2(1-\gamma)\varepsilon}{DeRe}\big\langle\mathcal{M}_{Q^{(n)}}
        (\nabla \vv^{(n+1)}),\CL (Q^{(n+1)})\big\rangle\nonumber\\
&  +C(C_0)\|\nabla Q^{(n+1)}\|_{L^2}^2.
\end{align}
 Thus,  we obtain from (\ref{nablaQ-estimate})-(\ref{v-estimate}) that
\begin{align}\label{low-estimateQv}
 \frac12\frac{d}{dt}\Big(\|\vv^{(n+1)}\|^2_{L^2}
 +\frac{(1-\gamma)\varepsilon}{DeRe}\int_\BR(L_1|\nabla Q|^2+2L_2|Q_{ij,j}|^2)\ud\xx\Big)&\nonumber\\
 +\frac{\gamma}{2Re}\|\nabla \vv^{(n+1)}\|^2_{L^2}
 +\nu\|\CL (Q^{(n+1)})\|^2_{L^2} \leq~ &C(\delta,C_0,\nu)\big(1+E_s^{(n+1)}\big).
\end{align}

{\bf Step 3. $L^2$ estimates for $(\nabla^{s+1} Q^{(n+1)}, \nabla^s\vv^{(n+1)})$}

We now turn to the estimate of the higher order derivative for $Q^{(n+1)}$,
\begin{align*}
&\frac12\frac{d}{dt}\int_\BR\Big(L_1|\nabla^{s+1} Q|^2+2L_2|\nabla^s Q_{ij,j}|^2\Big)\ud\xx
=\big\langle\nabla^s\partial_t Q^{(n+1)},\nabla^s\CL (Q^{(n+1)})\big\rangle\nonumber\\
&=\underbrace{-\big\langle\nabla^s(\vv^{(n)}\cdot\nabla Q^{(n+1)}),\nabla^s\CL (Q^{(n+1)})\big\rangle}_{I}
  +\frac{4\ve}{De}\underbrace{\big\langle-\nabla^s\mathcal{M}_{Q^{(n)}}(\CL (Q^{(n+1)})),\nabla^s\CL (Q^{(n+1)})\big\rangle}_{II}\nonumber\\
  &\quad+\frac{4}{De}\underbrace{\big\langle\nabla^s\mathcal{M}_{Q^{(n)}}(-B_{Q^{(n)}}+\alpha Q^{(n)}),\nabla^s\CL (Q^{(n+1)})\big\rangle}_{III}
   +\underbrace{2\big\langle\nabla^s\mathcal{M}_{Q^{(n)}}(\nabla\vv^{(n+1)}), \nabla^s\CL (Q^{(n+1)})\big\rangle}_{IV}.\nonumber
\end{align*}
These terms can be estimated as following:
\begin{align*}
I&~\leq C\|\vv^{(n)}\|_{H^s}\|\nabla Q^{(n+1)}\|_{H^s}\|\nabla^s\CL (Q^{(n+1)})\|_{L^2}\le C(\delta, C_0)(E_s^{(n+1)}F_s^{(n+1)})^{1/2},\nonumber\\
II
&~\leq-\big\langle\mathcal{M}_{Q^{(n)}}(\nabla^s\CL (Q^{(n+1)})),\nabla^s\CL (Q^{(n+1)})\big\rangle
+\big\langle [\nabla^s,\mathcal{M}_{Q^{(n)}} ]\CL (Q^{(n+1)}), \nabla^s\CL (Q^{(n+1)})\big\rangle\\
&~\leq -\nu\|\nabla^s\CL (Q^{(n+1)})\|^2_{L^2}
+C(\delta)\|{Q^{(n)}}-Q^*\|_{H^s}
\|\CL (Q^{(n+1)})\|_{H^{s-1}}\|\nabla^s\CL (Q^{(n+1)})\|_{L^2}\nonumber\\
&~\le  -\nu\|\nabla^s\CL (Q^{(n+1)})\|^2_{L^2}
+C(\delta, C_0)(E_s^{(n+1)}F_s^{(n+1)})^{1/2},\\
III &~\le C_\delta\|Q^{(n)}-Q^*\|_{H^{s+1}}\|\nabla^{s-1}\CL (Q^{(n+1)})\|_{L^2}
\le C(\delta, C_0)(E_s^{(n+1)})^{1/2},\\
IV&~=2\big\langle\mathcal{M}_{Q^{(n)}}(\nabla^{s+1}\vv^{(n+1)}), \nabla^s\CL (Q^{(n+1)})\big\rangle
+2\big\langle [\nabla^s,\mathcal{M}_{Q^{(n)}} ]\nabla \vv^{(n+1)}, \nabla^s\CL (Q^{(n+1)})\big\rangle
\nonumber\\
&~\le 2\big\langle\mathcal{M}_{Q^{(n)}}(\nabla^{s+1}\vv^{(n+1)}), \nabla^s\CL (Q^{(n+1)})\big\rangle+C(\delta)\|{Q^{(n)}}\|_{H^s}
\|\nabla \vv^{(n+1)}\|_{H^{s-1}}\|\nabla^s\CL (Q^{(n+1)})\|_{L^2}\nonumber\\
&~\le 2\big\langle\mathcal{M}_{Q^{(n)}}(\nabla^{s+1}\vv^{(n+1)}), \nabla^s\CL (Q^{(n+1)})\big\rangle
+C(\delta, C_0)(E_s^{(n+1)}F_s^{(n+1)})^{1/2}.
\end{align*}
Thus we get
\begin{align}\nonumber
&\frac12\frac{d}{dt}\int_\BR\Big(L_1|\nabla^{s+1} Q|^2+2L_2|\nabla^s Q_{ij,j}|^2\Big)\ud\xx\\
&\le2
\big\langle\mathcal{M}_{Q^{(n)}}(\nabla^{s+1}\vv^{(n+1)}), \nabla^s\CL (Q^{(n+1)})\big\rangle -\nu\|\nabla^s\CL (Q^{(n+1)})\|^2_{L^2}\nonumber\\
&~+C(\delta, C_0)(E_s^{(n+1)})^{1/2}\Big(1+(F_s^{(n+1)})^{1/2}\Big).\label{high-estimateQ}
\end{align}

For the estimate of the higher order derivative for $\vv^{(n+1)}$, we have
\begin{align*}
&\frac12\frac{d}{dt}\frac{DeRe}{1-\gamma}\|\nabla^s\vv^{(n+1)}\|^2_{L^2}
=\frac{DeRe}{1-\gamma}\big\langle\partial_t\nabla^s\vv^{(n+1)},\nabla^s\vv^{(n+1)}\big\rangle\nonumber\\
=&-\frac{DeRe}{1-\gamma}\big\langle\nabla^s(\vv^{(n)}\cdot\nabla\vv^{(n+1)}),\nabla^s\vv^{(n+1)}\big\rangle
+\frac{\gamma De}{1-\gamma}\big\langle\nabla^s\Delta \vv^{(n+1)},\nabla^s\vv^{(n+1)}\big\rangle\nonumber\\
&-\frac{De}{2}\big\langle\nabla^s(\DD^{(n+1)}:M^{(4)}_{Q^{(n)}}),\nabla^{s+1}\vv^{(n+1)}\big\rangle
-{2}\big\langle\nabla^s\mathcal{M}_{Q^{(n)}}
(B_{Q^{(n)}}-\alpha Q^{(n)}),\nabla^{s+1}\vv^{(n+1)}\big\rangle\nonumber\\
& -2\ve\big\langle\nabla^s\mathcal{M}_{Q^{(n)}}(\CL (Q^{(n+1)})),\nabla^{s+1}\vv^{(n+1)}\big\rangle
-\ve\big\langle\nabla^s\big(\sigma^d(Q^{(n)}, Q^{(n+1)})\big),\nabla^{s+1}\vv^{(n+1)}\big\rangle\nonumber\\
\triangleq&~ I+II+III+IV+V+VI.
\end{align*}
Estimating them term by term, we obtain
\begin{align*}
I\le& ~C\big(\|\vv^{(n)}\|_{L^\infty}\|\nabla\vv^{(n+1)}\|_{H^{s}}
+\|\vv^{(n)}\|_{H^s}\|\nabla\vv^{(n+1)}\|_{L^\infty}\big)\|\nabla^s\vv^{(n+1)}\|_{L^2}\\
\le & ~C_\delta\Big(E_s^{(n+1)}(E_s^{(n+1)}+F_s^{(n+1)})\Big)^{1/2},\\
III =&~-\frac{De}{2}\big\langle(\nabla^s\DD^{(n+1)}:M^{(4)}_{Q^{(n)}}),\nabla^{s}\DD^{(n+1)}\big\rangle
+\frac{De}{2}\big\langle[\nabla^s,M^{(4)}_{Q^{(n)}}:]\DD^{(n+1)}, \nabla^{s}\DD^{(n+1)}\big\rangle\\
\le&~-\frac{De}{2}\big\langle(\nabla^s\DD^{(n+1)}:M^{(4)}_{Q^{(n)}}),\nabla^{s}\DD^{(n+1)}\big\rangle
+C(\delta,C_0)\|\nabla Q^{(n)}\|_{H^s}\|\DD^{(n+1)}\|_{H^{s-1}}\|\nabla^{s+1}\vv^{(n+1)}\|_{L^2}\\
\le&~-\frac{De}{2}\big\langle(\nabla^s\DD^{(n+1)}:M^{(4)}_{Q^{(n)}}),\nabla^{s}\DD^{(n+1)}\big\rangle
+C(\delta, C_0)\big(E_s^{(n+1)}F_s^{(n+1)}\big)^{1/2},
\end{align*}
\begin{align*}
IV\le&~ C(\delta)\|\nabla Q^{(n)}\|_{H^s}\|\nabla^{s}\vv^{(n+1)}\|_{L^2}
\le ~C(\delta, C_0)(E_s^{(n+1)})^{1/2},\\
V=&~-2\ve\big\langle\mathcal{M}_{Q^{(n)}}(\nabla^s\CL (Q^{(n+1)})),\nabla^{s+1}\vv^{(n+1)}\big\rangle+2
\ve\big\langle [\nabla^s,\mathcal{M}_{Q^{(n)}}]\CL (Q^{(n+1)}), \nabla^{s+1} \vv^{(n+1)}\big\rangle\\
\le&~-2\ve\big\langle\mathcal{M}_{Q^{(n)}}(\nabla^s\CL (Q^{(n+1)})),\nabla^{s+1}\vv^{(n+1)}\big\rangle+C(\delta)\|{Q^{(n)}}-Q^*\|_{H^s}
\|\CL (Q^{(n+1)})\|_{H^{s-1}}\|\nabla^{s+1}\vv^{(n+1)}\|_{L^2}\nonumber\\
\le&~-2\ve\big\langle\mathcal{M}_{Q^{(n)}}(\nabla^s\CL (Q^{(n+1)})),\nabla^{s+1}\vv^{(n+1)}\big\rangle
+C(\delta, C_0)(E_s^{(n+1)}F_s^{(n+1)})^{1/2},\nonumber\\
VI\le& ~C\|\nabla Q^{(n)}\|_{H^s}\|\nabla Q^{(n+1)}\|_{H^s}\|\nabla\vv^{(n+1)}\|_{H^s}\le C(C_0)(E_s^{(n+1)}F_s^{(n+1)})^{1/2}.
\end{align*}
Thus we get
\begin{align}\label{high-estimatev}
&\frac12\frac{d}{dt}\frac{DeRe}{1-\gamma}\|\nabla^s\vv^{(n+1)}\|^2_{L^2}\nonumber\\
\le &~
-\frac{\gamma De}{1-\gamma}\|\nabla^{s+1}\vv^{(n+1)})\|_{L^2}^2
-2\ve\big\langle\mathcal{M}_{Q^{(n)}}(\nabla^s\CL (Q^{(n+1)})),\nabla^{s+1}\vv^{(n+1)}\big\rangle\nonumber\\
&~+C(\delta, C_0)(E_s^{(n+1)})^{1/2}\Big(1+(F_s^{(n+1)})^{1/2}\Big).
\end{align}

Combining (\ref{Q-estimate}), (\ref{low-estimateQv}), (\ref{high-estimateQ}) and (\ref{high-estimatev}), we know that it holds
\begin{align}\label{apriori-estimate}
 &\frac12\frac{d}{dt}E_s^{(n+1)}+\nu F_s^{(n+1)}\leq  C({\delta},C_0)\big(1+E_s^{(n+1)}\big),
\end{align}
for $\nu>0$ small enough. By Gronwall's inequality, we get
\begin{eqnarray}
E_s^{(n+1)}(t)\leq \big(1+E_s^{(n+1)}(0)\big)e^{C(\delta,C_0)t}-1=\big(1+E_s(Q_I,\vv_I)\big)e^{C(\delta,C_0)t}-1,
\end{eqnarray}
for any $t\in [0,T]$. Then if we take $T_0>0$ such that $C(\delta, C_0)T_0\le \ln(1+C_0)-\ln(1+E_s(Q_I,\vv_I)),$ then
$\sup\limits_{0\le t\le T_0}E_s^{(n+1)}(t)\leq C_0.$  In addition,
\begin{align*}
&\|\int_0^t\partial_tQ^{(n+1)}(t,\xx)dt\|_{L^\infty}\le\int_0^t\|\partial_tQ^{(n+1)}(t,\xx)\|_{H^2}dt\\
&\le C(\delta, C_0)\int_0^t\Big(\|\CL (Q^{(n+1)})\|_{H^2}+\|\nabla\vv^{(n+1)}\|_{H^2}+\|Q^{(n+1)}-Q^*\|_{H^3}+1\Big)dt
\le C(\delta, C_0)t.
\end{align*}
Thus, together with the assumption $Q_I\in\QPd$, it yields that $Q^{(n+1)}\in \mathbb{Q}_{phy,\delta/2}$ for
$t\in[0,T_0]$, if we choose $T_0$ sufficiently small. Then we obtain $(Q^{(n+1)}, \vv^{(n+1)})\in X(\delta, T, C_0)$ for $T\le T_0$.

\subsection{Convergence of the sequence}

In this subsection, we are going to show that the approximate solution sequence
$\{(\vv^{(\ell)},Q^{(\ell)})\}_{\ell\in\mathbb{N}}$ is a Cauchy sequence.

We set
\begin{align*}
 &\delta_{Q}^{\ell+1}=Q^{(\ell+1)}-Q^{(\ell)},~~\delta_{B}^{\ell}=B_{Q^{(\ell)}}-B_{Q^{(\ell-1)}},
 ~~\delta_{M^{(4)}}^{\ell}=M^{(4)}_{Q^{(\ell)}}-M^{(4)}_{Q^{(\ell-1)}},\\
 &\delta_{\vv}^{\ell+1}=\vv^{(\ell+1)}-\vv^{(\ell)},~~\delta_{\DD}^{\ell+1}=\DD^{(\ell+1)}-\DD^{(\ell)},
 ~~\delta_{p}^{\ell+1}=p^{(\ell+1)}-p^{(\ell)}.
\end{align*}
By taking the difference between the equations for $(\vv^{(\ell+1)},Q^{(\ell+1)})$ and
$(\vv^{(\ell)},Q^{(\ell)})$, we find that
\begin{align}
 \frac{\partial \delta_{Q}^{\ell+1}}{\partial t}+\vv^{(\ell)}\cdot \nabla \delta_{Q}^{\ell+1}
        =&\frac{2\ve}{De}\Big(\mathcal{M}_{Q^{(\ell)}}
       (\CL (\delta_{Q}^{\ell+1}))+\mathcal{M}^{T}_{Q^{(\ell)}}(\CL (\delta_{Q}^{\ell+1}))\Big)\nonumber\\
       &+\mathcal{M}_{Q^{(\ell)}}(\nabla \delta_{\vv}^{\ell+1})
       +\mathcal{M}^{T}_{Q^{(\ell)}}(\nabla \delta_{\vv}^{\ell+1})+\delta F^{\ell}_1,\label{cauchyeqnewQ1}\\
\frac{\partial \delta_{\vv}^{\ell+1}}{\partial t}+\vv^{(\ell)}\cdot\nabla \delta_{\vv}^{\ell+1}
           =&-\nabla \delta_{p}^{\ell+1}+\frac{\gamma}{Re}\Delta \delta_{\vv}^{\ell+1}+
        \frac{1-\gamma}{2Re}\nabla\cdot(\delta_{\DD}^{\ell+1}:M^{(4)}_{Q^{(\ell)}})+\nabla\cdot\delta F^{\ell}_2\nonumber\\
    &-\frac{(1-\gamma)\varepsilon}{DeRe}\nabla\cdot\Big(2\mathcal{M}_{Q^{(\ell)}}
        (\CL(\delta_{Q}^{\ell+1}))+\sigma^d( Q^{(\ell)}, \delta_{Q}^{\ell+1}\big)\Big)
     ,\label{cauchyeqnewQ2}\\
\nabla\cdot \delta_{\vv}^{\ell+1}=&~0,\label{cauchyeqnewQ3}
\end{align}
where
\begin{align*}
\delta F^{\ell}_1=&\frac{2\ve}{De}\Big(\delta_{Q}^{\ell}\cdot\CL (Q^{(\ell)})
        +\CL (Q^{(\ell)})\cdot\delta_{Q}^{\ell}+2\CL(Q^{(\ell)}):\delta_{M^{(4)}}^{\ell}\Big)\\
       &+\delta_{Q}^{\ell}\cdot\nabla\vv^{(\ell)}+(\nabla\vv^{(\ell)})^T\cdot\delta_{Q}^{\ell}
       -2\DD^{(\ell)}:\delta_{M^{(4)}}^{\ell}-\delta_{\vv}^{\ell}\cdot\nabla Q^{(\ell)}\\
       &+\frac{2}{De}\Big(\mathcal{M}_{Q^{(\ell)}}(-{B}_{Q^{(\ell)}}+\alpha {Q}^{(\ell)})+\mathcal{M}_{Q^{(\ell)}}^T(-{B}_{Q^{(\ell)}}+\alpha {Q}^{(\ell)})\\
       &-\mathcal{M}_{Q^{(\ell-1)}}(-{B}_{Q^{(\ell-1)}}+\alpha {Q}^{(\ell-1)})-\mathcal{M}_{Q^{(\ell-1)}}^T(-{B}_{Q^{(\ell-1)}}+\alpha {Q}^{(\ell-1)})\Big),\\
\delta F^{\ell}_2=&\frac{1-\gamma}{2Re}\DD^{(\ell)}:\delta_{M^{(4)}}^{\ell}
 -\frac{1-\gamma}{DeRe}\Big(2\mathcal{M}_{Q^{(\ell)}}(-{B}_{Q^{(\ell)}}+\alpha {Q}^{(\ell)})
 -2\mathcal{M}_{Q^{(\ell-1)}}(-{B}_{Q^{(\ell-1)}}+\alpha {Q}^{(\ell-1)})\\
 &+\varepsilon\big(2\delta_{Q}^{\ell}\cdot\CL (Q^{(\ell)})-2\CL (Q^{(\ell)}):\delta_{M^{(4)}}^{\ell}
 +\sigma^d(\delta_{Q}^{\ell}, Q^{(\ell)})\big)\Big)-\delta_{\vv}^{\ell}\otimes\vv^{(\ell)}.
\end{align*}
From Proposition \ref{prop:diff-CM}, we have
\begin{align*}
 \|\delta F^{\ell}_1\|_{L^2}\leq &~C(\delta,C_0)\big(\|\delta_{Q}^{\ell}\|_{H^{1}}+\|\delta_{\vv}^{\ell}\|_{L^2}\big),\\
\|\delta F^{\ell}_2\|_{L^2}\leq& C(\delta,C_0)\big(\|\delta_{Q}^{\ell}\|_{H^1}+\|\delta_{\vv}^{\ell}\|_{L^2}\big).
\end{align*}

Similar to the proof of (\ref{low-estimateQv}), we can deduce that there
exist $\nu>0$ small enough and $C(\delta, C_0,\nu)>0$, such that
\begin{align}\label{low-estimate-diff}
 &\frac12\frac{d}{dt}\Big(\|\delta_\vv^{\ell+1}\|^2_{L^2}+\|\delta_Q^{\ell+1}\|^2_{L^2}
+\frac{(1-\gamma)\ve}{DeRe}\int_\BR(L_1|\nabla \delta_Q^{\ell+1}|^2+2L_2|(\delta_Q^{\ell+1})_{ij,j}|^2)\ud\xx\Big)\nonumber\\
 &+\frac{\gamma}{2Re}\|\nabla \delta_\vv^{\ell+1}\|^2_{L^2}
 +\nu\|\CL(\delta_Q^{\ell+1})\|^2_{L^2}\nonumber\\
 \leq&C(\delta,C_0,\nu)\big(\|\delta_\vv^{\ell+1}\|^2_{L^2}+\|\delta_Q^{\ell+1}\|^2_{H^1}+\|\delta_\vv^{\ell}\|^2_{L^2}+\|\delta_Q^{\ell}\|^2_{H^1}\big).
\end{align}
We denote
\[
 \widetilde{E}^{(\ell)}_0(t)=\|\delta_\vv^{\ell}\|^2_{L^2}+\|\delta_Q^{\ell}\|^2_{L^2}
 +\frac{(1-\gamma)\ve}{DeRe}\int_\BR(L_1|\nabla \delta_Q^{\ell+1}|^2+2L_2|(\delta_Q^{\ell+1})_{ij,j}|^2)\ud\xx.
\]
Then (\ref{low-estimate-diff}) implies
\begin{align*}
 \frac{d}{dt}\widetilde{E}^{(\ell+1)}_{0}(t)
 \leq  C\Big(\widetilde{E}^{(\ell)}_{0}(t)+\widetilde{E}^{(\ell+1)}_{0}(t)\Big).
\end{align*}
Thus, we get
\begin{align*}
 \widetilde{E}^{(\ell+1)}_{0}(t)\leq C\int^t_0e^{C(t-\tau)}\widetilde{E}^{(\ell)}_{0}(\tau)d\tau
 \leq C\int^T_0e^{C(T-\tau)}d\tau \sup_{t\in(0,T]}\widetilde{E}^{(\ell)}_{0}(t).
\end{align*}
Taking $T<T_0$ small enough such that $C\int^T_0e^{C(T-\tau)}d\tau\leq\frac12$, we obtain
\begin{align*}
 \sup_{t\in(0,T]}\widetilde{E}^{(\ell+1)}_{0}(t)\leq\frac12 \sup_{t\in(0,T]}\widetilde{E}^{(\ell)}_{0}(t),
\end{align*}
Therefore, there exist $Q-Q^{\ast}\in C([0,T], H^1)$ and $\vv\in C([0,T], L^2)$, such that
\begin{align}
Q^{(n)}-Q^{\ast}\to Q-Q^{\ast} \in C([0,T], H^1),\quad \vv^{(n)}\to \vv \in C([0,T], L^2).
\end{align}
By the uniform bounds and interpolation, we have for any $s'\in(0,s)$,
\begin{align}
Q^{(n)}-Q^{\ast}\to Q-Q^{\ast} \in C([0,T], H^{s'+1}),\quad \vv^{(n)}\to \vv \in C([0,T], H^{s'}).
\end{align}
Thus we $(Q,\vv)$ is a classical solution of (\ref{eqnewQ1})-(\ref{eqnewQ3}).
The uniqueness of $(Q,\vv)$ is guaranteed by the same energy estimate as we have done to the prove the convergence of
$\{(Q^{(n)},\vv^{(n)})\}$. Moreover, by the standard regularity argument for parabolic system, we have that
\begin{align}
 Q-Q^{\ast} \in C([0,T], H^{s+1}),\quad \vv \in C([0,T], H^{s})\cap L^2([0,T], H^{s+1}).
\end{align}
We omit the details here. This completes the proof of Theorem \ref{thm:main1}.

\section{Some linearized operators}
In this section, we study some important linearized operators which will
be used in deriving the Ericksen-Leslie system from the molecule-based $Q$-tensor system (\ref{eqnewQ1})-(\ref{eqnewQ3}).

For a given $\bar{Q}=Q(\bar{B})$, the linearized operator of $Q(B)$ is defined as:
\begin{align*}
\mathcal{Q}_{\bar{Q}}(B)&~:=\frac{d}{dt}\big(Q(\bar{B}+tB)-Q(\bar{B})\big)\\
&~= M^{(4)}_{\bar{Q}}:B-\big(\frac13\II+\bar{Q}\big)(\bar{Q}:B),\quad\textrm{for } B\in\mathbb{Q}.
\end{align*}
We can also introduce the linearized operator of $B(Q)$ around $\bar{Q}$,
which is actually $\mathcal{Q}_{\bar{Q}}^{-1}$, since
$Q(B)$ and $B(Q)$ are inverse functions of each other.

The following proposition shows that $\mathcal{Q}_{\bar{Q}}$ is a self-adjoint and positive operator.
\begin{proposition}\label{prop:L-Q}
For $\bar{Q}\in\QPd$ and $B_1,B_2\in\mathbb{Q}$, we have $\mathcal{Q}_{\bar{Q}}(B_1):B_2=\mathcal{Q}_{\bar{Q}}(B_2):B_1$.
Moreover, if $B_1\neq 0$, then $\mathcal{Q}_{\bar{Q}}(B_1):B_1\ge 0.$
\end{proposition}
\begin{proof}
By the definition of $M^{(4)}_{\bar{Q}}$ and the fact that $\int_\BS f_{\bar{Q}}\ud\mm=1$, it is direct to check
\begin{align*}
\mathcal{Q}_{\bar{Q}}(B_1):B_2 &~=\int_\BS(\mm\mm:B_1)(\mm\mm:B_2)f_{\bar{Q}}\ud\mm-(\bar{Q}:B_1)(\bar{Q}:B_2)=\mathcal{Q}_{\bar{Q}}(B_2):B_1,\\
\mathcal{Q}_{\bar{Q}}(B_1):B_1 &~=\int_\BS(\mm\mm:B_1)^2f_{\bar{Q}}\ud\mm-\Big(\int_\BS(\mm\mm:B_1)f_{\bar{Q}}\ud\mm\Big)^2\ge0,
\end{align*}
which concludes the proof.
\end{proof}
We are particularly interested in the linearized operators around the equilibrium tensor
$Q_0=S_2(\nn\nn-\frac13\II)$, where $S_2$ are introduced in Section 2.
We denote $\mathcal{Q}_{Q_0}(B)$ by $\mathcal{Q}_\nn(B)$ for $B_0=\eta(\nn\nn-\frac13\II)$.
For use of convenience, we calculate $\mathcal{Q}_\nn$ explicitly.

For the equilibrium tensor $Q_0$, the distribution function $f_{Q_0}$ and the order parameter tensor $M^{(4)}_{Q_0}$
can be written as
\begin{align}
f_{Q_0}=&\frac{e^{\eta(\mm\cdot\nn)^2}}{\int_{\mathbb{S}^2}e^{\eta(\mm\cdot\nn)^2}d\mm},\label{f_Q0n}\\
M^{(4)}_{Q_0,ijkl}=&S_4n_in_jn_kn_l+\frac{S_2-S_4}{7}(n_in_j\delta_{kl}+n_in_k\delta_{jl}
   +n_in_l\delta_{jk}+n_jn_k\delta_{il}\nonumber\\
   &+n_jn_l\delta_{ik}+n_kn_l\delta_{ij})+(\frac{S_4}{35}-\frac{2S_2}{21}+\frac{1}{15})
   (\delta_{ij}\delta_{kl}+\delta_{ik}\delta_{jl}+\delta_{il}\delta_{jk}).\label{MQ04_n}
\end{align}
Substituting (\ref{f_Q0n}) and (\ref{MQ04_n}) to the linear operator
\[
 \mathcal{Q}_{\nn}(Q)=M^{(4)}_{Q_0}:Q-\frac13\II(Q_0:Q)-Q_0(Q_0:Q),
\]
then we have
\begin{eqnarray}\label{def:L_n}
\mathcal{Q}_{\nn}(Q)=\xi_1(\nn\nn-\frac13\II)(\nn\nn:Q)+\xi_2\big(\nn\nn\cdot Q+Q\cdot\nn\nn
-\frac23\II(\nn\nn:Q)\big)+\xi_3Q,
\end{eqnarray}
where
\begin{align}\label{def:xi}
 \xi_1=S_4-S_2^2,~~\xi_2=\frac{2(S_2-S_4)}{7},~~\xi_3=2(\frac{S_4}{35}-\frac{2S_2}{21}+\frac{1}{15}).
\end{align}

To calculate $\mathcal{Q}^{-1}_{\nn}$ explicitly, we may assume that
\begin{eqnarray}\label{def:inverseL_n}
\mathcal{Q}^{-1}_{\nn}(Q)=\psi_1(\nn\nn-\frac13\II)(\nn\nn:Q)+\psi_2(\nn\nn\cdot Q+Q\cdot\nn\nn
-\frac23\II\nn\nn:Q)+\psi_3Q,
\end{eqnarray}
where $\psi_i(1\le i\le 3)$ are constants. Then we have
\begin{eqnarray*}
\mathcal{Q}^{-1}_{\nn}(\mathcal{Q}_{\nn}(Q))&=&\psi_1(\nn\nn-\frac13\II)(\nn\nn:\mathcal{Q}_{\nn}(Q))
+\psi_2\Big(\nn\nn\cdot \mathcal{Q}_{\nn}(Q)+\mathcal{Q}_{\nn}(Q)\cdot\nn\nn\\
&~&-\frac23\II\nn\nn:\mathcal{Q}_{\nn}(Q)\Big)+\psi_3\mathcal{Q}_{\nn}(Q)\\
&=&\psi_1(\frac23\xi_1+\frac43\xi_2+\xi_3)(\nn\nn-\frac13\II)(\nn\nn:Q)
+\psi_2(\frac43\xi_1+\frac23\xi_2)(\nn\nn-\frac13\II)(\nn\nn:Q)\\
&~&+\psi_2(\xi_2+\xi_3)(\nn\nn\cdot Q+Q\cdot\nn\nn-\frac23\II\nn\nn:Q)+\psi_3\xi_1(\nn\nn-\frac13\II)(\nn\nn:Q)\\
&~&+\psi_3\xi_2(\nn\nn\cdot Q+Q\cdot\nn\nn-\frac23\II\nn\nn:Q)+\psi_3\xi_3Q\\
&=&\Big(\psi_1(\frac23\xi_1+\frac43\xi_2+\xi_3)+\psi_2(\frac43\xi_1+\frac23\xi_2)+\psi_3\xi_1\Big)
   (\nn\nn-\frac13\II)(\nn\nn:Q)\\
&~&+\Big(\psi_2(\xi_2+\xi_3)+\psi_3\xi_2\Big)(\nn\nn\cdot Q+Q\cdot\nn\nn-\frac23\II\nn\nn:Q)+\psi_3\xi_3Q.
\end{eqnarray*}
Therefore, the coefficients $\psi_i(1\le i\le 3)$ satisfy
\begin{eqnarray}
&&\psi_1(\frac23\xi_1+\frac43\xi_2+\xi_3)+\psi_2(\frac43\xi_1+\frac23\xi_2)+\psi_3\xi_1=0,\nonumber\\
&&\psi_2(\xi_2+\xi_3)+\psi_3\xi_2=0,~~
\psi_3\xi_3=1.\label{psi23}
\end{eqnarray}
By (\ref{def:xi}) and the definitions of $S_2$ and $S_4$(see (\ref{eqn:S})), we get that
\begin{align*}
&\xi_2=\frac{2(S_2-S_4)}{7}=\frac{6A_2-5A_4-A_0}{4A_0}>0,~~~\\
&\xi_3=2\Big(\frac{S_4}{35}-\frac{2S_2}{21}+\frac{1}{15}\Big)=\frac{A_4-2A_2+A_0}{4A_0}>0,\\
&\frac23\xi_1+\frac43\xi_2+\xi_3=\frac{3(A_0A_4-A_2^2 )}{2 A_0^2}>0.
\end{align*}
Thus, the coefficients $\psi_1,\psi_2,\psi_3$ can be uniquely determined.

Another important linear operator is the linearized operator $\mathcal{H}_{\nn}(Q)$ of $B(Q)-\alpha Q$ around $Q_0$, which is given by
\begin{eqnarray}\label{linearH_n}
\mathcal{H}_{\nn}(Q)=\mathcal{Q}^{-1}_{\nn}(Q)-\alpha Q,
\end{eqnarray}
plays an important role in next sections.

First, we introduce a two-dimensional subspace of $\mathbb{Q}$ as
\begin{align*}
\mathbb{Q}^{in}_{\nn}=\big\{\nn\nn^{\perp}+\nn^{\perp}\nn:~\nn^{\perp}\in\mathbb{V}_{\nn}\big\},
\end{align*}
where $\mathbb{V}_{\nn}:=\{\nn\in\mathbb{R}^3|\nn^{\perp}\cdot\nn=0\},$ and let $\mathbb{Q}^{out}_{\nn}$
be the orthogonal complement of $\mathbb{Q}^{in}_{\nn}$ in $\mathbb{Q}$.
The following proposition gives a characterization on the kernel space and non-negativity of $\CH_{\nn}$:
\begin{proposition}\label{prop:H_n}
$(i)$ $\mathcal{H}_{\nn} \mathbb{Q}^{in}_{\nn}=0,$ i.e. $\mathcal{H}_{\nn}(Q)\in \mathbb{Q}^{out}_{\nn}$.\\
$(ii)$ There exists a positive constant $c_0$ such that for any $Q\in \mathbb{Q}^{out}_{\nn}$,
\begin{align*}
\langle\mathcal{H}_{\nn}(Q), Q\rangle\geq c_0|Q|^2.
\end{align*}
\end{proposition}
\begin{proof}
$(i)$ From (\ref{def:inverseL_n}) and (\ref{linearH_n}), the linearized operator $\CH_{\nn}$ can be written as
\begin{eqnarray*}
\mathcal{H}_{\nn}(Q)=\psi_1(\nn\nn-\frac13\II)(\nn\nn:Q)+\psi_2(\nn\nn\cdot Q+Q\cdot\nn\nn
-\frac23\II\nn\nn:Q)+(\psi_3-\alpha)Q.
\end{eqnarray*}
where $\psi_1,\psi_2,\psi_3$ are given by (\ref{psi23}).
By (\ref{def:xi}) and definitions of $S_2$ and $S_4$, we have
\begin{align}\label{eqn:xi-alpha}
\xi_2+\xi_3=\frac{A_2-A_4}{A_0}=\frac{1}{\alpha}.
\end{align}
Together with (\ref{psi23}), we know
$\psi_2+\psi_3=(\xi_2+\xi_3)^{-1}=\alpha$.
Thus, we get
\begin{eqnarray}\label{eq:Hn}
\mathcal{H}_{\nn}(Q)=\psi_1(\nn\nn-\frac13\II)(\nn\nn:Q)
+\psi_2\big(-Q+\nn\nn\cdot Q+Q\cdot\nn\nn-\frac23\II(\nn\nn:Q)\big).
\end{eqnarray}
This yields the assertions in $(i)$ by observing
\begin{align}\label{Qout-decomp}
(\nn\nn-\frac13\II)(\nn\nn:Q)\in\mathbb{Q}^{out}_{\nn},~~~
-Q+\nn\nn\cdot Q+Q\cdot\nn\nn-\frac23\II(\nn\nn:Q)\in\mathbb{Q}^{out}_{\nn}.
\end{align}

$(ii)$ From the assertion in $(i)$ and (\ref{Qout-decomp}), we have
\begin{align*}
\mathcal{Q}_{\nn}^{-1}(Q)=\alpha Q+ \mathcal{H}_{\nn}(Q).
\end{align*}
Thus $\mathcal{Q}_{\nn}^{-1}(Q)\in\QO$ if and only if $Q\in\QO$.
Together with the fact
that $\mathcal{Q}_{\nn}$ is a bounded operator, we only need to prove that
\begin{align*}
\mathcal{H}_{\nn}(\mathcal{Q}_{\nn}(B)):\mathcal{Q}_{\nn}(B)
\geq c_0|B|^2,
\end{align*}
for some positive constant $c_0$ and any $B\in\QO$. From (\ref{def:L_n}), we have
\begin{align*}
 \langle B,\mathcal{Q}_{\nn}(B)\rangle=&\xi_1|\nn\nn:B|^2+2\xi_2|\nn\cdot B|^2
 +\xi_3|B|^2,
\end{align*}
and
\begin{align*}
\langle\mathcal{Q}_{\nn}(B),\mathcal{Q}_{\nn}(B)\rangle
=&\big(\frac23(\xi_1+2\xi_2)^2-2\xi_2^2+2\xi_2\xi_3\big)|\nn\nn:B|^2
  +(2\xi_2^2+4\xi_2\xi_3)|\nn\cdot B|^2+\xi_3^2|B|^2.
\end{align*}
Therefore we get
\begin{eqnarray*}
\mathcal{H}_{\nn}(\mathcal{Q}_{\nn}(B)):\mathcal{Q}_{\nn}(B)=
\big( B-\alpha\mathcal{Q}_{\nn}(B)\big):\mathcal{Q}_{\nn}(B)
=\beta_1|\nn\nn:B|^2+\beta_2|\nn\cdot B|^2+\beta_3|B|^2,
\end{eqnarray*}
where the coefficients are given by
\begin{align*}
&\beta_1=\xi_1-\alpha\big(\frac23(\xi_1+2\xi_2)^2-2\xi_2^2+2\xi_2\xi_3\big),\\
&\beta_2=2\xi_2-\alpha(2\xi_2^2+4\xi_2\xi_3),\quad \beta_3=\xi_3-\alpha\xi_3^2.
\end{align*}
From (\ref{eqn:xi-alpha}), we have
\begin{align*}
\beta_2+2\beta_3=2(\xi_2+\xi_3)\big(1-\alpha(\xi_2+\xi_3)\big)=0,
\end{align*}
which implies
\begin{align*}
\mathcal{H}_{\nn}(\mathcal{Q}_{\nn}(B)):\mathcal{Q}_{\nn}(B)
=\beta_1|\nn\nn:B|^2+\beta_3(|B|^2-2|B\cdot\nn|^2)
=(\beta_1-2\beta_3)|\nn\nn:B|^2+\beta_3|B|^2,
\end{align*}
for $B\in\QO$. If $\beta_1>2\beta_3$, the assertion is apparently true.
If $\beta_1<2\beta_3$, it is direct to check that for traceless matrix $B$
\begin{align*}
|B|^2\ge \frac32|B:\nn\nn|^2.
\end{align*}
Therefore
\begin{align*}
&\mathcal{H}_{\nn}(\mathcal{Q}_{\nn}(B)):\mathcal{Q}_{\nn}(B)
=(\beta_1-2\beta_3)|\nn\nn:B|^2
+\beta_3|B|^2\\
&\ge\frac{2}{3}(\beta_1-2\beta_3)|B|^2+\beta_3|B|^2
=\frac13(2\beta_1-\beta_3)|B|^2.
\end{align*}
Some further tedious calculations give that
\begin{align*}
\beta_1-\frac{\beta_3}{2}=&\xi_1-\frac12\xi_3-\alpha\big(\frac23(\xi_1+2\xi_2)^2-2\xi_2^2+2\xi_2\xi_3-\frac12\xi_3^2\big)\\
=&-\frac{8A_0A_2+9A_2^2-17A_0A_4}{4 A_0^2}-\frac{-27 A_2^4+ 54 A_0 A_2^2 A_4+ A_0^2(16 A_2^2-32A_2A_4-11A_4^2)}{8 A_0^3(A_2-A_4)}\\
=&\frac{9 (-3A_2^4-2A_0A_2^2 (A_2-4A_4)+A_0^2 (2A_2-5A_4)A_4)}{8A_0^3 (A_2-A_4)}\\
=&\frac{9(A_0A_4-A_2^2)(3A_2^2+2A_0A_2-5A_0A_4)}{8A_0^3(A_2-A_4)}>0.
\end{align*}
This concludes the proof.
\end{proof}

We denote by $\mathcal{P}^{in}$ the projection operator from $\mathbb{Q}$ to
$\mathbb{Q}^{in}_{\nn}$ and by $\mathcal{P}^{out}$ the projection operator from $\mathbb{Q}$ to
$\mathbb{Q}^{out}_{\nn}$. By direct computation we have
\[
 |Q-(\nn\nn^{\perp}+\nn^{\perp}\nn)|^2
 =|Q|^2-2|Q\cdot\nn|^2+2|Q:\nn\nn|^2+|\nn^{\perp}-(\II-\nn\nn)\cdot Q\cdot\nn|^2.
\]
Therefore, there holds
\begin{align}
 \mathcal{P}^{in}(Q)=&\nn[(\II-\nn\nn)\cdot Q\cdot\nn]+[(\II-\nn\nn)\cdot Q\cdot\nn]\nn\nonumber\\
 =&(\nn\nn\cdot Q+Q\cdot\nn\nn)-2(Q:\nn\nn)\nn\nn,\label{projection_in1}\\
 |\mathcal{P}^{in}(Q)|^2=&2|Q\cdot\nn|^2-2|Q:\nn\nn|^2,\label{projection_in2}
\end{align}
and
\begin{align}
 \mathcal{P}^{out}(Q)=Q- \mathcal{P}^{in}(Q)
 =Q-(\nn\nn\cdot Q+Q\cdot\nn\nn)+2(Q:\nn\nn)\nn\nn.\label{projection_out1}
\end{align}

Another two linear operators will be frequently used in the later are
$\CJ_{\bar{Q}}:\mathbb{R}^{3\times3}\mapsto\mathbb{Q}$ and
$\mathcal{U}_{\bar{Q}}:\mathbb{Q}\mapsto\mathbb{Q}$, which are defined as
\begin{align*}
&\mathcal{J}_{\bar{Q}}(A):=\frac12\big(\CM_{\bar{Q}}(A)+\CM^T_{\bar{Q}}(A)\big)=\frac16(A+A^T)+\frac12(A^T\cdot \bar{Q}+\bar{Q}\cdot A)-A:M^{(4)}_{\bar{Q}},\quad \text{for }A\in\mathbb{R}^{3\times3};\\
&\mathcal{U}_{\bar{Q}}(B):=M^{(6)}_{\bar{Q}}:B-(\bar{Q}:B)M^{(4)}_{\bar{Q}},\quad \text{for }B\in\mathbb{Q}.
\end{align*}
When $\bar{Q}=S_2(\nn\nn-\frac13\II)$, and $A\in \QO$, then by (\ref{projection_out1}) we infer that
\begin{align*}
\bar{Q}\cdot A=S_2(\nn\nn-\frac13\II)\cdot\mathcal{P}^{out}A= S_2(A:\nn\nn)\nn\nn-\frac13S_2\mathcal{P}^{out}A,
\end{align*}
which is symmetric. Thus, $\bar{Q}\cdot A=A^T\cdot\bar{Q}$, and then we have
\begin{align}\label{eq:M-J}
\CM_{\bar{Q}}(A)=\CM^T_{\bar{Q}}(A)=\CJ_{\bar{Q}}(A).
\end{align}
A direct consequence of (\ref{eq:M-J}) and Proposition \ref{prop:H_n} is that, for $\bar{Q}=S_2(\nn\nn-\frac13\II)$,
\begin{align}\label{eq:M-J1}
\CM_{\bar{Q}}(\CH_{\bar{Q}}(Q))=\CJ_{\bar{Q}}(\CH_{\bar{Q}}(Q)),\quad \text{for any} ~Q\in\mathbb{Q}.
\end{align}
We denote $\CJ_{\bar{Q}}$ by $\CJ_\nn$ for simplicity when $\bar{Q}=S_2(\nn\nn-\frac13\II)$.
It should be noticed that $\CJ_\nn$ is not self-adjoint operator on $\mathbb{R}^{3\times3}$ but is self-adjoint on
the space $\mathbb{Q}$.

Direct computation gives that
\begin{align*}
&\CQ_{\nn}(\nn\nn^{\perp}+\nn^{\perp}\nn)=\Big(\frac{2(S_2-S_4)}{7}
+2(\frac{S_4}{35}-\frac{2S_2}{21}+\frac{1}{15})\Big)(\nn\nn^{\perp}+\nn^{\perp}\nn)\in \mathbb{Q}^{in}_{\nn},\\
&\CJ_{\nn}(\nn\nn^{\perp}+\nn^{\perp}\nn)=\Big(\frac13+\frac16S_2-\frac{2(S_2-S_4)}{7}
-2(\frac{S_4}{35}-\frac{2S_2}{21}+\frac{1}{15})\Big)(\nn\nn^{\perp}+\nn^{\perp}\nn)\in \mathbb{Q}^{in}_{\nn},
\end{align*}
which imply
\begin{align*}
\mathcal{Q}_\nn\mathbb{Q}^{in}_{\nn}\subseteq\mathbb{Q}^{in}_{\nn},\qquad\mathcal{J}_\nn\mathbb{Q}^{in}_{\nn}\subseteq\mathbb{Q}^{in}_{\nn}.
\end{align*}
As $\mathcal{Q}_\nn$ and $\mathcal{J}_\nn$ are self-adjoint on $\mathbb{Q}$, we also have
\begin{align*}
\mathcal{Q}_\nn\mathbb{Q}^{out}_{\nn}\subseteq\mathbb{Q}^{out}_{\nn},\qquad\mathcal{J}_\nn\mathbb{Q}^{out}_{\nn}\subseteq\mathbb{Q}^{out}_{\nn}.
\end{align*}
In summary, we get
\begin{align}\label{comm:LJ}
[\mathcal{Q}_\nn,~\mathcal{P}^{in}]=[\mathcal{Q}_\nn,~\mathcal{P}^{out}]=0,\qquad [\CJ_\nn,~\mathcal{P}^{in}]=[\CJ_\nn,~\mathcal{P}^{out}]=0.
\end{align}

\section{Rigorous derivation from the $Q$-tensor theory to the Ericksen-Leslie theory}
In this section, by making the Hilbert expansion for the solution
of the molecule-based $Q$-tensor systems (\ref{eqnewQ1})-(\ref{eqnewQ3}), we present a
rigorous derivation from the molecule-based $Q$-tensor theory to the Ericksen-Leslie theory.

\subsection{The Hilbert expansion}
Let $(Q^{\ve},\vv^{\ve})$ be a solution of the system (\ref{eqepsiQ1})-(\ref{eqepsiQ3}).
We perform the following so-called Hilbert expansion:
\begin{align}
&Q^{\ve}=\sum^3_{k=0}\ve^kQ_k+\ve^3 Q_{R}\eqdefa\widetilde{Q}+\ve^3 Q_{R},\label{Qvar}\\
&\vv^{\ve}=\sum^2_{k=0}\ve^k\vv_k+\ve^3 \vv_{R}\eqdefa\tilde{\vv}+\ve^3 \vv_{R},\label{vvar}
\end{align}
where $Q_i(0\le i\le 3),\vv_j(0\le j\le2)$ are independent of $\ve$ and will be determined in what follows. $(Q_R,\vv_R)$
are called the remainder term which depend upon $\ve$.

Substituting the above expansion to (\ref{eqepsiQ1})-(\ref{eqepsiQ3}), and expanding all the terms with respect to $\ve$, we can get several systems of equations to solve $(Q_i, \vv_i)(0\le i\le 2)$ and $Q_3$ by collecting all the terms of the same order with respect to $\ve$. In \cite{WZZ1} and \cite{WZZ4}, the expanding can be performed directly as it involves only polynomials of variables. In contrast, the dependence of $B$ and $\CM_B$ on $Q$ is much more complicated here.

First, we make the following formal expansion for $Z_{Q^\ve}$ and $B_{Q^\ve}$:
\begin{align}
{B}_{Q^\ve}=B_0+\ve B_1+\ve^2 B_2+\ve^3 B_3+\ve^3 B_R+\ve^4 \mathfrak{R}_B,\\
Z_{Q^\ve}=Z_0+\ve Z_1+\ve^2 Z_2+\ve^3 Z_3+\ve^3 Z_R+\ve^4 \mathfrak{R}_Z.
\end{align}
Here $B_i, Z_i(0\le i\le3)$ depend on $Q_i(0\le i\le 3)$ only, and $B_R, Z_R$ depends on $Q_R$ and $Q_i(0\le i\le 3)$ .
Moreover, $B_i, Z_i(0\le i\le3)$ are independent of $\ve$ and $B_R, Z_R$ are the linear funtions of $Q_R$. All the terms with higher order of $\ve$
are put in $\ve^4 \mathfrak{R}_B$ and $\ve^4\mathfrak{R}_Z$.
To perform the Hilbert expansion, we have to write $B_i, Z_i$ and $B_R, Z_R$ in terms of $Q_i, Q_R$ explicitly.

By viewing $Z_{Q^\ve}$ as a function of $B^\ve$, we have:
\begin{align}\label{widetilde_Z}
Z_{{Q}^\ve}=&\int_{\mathbb{S}^2}\exp\Big(\mm\mm:
(\sum^3_{k=0}\ve^kB_k+\ve^3 B_R+\ve^4 \mathfrak{R}_{B})\Big)\ud\mm\nonumber\\
=&Z_{Q_0}\Big(1+\ve Q_0:B_1+\ve^2(Q_0:B_2+\widehat{Z}_1)
 +\ve^3(Q_0:B_3+2\widehat{Z}_2\nonumber\\
&\qquad +\widehat{Z}_3)+\ve^3Q_0:B_R\Big)+\ve^4\mathfrak{R}_Z,
\end{align}
where
\begin{align*}
 \widehat{Z}_1=&\frac{1}{2Z_{Q_0}}\int_{\mathbb{S}^2}(\mm\mm:B_1)^2\exp(\mm\mm:B_0)\ud\mm,\\
 \widehat{Z}_2=&\frac{1}{2Z_{Q_0}}\int_{\mathbb{S}^2}(\mm\mm:B_1)(\mm\mm:B_2)\exp(\mm\mm:B_0)\ud\mm,\\
 \widehat{Z}_3=&\frac{1}{6Z_{Q_0}}\int_{\mathbb{S}^2}(\mm\mm:B_1)^3\exp(\mm\mm:B_0)\ud\mm.
\end{align*}
By the expression of $Q^\ve$, we have
\begin{align}\label{widetilde_Q}
{Q}^\ve=&\frac{1}{ Z_{{Q}^\ve}}\int_{\mathbb{S}^2}(\mm\mm-\frac13\II)\exp\Big(\mm\mm:
(\sum^3_{k=0}\ve^kB_k+\ve^3 B_R+\ve^4 \mathfrak{R}_{B})\Big)\ud\mm\nonumber\\
=&\Big(Q_0+\ve(M^{(4)}_{Q_0}-\frac13\II Q_0):B_1+\ve^2\big[(M^{(4)}_{Q_0}
  -\frac13\II Q_0):B_2+\widehat{Q}_1\big]\nonumber\\
&\quad+\ve^3\big((M^{(4)}_{Q_0}-\frac13\II Q_0):B_3+2\widehat{Q}_2+\widehat{Q}_3\big)
+\ve^3\big((M^{(4)}_{Q_0}-\frac13\II Q_0):B_R\big)+O(\ve^4)\Big)\nonumber\\
 &\cdot\Big({1+\ve Q_0:B_1+\ve^2(Q_0:B_2+\widehat{Z}_1)
  +\ve^3(Q_0:B_3+2\widehat{Z}_2+\widehat{Z}_3)+\ve^3Q_0:B_R+O(\ve^4)}\Big)^{-1}\nonumber\\
=&Q_0+\ve\Big((M^{(4)}_{Q_0}-\frac13\II Q_0):B_1-(Q_0:B_1)Q_0\Big)
  +\ve^2\Big((M^{(4)}_{Q_0}-\frac13\II Q_0):B_2+\widehat{Q}_1\nonumber\\
 &-Q_0(Q_0:B_2+\widehat{Z}_1)-\{(M^{(4)}_{Q_0}-\frac13\II Q_0):B_1\}\big(Q_0:B_1\big)+Q_0(Q_0:B_1)^2\Big)\nonumber\\
 &+\ve^3\Big((M^{(4)}_{Q_0}-\frac13\II Q_0):B_3+2\widehat{Q}_2+\widehat{Q}_3-Q_0(Q_0:B_3+2\widehat{Z}_2+\widehat{Z}_3)
 \nonumber\\
 &-(Q_0:B_2+\widehat{Z}_1)(M^{(4)}_{Q_0}-\frac13\II Q_0):B_1-\{(M^{(4)}_{Q_0}-\frac13\II Q_0):B_2+\widehat{Q}_1\}(Q_0:B_1)
 \nonumber\\
 &+2Q_0(Q_0:B_1)(Q_0:B_2+\widehat{Z}_1)+(Q_0:B_1)^2(M^{(4)}_{Q_0}-\frac13\II Q_0):B_1-Q_0(Q_0:B_1)^3\Big)\nonumber\\
 &+\ve^3\Big((M^{(4)}_{Q_0}-\frac13\II Q_0):B_R-Q_0(Q_0:B_R)\Big)+O(\ve^4),
\end{align}
where
\begin{align*}
\widehat{Q}_1=&\frac1{2Z_{Q_0}}\int_{\mathbb{S}^2}(\mm\mm-\frac13\II)(\mm\mm:B_1)^2\exp(\mm\mm:B_0)\ud\mm,\\
 \widehat{Q}_2=&\frac1{2Z_{Q_0}}\int_{\mathbb{S}^2}(\mm\mm-\frac13\II)(\mm\mm:B_1)(\mm\mm:B_2)\exp(\mm\mm:B_0)\ud\mm,\\
\widehat{Q}_3=&\frac1{6Z_{Q_0}}\int_{\mathbb{S}^2}(\mm\mm-\frac13\II)(\mm\mm:B_1)^3\exp(\mm\mm:B_0)\ud\mm.
\end{align*}
Noting the definition of linear operator $\mathcal{Q}_{Q_0}$ and
from (\ref{widetilde_Q}) we can duduce that
\begin{align}
 Q_1=&~\mathcal{Q}_{Q_0}(B_1),\label{Q1}\\
 Q_2=&~\mathcal{Q}_{Q_0}(B_2)-(Q_0:B_1)\mathcal{Q}_{Q_0}(B_1)+\widehat{Q}_1-\widehat{Z}_1Q_0,\label{Q2}\\
 Q_3=&~\mathcal{Q}_{Q_0}(B_3)-(Q_0:B_1)\mathcal{Q}_{Q_0}(B_2)
    +\big((Q_0:B_1)^2-Q_0:B_2-\widehat{Z}_1\big)\mathcal{Q}_{Q_0}(B_1)\nonumber\\
    &~+(Q_0:B_1)(\widehat{Z}_1Q_0-\widehat{Q}_1)-(2\widehat{Z}_2+\widehat{Z}_3)Q_0+2\widehat{Q}_2+\widehat{Q}_3, \label{Q3}\\
Q_R=&~\mathcal{Q}_{Q_0}(B_R).
\end{align}
Thanks to the invertibility of $\mathcal{Q}_{Q_0}$, we know that $B_i$ can be explicitly given by $Q_j(0\le j\le i)$, and
$B_R$ is linearly depend on $Q_R$.

Similarly, we next make the expansion for $M^{(4)}_{{Q}^\ve}$:
\begin{align}
 M^{(4)}_{{Q}^\ve}= M^{(4)}_{0}+\ve M^{(4)}_{1}+\ve^2 M^{(4)}_{2}+\ve^3 M^{(4)}_{3}+\ve^3 M^{(4)}_{R}+\ve^4\mathfrak{R}_{M^{(4)}}.
\end{align}
Then we have
\begin{align}\label{M4tildeQ}
 M^{(4)}_{{Q}^\ve}=~&\frac{1}{ Z_{{Q}^\ve}}\int_{\mathbb{S}^2}
   \mm\mm\mm\mm\exp\Big(\mm\mm:(\sum^3_{k=0}\ve^kB_k+\ve^3 B_R+\ve^4 \mathfrak{R}_{B})\Big)\ud\mm\nonumber\\
=~&\Big(M^{(4)}_{Q_0}+\ve M^{(6)}_{Q_0}:B_1
        +\ve^2(M^{(6)}_{Q_0}:B_2+\widehat{M}_1^{(4)})
        +\ve^3(M^{(6)}_{Q_0}:B_3+2\widehat{M}_2^{(4)}+\widehat{M}_3^{(4)})\nonumber\\
&+\ve^3(M^{(6)}_{Q_0}:B_R)+O(\ve^4)\Big)\nonumber\\
&\cdot\!\Big({1+\ve Q_0:B_1+\ve^2(Q_0:B_2+\widehat{Z}_1)
   +\ve^3(Q_0:B_3+2\widehat{Z}_2+\widehat{Z}_3)+\ve^3Q_0:B_R+O(\ve^4)}\Big)^{-1}\nonumber\\
=~&M^{(4)}_{Q_0}+\ve\Big(M^{(6)}_{Q_0}:B_1-(Q_0:B_1)M^{(4)}_{Q_0}\Big)\nonumber\\
~&+\ve^2\Big(M^{(6)}_{Q_0}:\big(B_2-B_1(Q_0:B_1)\big)-M^{(4)}_{Q_0}\big(Q_0:B_2-(Q_0:B_1)^2+\widehat{Z}_1\big)
    +\widehat{M}_1^{(4)}\Big)\nonumber\\
~&+\ve^3\Big(M^{(6)}_{Q_0}:\big[B_3+B_2(Q_0:B_1)+B_1\big((Q_0:B_1)^2-(Q_0:B_2+\widehat{Z}_1)\big)\big]
\nonumber\\
&-M^{(4)}_{Q_0}\big[Q_0:B_3+2\widehat{Z}_2+\widehat{Z}_3-2(Q_0:B_1)(Q_0:B_2+\widehat{Z}_1)+(Q_0:B_1)^3\big]\nonumber\\
    &+2\widehat{M}_2^{(4)}+\widehat{M}_3^{(4)}-\widehat{M}_1(Q_0:B_1)\Big) +\ve^3\Big(M^{(6)}_{Q_0}:B_R
   -M^{(4)}_{Q_0}(Q_0:B_R)\Big)\nonumber\\
   & +O(\ve^4).
\end{align}
Here $\widehat{Z}_i$ are defined after the expansion of $Z_{Q^\ve}$, and $\widehat{M}^{(4)}_i$ are defined as
\begin{align*}
\widehat{M}_1^{(4)}=&\frac1{2Z_{Q_0}}\int_{\mathbb{S}^2}\mm\mm\mm\mm(\mm\mm:B_1)^2\exp(\mm\mm:B_0)\ud\mm,\\
\widehat{M}_2^{(4)}=&\frac1{2Z_{Q_0}}\int_{\mathbb{S}^2}\mm\mm\mm\mm(\mm\mm:B_1)(\mm\mm:B_2)\exp(\mm\mm:B_0)\ud\mm,\\
\widehat{M}_3^{(4)}=&\frac1{6Z_{Q_0}}\int_{\mathbb{S}^2}\mm\mm\mm\mm(\mm\mm:B_1)^3\exp(\mm\mm:B_0)\ud\mm.
\end{align*}
Noting the definition of linear operator $\mathcal{U}_{Q_0}$ and from (\ref{M4tildeQ}) we get
\begin{align*}
 M_1^{(4)}=&\mathcal{U}_{Q_0}(B_1),\\
 M_2^{(4)}=&\mathcal{U}_{Q_0}(B_2)-(Q_0:B_1)\mathcal{U}_{Q_0}(B_1)+\widehat{M}_1^{(4)}-\widehat{Z}_1M^{(4)}_{Q_0},\\
 M_3^{(4)}=&\mathcal{U}_{Q_0}(B_3)-(Q_0:B_1)\mathcal{U}_{Q_0}(B_2)
    +\big((Q_0:B_1)^2-Q_0:B_2-\widehat{Z}_1\big)\mathcal{U}_{Q_0}(B_1)\nonumber\\
    &+(Q_0:B_1)(\widehat{Z}_1M^{(4)}_{Q_0}-\widehat{M}_1^{(4)})-(2\widehat{Z}_2+\widehat{Z}_3)M^{(4)}_{Q_0}
    +2\widehat{M}_2^{(4)}+\widehat{M}_3^{(4)},\\
 M_R^{(4)}=&~\mathcal{U}_{Q_0}(B_R).
\end{align*}

Now, we can write down the expansion of the original system (\ref{eqepsiQ1})-(\ref{eqepsiQ3})
and collect the terms with same order of $\ve$. Specifically, we have\\
$\bullet$~{\bf The $O(\ve^{-1})$ system}\\
\begin{eqnarray}\label{expan:Q0-1}
 \mathcal{M}_{Q_0}(B_0-\alpha Q_0)=0,
\end{eqnarray}
$\bullet$~{\bf The zero-order term in $\ve$}\\
\begin{align}
 \frac{\partial Q_0}{\partial t}+\vv_0\cdot\nabla Q_0=&-4\CJ_{Q_0}\Big(\CH_{{Q_0}}(Q_1)+\CL (Q_0)\Big)
     +2\CJ_{Q_0}(\kappa^T_0),\label{expan-newQ-Q0}\\
\frac{\partial \vv_0}{\partial t}+\vv_0\cdot\nabla \vv_0=&
        -\nabla p_0 +\frac{\gamma}{Re}\Delta\vv_0+\frac{1-\gamma}{2Re}
       \nabla\cdot(\DD_0:M^{(4)}_{Q_0})\nonumber\\
     &+\frac{1-\gamma}{Re}\nabla\cdot\Big(
     2\CJ_{{Q_0}}(\CH_{{Q_0}}(Q_1))
            +2\mathcal{M}_{Q_0}(\CL (Q_0))
            +\sigma^d(Q_0,Q_0)\Big),\label{expan-newQ-v0}\\
 \nabla\cdot \vv_0=&~0.\label{expan-newQ-0}
\end{align}
$\bullet$~{\bf The first-order term in $\ve$}\\
\begin{align}
 \frac{\partial Q_1}{\partial t}+\vv_0\cdot\nabla Q_1=&-\vv_1\cdot\nabla Q_0
     -4\CJ_{{Q_0}}\Big(\CH_{{Q_0}}(Q_2)+\CL (Q_1)\Big)+2\CJ_{Q_0}(\kappa^T_1)+\FF_1,\label{expan-newQ-Q1}\\
\frac{\partial \vv_1}{\partial t}+\vv_0\cdot\nabla \vv_1=&
   -\vv_1\cdot\nabla \vv_0
   -\nabla p_1 +\frac{\gamma}{Re}\Delta\vv_1+\frac{1-\gamma}{2Re}
       \nabla\cdot \Big(\DD_0:M_1^{(4)}+\DD_1:M^{(4)}_{Q_0}\Big)\nonumber\\
  &+\frac{1-\gamma}{Re}\nabla\cdot\Big(2\CM_{{Q_0}}(\CH_{{Q_0}}(Q_2))+2\CM_{Q_0}(\CL (Q_1))-\GG_1\nonumber\\
  & +\sigma^d(Q_0,Q_1)+\sigma^d(Q_1,Q_0)\Big),\label{expan-newQ-v1}\\
 \nabla\cdot \vv_1=&~0.\label{expan-newQ-1}
\end{align}
$\bullet$~{\bf The second-order term in $\ve$}\\
\begin{align}
 \frac{\partial Q_2}{\partial t}+\vv_0\cdot\nabla Q_2=&-\vv_1\cdot\nabla Q_1-\vv_2\cdot\nabla Q_0
       -4\CJ_{{Q_0}}\Big(\CH_{{Q_0}}(Q_3)+\CL (Q_2)\Big)\nonumber\\
      &+\CJ_{Q_0}(\kappa^T_2)+\FF_2,\label{expan-newQ-Q2}\\
\frac{\partial \vv_2}{\partial t}+\vv_0\cdot\nabla \vv_2=&
  -\vv_1\cdot\nabla \vv_1-\vv_2\cdot\nabla \vv_0
 -\nabla p_2 +\frac{\gamma}{Re}\Delta\vv_2\nonumber\\
 &+\frac{1-\gamma}{2Re}
       \nabla\cdot\Big(\DD_0:M_2^{(4)}+\DD_1:M^{(4)}_1+\DD_2:M^{(4)}_{Q_0}\Big)\nonumber\\
  &+\frac{1-\gamma}{Re}\nabla\cdot\Big(2\CM_{{Q_0}}(\CH_{{Q_0}}(Q_3))+2\CM_{Q_0}(\CL (Q_2))-\GG_2\nonumber\\
 &+\sigma^d(Q_0,Q_2)
 +\sigma^d(Q_1,Q_1)+\sigma^d(Q_2,Q_0)\Big),\label{expan-newQ-v2}\\
 \nabla\cdot \vv_2=&~0.\label{expan-newQ-2}
\end{align}
Here, $\FF_1,\FF_2,\GG_1,\GG_2$ are defined as following:
\begin{align*}
\FF_1=&\overline{\FF}_1+\widetilde{\FF}_1,\\
\overline{\FF}_1=&4\alpha\Big(Q_1\cdot Q_1-(Q_1+(Q_0:B_1)Q_0):M_1^{(4)}
-Q_0:(\widehat{M}_1^{(4)}-\widehat{Z}_1M^{(4)}_{Q_0})\Big),\\
\widetilde{\FF}_1=&Q_1\cdot(-2\CL (Q_0)+\kappa_0^T)
+(-2\CL (Q_0)+\kappa_0)\cdot Q_1
-2(-2\CL (Q_0)+\DD_0):M_1^{(4)},\\
\GG_1=&\frac12\overline{\FF}_1-2(Q_1\cdot\CL (Q_0)-\CL (Q_0):M_1^{(4)}),
\end{align*}
and
\begin{align*}
\FF_2=&Q_2\cdot(2\alpha Q_1-2\CL(Q_0)+\kappa_0^T)
+(2\alpha Q_1-2\CL(Q_0)+\kappa_0)\cdot Q_2\\
&+Q_1\cdot(2\alpha Q_2-2\CL (Q_1)+\kappa_1^T)
+(2\alpha Q_2-2\CL (Q_1)+\kappa_1)\cdot Q_1\\
&-2(2\alpha Q_1-2\CL(Q_0)+\DD_0):M_2^{(4)}
-2(2\alpha Q_2-2\CL(Q_1)+\DD_1):M_1^{(4)}\\
&-4\alpha Q_0:\Big(-(Q_0:B_1)M_2^{(4)}
    +\big((Q_0:B_1)^2-Q_0:B_2-\widehat{Z}_1\big)M_1^{(4)}\nonumber\\
&+(Q_0:B_1)(\widehat{Z}_1M^{(4)}_{Q_0}-\widehat{M}_1^{(4)})-(2\widehat{Z}_2+\widehat{Z}_3)M^{(4)}_{Q_0}
    +2\widehat{M}_2^{(4)}+\widehat{M}_3^{(4)}\Big),\\
\GG_2=&Q_2\cdot(2\alpha Q_1-2\CL(Q_0))+Q_1\cdot(2\alpha Q_2-2\CL(Q_1))\\
&-(2\alpha Q_1-2\CL(Q_0)):M_2^{(4)}
-(2\alpha Q_2-2\CL(Q_1)):M_1^{(4)}\\
&-2\alpha Q_0:\Big(-(Q_0:B_1)M_2^{(4)}
    +\big((Q_0:B_1)^2-Q_0:B_2-\widehat{Z}_1\big)M_1^{(4)}\nonumber\\
&+(Q_0:B_1)(\widehat{Z}_1M^{(4)}_{Q_0}-\widehat{M}_1^{(4)})-(2\widehat{Z}_2+\widehat{Z}_3)M^{(4)}_{Q_0}
    +2\widehat{M}_2^{(4)}+\widehat{M}_3^{(4)}\Big).
\end{align*}

The equation of $O(\ve^{-1})$ (\ref{expan:Q0-1}) is equivalent to $B_0-\alpha Q_0=0$. Thanks to Proposition \ref{prop:critical},
$Q_0$ takes the form
\begin{align}\label{eq:Q0}
Q_0(t,\xx)=S_2\big(\nn(t,\xx)\nn(t,\xx)-\frac{1}{3}\II\big),
\end{align}
for some $\nn(t,\xx)\in \BS$.

The evolution of $\nn(t,\xx)$ is determined by the $O(1)$ system (\ref{expan-newQ-Q0})-(\ref{expan-newQ-0}).
At first glance, this system is not closed since it involves $Q_1$ which is unknown.
However, if we project (\ref{expan-newQ-Q0}) into the subspace $\QI=\mathrm{Ker}~\CH_\nn$,
then $Q_1$ is vanished in (\ref{expan-newQ-Q0}) by Proposition \ref{prop:H_n}. In addition, if we
project (\ref{expan-newQ-Q0}) into the subspace $\QO=(\mathrm{Ker}~\CH_\nn)^\bot$, then
we can solve $\CH_\nn (Q_1)$ in terms of $(Q_0,\vv_0)$. Thus $Q_1$ can also be eliminated in (\ref{expan-newQ-v0}).
Actually, the following proposition shows that the system (\ref{expan-newQ-Q0})-(\ref{expan-newQ-0})
implies $(\nn, \vv_0)$ satisfies the Ericksen-Leslie system with coefficients depending on the
molecule parameters. One can see the detailed proof in \cite{WZZ3}.
\begin{proposition}\label{prop:EL}
If $(\vv_0,Q_0)$ is a strong solution of the system (\ref{expan-newQ-Q0})--(\ref{expan-newQ-0}), then $(\nn, \vv_0)$ is necessary a
solution of the Ericksen-Leslie system (\ref{eq:EL-v})--(\ref{eq:EL-n}),
where the coefficients are given by (\ref{leslie1-intro})-(\ref{OF-LD-relation-intro}).
\end{proposition}
In the next subsections, we will show how to solve $Q_i(1\le i\le 3)$ and $\vv_j(1\le j\le 2)$ from
(\ref{expan-newQ-Q1}-\ref{expan-newQ-2}). The whole procedure is very similar to the one used in \cite{WZZ1} and \cite{WZZ4}.

\subsection{Existence of the Hilbert expansion}
Assume that $(\vv_0,\nn)$ is a solution of the systems (\ref{eq:EL-v})-(\ref{eq:EL-n}) on $[0,T]$ such that
\[
 \vv_0\in C([0,T];H^k),~~\nabla \nn\in C([0,T];H^k)
\]
for $k\geq20$. Since $Q_0=S_2(\nn(t,\xx)\nn(t,\xx)-\frac13\II)$, we have
$Q_0\in C([0,T],H^{k+1})$.

Let $Q_1=Q^{\top}_1+Q^{\perp}_1$ with $Q^{\top}_1\in\mathbb{Q}^{in}_{\nn}$
and $Q^{\perp}_1\in\mathbb{Q}^{out}_{\nn}$. Notice that we can solve
$Q^{\perp}_1$ by the equation (\ref{expan-newQ-Q0}) and have $Q^{\perp}_1\in C([0,T];H^{k-1})$.
In order to solve $(\vv_1,Q^{\top}_1)$, we need to derive a closed system for $(\vv_1,Q^{\top}_1)$
from (\ref{expan-newQ-Q1})-(\ref{expan-newQ-1}). We will also show that this system is {\it linear}
and {\it have a closed energy estimate}, which implies the solution $(\vv_1,Q^{\top}_1)$ will not
blow up in $[0, T]$.

In what follows, we denote by $L(Q^{\top}_1,\vv_1)$ the terms which only depend on $(Q^{\top}_1, \vv_1)$
(not their derivatives) linearly with the coefficients belonging to $C([0,T];H^{k-1})$. We also use
 $R\in C([0,T];H^{k-3})$ to denote the terms depending only on $\nn,\vv_0$ and $\QQ_1^\bot$.

\begin{lemma}\label{lem:s5-1}
It holds that
\begin{eqnarray*}
&~&\CP^{out}\big(\frac{\partial Q_1}{\partial t}+\vv_0\cdot\nabla Q_1\big)=L(Q^{\top}_1)+R,\\
&~&\CP^{in}\big(\frac{\partial Q_1}{\partial t}+\vv_0\cdot\nabla Q_1\big)
  =\frac{\partial Q^{\top}_1}{\partial t}+\vv_0\cdot\nabla Q^{\top}_1+L(Q^{\top}_1)+R.
\end{eqnarray*}
\end{lemma}
\begin{proof}
The proof can be found in \cite{WZZ3}.
\end{proof}

For any $Q\in \mathbb{Q}$, we set
\begin{align*}
\overline{\FF}_1(Q)=4\alpha\Big(Q^2-\Big(Q+Q_0(\CQ^{-1}_{\nn}(Q):Q_0)\Big):\CU_{\nn}\big(\CQ^{-1}_{\nn}(Q)\big)
-Q_0:\big(\widehat{M}_1(Q)-\widehat{Z}_1(Q)M^{(4)}_{Q_0}\big)\Big),
\end{align*}
where $\widehat{M}_1(Q)$ and $\widehat{Z}_1(Q)$ are nonlinear functions with respect to $Q$,
\begin{eqnarray*}
&~&\widehat{M}_1(Q)=\frac{1}{2Z_{Q_0}}\int_{\mathbb{S}^2}\mm\mm\mm\mm\big(\mm\mm:\CQ^{-1}_{\nn}(Q)\big)^2\exp(\mm\mm:B_0)\ud\mm,\\
&~&\widehat{Z}_1(Q)=\frac{1}{2Z_{Q_0}}\int_{\mathbb{S}^2}\big(\mm\mm:\CQ^{-1}_{\nn}(Q)\big)^2\exp(\mm\mm:B_0)\ud\mm.
\end{eqnarray*}
Therefore, note that $Q_1=Q^{\top}_1+Q^{\perp}_1$, we have
\begin{eqnarray*}
 \overline{\FF}_1=\overline{\FF}_1(Q_1)=\overline{\FF}_1(Q^{\top}_1)+L(Q^{\top}_1),
\end{eqnarray*}
where the definition of $L(\cdot)$ is as the above. The next lemma tells us that when we take the projection
$\CP^{in}$ on $\overline{\FF}_1$, the terms which are nonlinear with respect to $Q^{\top}_1$ will vanish.
\begin{lemma}\label{lem:s5-2}
$\overline{\FF}_1(Q^{\top}_1)\in \mathbb{Q}^{out}_{\nn},$ that is, $\CP^{in}\overline{\FF}_1=L(Q^{\top}_1)$.
 \end{lemma}
\begin{proof} Let $Q_1^\top=\nn\tilde\nn+\tilde\nn\nn$ where $\tilde\nn\bot\nn$. It suffices to prove that
$\overline{\FF}_1(Q^{\top}_1):(\nn\pp+\pp\nn)=0$ for any $\pp\bot\nn$.

Due to the definition of $\CQ^{-1}_{\nn}$ we know $\CQ^{-1}_{\nn}(Q_1^\top)=(\psi_2+\psi_3)Q_1^\top$,
where $\psi_2, \psi_3$ are coefficients defined in (\ref{psi23}). Thus we have
\begin{align*}
Q_0:\CQ^{-1}_{\nn}(Q_1^\top)=S_2(\psi_2+\psi_3)(\nn\nn-\frac13\II):(\nn\tilde\nn+\tilde\nn\nn)=0.
\end{align*}
Direct calculation yields that
\begin{align*}
&\Big(Q_1^\top:\CU_{\nn}(\CQ^{-1}_{\nn}(Q_1^\top))\Big):(\nn\pp+\pp\nn)\\
~~&=\frac{8(\psi_2+\psi_3)}{Z_{Q_0}}\int_{\BS}
(\mm\cdot\nn)^3(\mm\cdot\tilde\nn)^2(\mm\cdot\pp)\exp\big(\eta(\mm\cdot\nn)^2\big)\ud\mm,\\
&\big(Q_0:\widehat{M}^{(4)}_1(Q_1^\top)\big):(\nn\pp+\pp\nn)\\
~~&=\frac{4S_2(\psi_2+\psi_3)^2}{Z_{Q_0}}\int_{\BS}\big((\mm\cdot\nn)^2-\frac13\big)
(\mm\cdot\nn)^3(\mm\cdot\tilde\nn)^2(\mm\cdot\pp)\exp\big(\eta(\mm\cdot\nn)^2\big)\ud\mm,\\
&\widehat{Z}_1(Q_1^\top)\big(Q_0:M^{(4)}_{Q_0}\big):(\nn\pp+\pp\nn)\\
~~&=\frac{2S_2\widehat{Z}_1(Q_1^\top)}{Z_{Q_0}}\int_{\BS}\big((\mm\cdot\nn)^2-\frac13\big)
(\mm\cdot\nn)(\mm\cdot\pp)\exp\big(\eta(\mm\cdot\nn)^2\big)\ud\mm.
\end{align*}
By the coordinate invariance, we may assume $\nn=(0,0,1)^T$ and $\tilde\nn=(a_1, b_1, 0)^T$, $\pp=(a_2, b_2, 0)^T$.
Let $\mm=(\sin\theta\cos\varphi,\sin\theta\sin\varphi,\cos\theta)^T$, then
\begin{align*}
&\Big(Q_1^\top:\CU_{\nn}(\CQ^{-1}_{\nn}(Q_1^\top))\Big):(\nn\pp+\pp\nn)\\
&=\frac{8(\psi_2+\psi_3)}{Z_{Q_0}}\int_0^{2\pi}\int_0^\pi\cos^3\theta\sin^4\theta(a_1\cos\varphi+b_1\sin\varphi)^2
(a_2\cos\varphi+b_2\sin\varphi)e^{\eta\cos^2\theta} d\theta d\varphi\\
&=0.
\end{align*}
Similarly, we have
\begin{align*}
\big(Q_0:\widehat{M}^{(4)}_1(Q_1^\top)\big):(\nn\pp+\pp\nn),
\qquad \widehat{Z}_1(Q_1^\top)\big(Q_0:M^{(4)}_{Q_0}\big):(\nn\pp+\pp\nn)=0.
\end{align*}
This completes the proof of Lemma \ref{lem:s5-2}.
\end{proof}

We are now in a position to derive the systems of $(\vv_1,Q^{\top})$. We denote
\begin{eqnarray*}
 &~&\A_1=\CP^{in}\big(\CJ_{\nn}(\CL (Q^{\top}_1))\big),~~\A_2=\CP^{out}\big(\CJ_{\nn}(\CL (Q^{\top}_1))\big),\\
 &~&\B_1=\CP^{in}\big(\CJ_{Q_0}(\nabla\vv_1)\big),~~\B_2=\CP^{out}\big(\CJ_{Q_0}(\nabla\vv_1)\big).
\end{eqnarray*}
Taking the projection $\CP^{in}$ on both sides of (\ref{expan-newQ-Q1}), note that $\CH_{\nn}(Q_2)\in\mathbb{Q}^{out}_{\nn}$
and $\CJ_{\nn}(\CL (Q_1))=\CJ_{\nn}(\CL (Q^{\top}_1))+R$, from Lemma \ref{lem:s5-1} and Lemma \ref{lem:s5-2} we get
\begin{eqnarray}\label{hilbert-exq1}
\frac{\partial Q^{\top}_1}{\partial t}+\vv_0\cdot\nabla Q^{\top}_1=-4\A_1+2\B_1
 +L(Q^{\top}_1)+L(\vv_1)+R.
\end{eqnarray}
Here we have absorbing $\CP^{in}\big(\vv_1\cdot\nabla Q_0\big)$ into $L(\vv_1)$.
Taking the projection $\CP^{out}$ on both sides of (\ref{expan-newQ-Q1}), we have
\begin{eqnarray*}
-4\CJ_{\nn}\big(\CH_{\nn}(Q_2)\big)-4\A_2+2\B_2+\overline{\FF}_1(Q^{\top}_1)+L(Q^{\top}_1)+L(\vv_1)+R=0,
\end{eqnarray*}
which implies that
\begin{align}\label{Q2B1W}
-2\CJ_{\nn}(\CH_{\nn}(Q_2))+\GG_1=2\A_2-\B_2+L(Q^{\top}_1)+L(\vv_1)+R.
\end{align}
Substituting (\ref{Q2B1W}) to (\ref{expan-newQ-v1}) and together with (\ref{hilbert-exq1}),
we obtain the following closed system for $(\vv_1,Q^{\top}_1)$
\begin{align}
\frac{\partial Q^{\top}_1}{\partial t}+\vv_0\cdot\nabla Q^{\top}_1=&~-4\A_1+\B_1
 +L(Q^{\top}_1)+L(\vv_1)+R,\label{hilbert-ex1}\\
\frac{\partial \vv_1}{\partial t}+\vv_0\cdot\nabla \vv_1~=&
   -\nabla p_1 +\frac{\gamma}{Re}\Delta\vv_1+\frac{1-\gamma}{2Re}\nabla\cdot \big(\DD_1:M^{(4)}_{Q_0}\big)\nonumber\\
  &-\frac{1-\gamma}{Re}\Big(\nabla\cdot\big(2\A_2
  -\B_2-2\CM_{Q_0}(\CL (Q^{\top}_1))+L(Q^{\top}_1)+R\big)\nonumber\\
  &-\nabla\cdot\big(\sigma^d(Q_0,Q^{\top}_1)+\sigma^d(Q^{\top}_1,Q_0)\big)\Big)+L(\vv_1),
  \label{hilbert-ex2}\\
 \nabla\cdot \vv_1=&~0.\label{hilbert-ex3}
\end{align}
Apparently, (\ref{hilbert-ex1})-(\ref{hilbert-ex3}) is a linear system of $(\vv_1,Q^{\top}_1)$.
To prove its solvability, we give a priori estimate for the energy
\[
 E_k\eqdefa\sum^{k-4}_{|\ell|=0}\Big(\langle\partial^{\ell}\vv_1,\partial^{\ell}\vv_1\rangle
 +\frac{1-\gamma}{2Re}\big\langle\partial^{\ell}Q^{\top}_1,\CL(\partial^{\ell}Q^{\top}_1)\big\rangle\Big)
 +\big\langle Q^{\top}_1,Q^{\top}_1\big\rangle.
\]
We will prove that there exists a positive constant $C$ such that
\begin{eqnarray}\label{338}
 \frac{d}{dt}E_k\leq C(E_k+\|R(t)\|_{H^{k-3}}),
\end{eqnarray}
which ensure that the systems (3.34)-(3.36) have a unique solution $(\vv_1,Q^{\top}_1)$ on $[0,T]$
satisfying
\begin{eqnarray}
 \vv_1\in C([0,T];H^{k-4}),~~Q^{\top}_1\in C([0,T];H^{k-3}).
\end{eqnarray}

Without loss of generality, we only prove (\ref{338}) in the case of $\ell=0$ and the proof is similar for the general case.
When $\ell=0$, the corresponding energy is given by
\[
 E_1=\langle\vv_1,\vv_1\rangle+\langle Q^{\top}_1,Q^{\top}_1\rangle
 +\frac{1-\gamma}{2Re}\big\langle Q^{\top}_1,\CL (Q^{\top}_1)\big\rangle.
\]
First, we get by (\ref{hilbert-ex1}) that
\begin{align}
 \frac12\frac{d}{dt}\big\langle Q^{\top}_1,Q^{\top}_1\big\rangle
 &=\big\langle-4\CJ_{\nn}(\CL (Q^{\top}_1))+\CJ_{\nn}(\nabla\vv_1),Q^{\top}_1\big\rangle
 +\big\langle L(\vv_1)+L(Q^{\top}_1)+G,Q^{\top}_1\big\rangle\nonumber\\
 &\leq\delta\|\nabla\vv_1\|^2_{L^2}+C_{\delta}\|Q^{\top}_1\|^2_{H^1}+C(\|\vv_1\|^2_{L^2}+\|R\|^2_{L^2}).\label{q1q1}
\end{align}
Meanwhile, we can obtain from (\ref{hilbert-ex1}) and (\ref{hilbert-ex2}) that
\begin{align}
 &\frac12\frac{d}{dt}\Big(\frac{2Re}{1-\gamma}\langle\vv_1,\vv_1\rangle+2\langle Q^{\top}_1,\CL (Q^{\top}_1)\rangle\Big)
 =\frac{2Re}{1-\gamma}\langle\partial_t\vv_1,\vv_1\rangle
 +2\big\langle\partial_tQ^{\top}_1,\CL (Q^{\top}_1)\big\rangle\nonumber\\
 =&-\frac{2\gamma}{1-\gamma}\|\nabla\vv_1\|^2_{L^2}-\big\langle\DD_1:M^{(4)}_{Q_0},\DD_1\big\rangle
 +\underbrace{2\big\langle2\A_2-\B_2
 -2\CM_{Q_0}(\CL (Q^{\top}_1)),\nabla\vv_1\big\rangle}_{I_1}\nonumber\\
 &+\underbrace{2\big\langle L(\vv_1)+L(Q^{\top}_1)+R-\frac{Re}{1-\gamma}\big(\sigma^d (Q_0,Q^{\top}_1)
 +\sigma^d (Q^{\top}_1,Q_0)\big), \nabla\vv_1\big\rangle}_{I_2}\nonumber\\
 &\underbrace{-2\big\langle\vv_0\cdot\nabla Q^{\top}_1,\CL (Q^{\top}_1)\big\rangle}_{I_3}
 +\underbrace{2\big\langle-4\A_1+ \B_1,\CL (Q^{\top}_1)\big\rangle}_{I_4}
 +\underbrace{2\big\langle L(\vv_1)+L(Q^{\top}_1)+R,\CL (Q^{\top}_1)\big\rangle}_{I_5}.\nonumber
\end{align}
For $I_2, I_3$ and $I_5$, we have
\begin{align}
I_3&=-2\int_{\mathbb{R}^3}\Big(
L_1\partial_{j}\vv_{0i}\partial_{i}Q^{\top}_{1i'j'}\partial_{j}Q^{\top}_{1i'j'}
+2L_2(\partial_{l}\vv_{0j}\partial_{j}Q^{\top}_{1kl}\partial_{m}Q^{\top}_{1km}
+\partial_{k}\vv_{0j}\partial_{j}Q^{\top}_{1kl}\partial_{m}Q^{\top}_{1lm})\Big)\ud\xx\nonumber\\
&\leq C\|Q^{\top}_1\|^2_{H^1},\\
I_2+I_5&\leq \delta\|\nabla\vv_1\|^2_{L^2}+C_{\delta}(\|\vv_1\|^2_{L^2}+\|Q^{\top}_1\|^2_{H^1}+\|R\|^2_{H^1}).
\end{align}
Now we turn to estimate $I_1+I_4$. Recalling the fact that for any
$Q\in \mathbb{Q}^{in}_{\nn} ~(\mathbb{Q}^{out}_{\nn})$,
$\CJ_{\nn}(Q)$ and $\CM_{Q_0}(Q)$ belong to $\mathbb{Q}^{in}_{\nn}~(\mathbb{Q}^{out}_{\nn})$, we have:
\begin{align}
&\Big\langle\CP^{out}\big(\CJ_\nn(\nabla\vv_1)\big),\nabla\vv_1\Big\rangle
=\Big\langle\CP^{out}\big(\CJ_\nn(\nabla\vv_1)\big),\DD_1\Big\rangle=\Big\langle\CM_\nn(\nabla\vv_1),\CP^{out}(\DD_1)\Big\rangle\nonumber\\
&=\Big\langle\nabla\vv_1,\CM_\nn\big(\CP^{out}(\DD_1)\big)\Big\rangle=\Big\langle\nabla\vv_1,\CJ_\nn\big(\CP^{out}(\DD_1)\big)\Big\rangle
=\Big\langle\CP^{out}(\DD_1), \CJ_\nn\big(\CP^{out}(\DD_1)\big)\Big\rangle\ge 0,\label{esti:Qv1-1}
\end{align}
where we have repeatedly used the symmetry of $\CJ_\nn(\cdot)$ and the self-adjointness
of $\CM_\nn$ (note that $\CM_\nn(\cdot)$ is not symmetric and $\CJ_\nn$ is not self-adjoint).
Similarly, it holds that
\begin{align}\label{esti:Qv1-2}
\Big\langle\CJ_{\nn}(\CL (Q^{\top}_1)),\CP^{in}\big(\CL (Q^{\top}_1)\big)\Big\rangle\ge0.
\end{align}
On the other hand, thanks to (\ref{eq:M-J}) and (\ref{comm:LJ}), we have
\begin{align}
&\Big\langle\CP^{out}\big(\CJ_{\nn}(\CL (Q^{\top}_1))\big), \nabla\vv_1\Big\rangle
+\Big\langle\CP^{in}\big(\CJ_{\nn}(\nabla\vv_1)\big), \CL (Q^{\top}_1)\Big\rangle\nonumber\\
&=\Big\langle\CJ_{\nn}\big(\CP^{out}(\CL (Q^{\top}_1))\big), \nabla\vv_1\Big\rangle
+\Big\langle\CJ_{\nn}(\nabla\vv_1), \CP^{in}\CL (Q^{\top}_1)\Big\rangle\nonumber\\
&=\Big\langle\CM_{\nn}\big(\CP^{out}(\CL (Q^{\top}_1))\big), \nabla\vv_1\Big\rangle
+\Big\langle\CM_{\nn}(\nabla\vv_1), \CP^{in}\CL (Q^{\top}_1)\Big\rangle\nonumber\\
&=\Big\langle\CM_{Q_0}(\CL (Q^{\top}_1)),\nabla\vv_1\Big\rangle. \label{esti:Qv1-3}
\end{align}
Combining (\ref{esti:Qv1-1})-(\ref{esti:Qv1-3}), we get
\begin{align}
I_1+I_4 \le 0.
\end{align}
Therefore, we obtain the following energy inequality
\[
\frac{d}{dt}E_1\leq C(E_1+\|R\|_{H^1}^2),
\]
which indicates the existence of $(\vv_1,Q_1)$.

Again, we write $Q_2=Q^{\top}_2+Q^{\perp}_2$ with $Q^{\top}_2\in\mathbb{Q}^{in}_{\nn}$
and $Q^{\perp}_2\in\mathbb{Q}^{out}_{\nn}$. By (\ref{Q2B1W}) we can solve $Q^{\perp}_2$ as
\begin{align}\label{solveQ2}
 Q^{\perp}_2=\frac12\CH^{-1}_{\nn}\CJ^{-1}_{\nn}\big(-2\A_2+\frac12\B_2-\GG_1-L(\vv_1)-L(Q^{\top}_1)-R\big)
 \in C([0,T];H^{k-5}).
\end{align}
Then, $(\vv_2,Q^{\top}_2)$ can be solved in a similar way as $(\vv_1,Q^{\top}_1)$.
$Q_3$ can be solved similarly as in (\ref{solveQ2})(unique up to a term in $\QI$). We omit the details
and leave them to the interest readers.

To summarize, we have proved:
\begin{proposition}\label{prop:Hilbert}
Let $(\vv_0, \nn)$ be a solution of (\ref{eq:EL-v})-(\ref{eq:EL-n}) on $[0,T]$ and satisfy
\beno
\vv_0\in C([0,T];H^{k}), \quad \nabla\nn\in C([0,T];H^{k})\quad \textrm{for} \quad k\ge 20.
\eeno
There exists the solution $(\vv_i, Q_i)(i=0,1,2)$ and $Q_3\in \mathbb{Q}^{out}_{\nn}$ of the system
(\ref{expan-newQ-Q1})-(\ref{expan-newQ-2}) satisfying
\beno
\vv_i\in  C([0,T];H^{k-4i}), \quad Q_i\in C([0,T];H^{k+1-4i})(i=0,1,2),\quad Q_3\in C([0,T];H^{k-11}).
\eeno
\end{proposition}

\subsection{The system for the remainder}

In this subsection, we focus on the derivation of systems of the remainder and uniform estimates for the remainder.
Throughout this subsection, we assume that $\vv_i\in C([0,T];H^{k-4i})$ for $i=0,1,2$ and $Q_i\in C([0,T];H^{k+1-4i})$ for $i=0, 1, 2, 3$.
We denote by $C$ a constant depending on $\displaystyle\sum_{i=0}^2\sup_{t\in [0,T]}\|\vv_i(t)\|_{H^{k-4i}}$
and $\displaystyle\sum_{i=0}^3\sup_{t\in [0,T]}\|Q_i(t)\|_{H^{k+1-4i}}$, and independent of $\ve$.

Let
\begin{align*}
E&~=\|Q_R\|_{H^1}+\ve\|\Delta Q_R\|_{L^2}+\ve^2\|\nabla\Delta Q_R\|_{L^2}
+\|\vv_R\|_{L^2}+\ve\|\nabla\vv_R\|_{L^2}+\ve^2\|\Delta\vv_R\|_{L^2},\\
F&~=\ve\|\nabla\CL(Q_R)\|_{L^2}+\ve^2\|\Delta\CL(Q_R)\|_{L^2}+\ve^2\|\Delta\nabla\vv_R\|_{L^2}.
\end{align*}
By Sobolev embedding inequality, for $k=0,1,2$,  we have
\begin{align}
&\ve^k\|Q_R\|_{H^k}+\ve^{k}\|\vv_R\|_{H^k}\le E, \quad \ve\|Q_R\|_{L^\infty}+\ve^2\|\vv_R\|_{L^\infty}\le C E,\\
&\ve^{k+1}\|\CL(Q_R)\|_{H^k}+\ve^3\|\nabla\vv_R\|_{L^\infty}\le C(E+\ve F),\quad
\end{align}
for some constant $C$. To simplify the formulation, we introduce a notation $\mathfrak{R}$
to denote all the terms (called {\it good terms}) which can be controlled by
\begin{align}
\|\mathfrak{R}\|_{L^2} +\ve\|\nabla\mathfrak{R}\|_{L^2}+\ve^2\|\Delta\mathfrak{R}\|_{L^2}
\le C(\ve E)(1+E+\ve F)+\ve f(E),
\end{align}
where $C(\cdot)$ and $f(\cdot)$: $\mathbb{R}^+\cup\{0\}\mapsto\mathbb{R}^+\cup\{0\}$ are increasing functions.
They may depend on $\|Q_i\|$ and the parameters of the system, but are independent of $\ve$.
The main feature of the righthand side is that it is almost controlled by $C(1+E)$ when $\ve\to0$. Therefore, we can deduce a
closed energy estimate uniformly in $\ve$, see Proposition \ref{prop:energy}.
Since $\|Q_0-Q^*\|_{H^k}, \|Q_i\|_{H^k}(k\le 3, 1\le i\le 3)$ are all bounded by a constant independent on $\ve$, we have that
\begin{align*}
\|Q^\ve-Q^*\|_{H^k}\le C+\ve^3\|Q_R\|_{H^k} \le C(\ve E),\qquad \|\vv^\ve\|_{H^k}\le C(\ve E).
\end{align*}

We explain the motivation to introduce this definition. To control the remainder term, first we have to
write down the evolution equation for $Q_R$ and $\vv_R$. In other words, we have to calculate
$$\frac1{\ve^3}\Big(\frac{\partial}{\partial t}Q^\ve-\frac{\partial}{\partial t}Q_0-\ve \frac{\partial}{\partial t}Q_1
-\ve^2\frac{\partial}{\partial t}Q_2-\ve^3\frac{\partial}{\partial t}Q_3\Big).$$
The system for $(Q^\ve,\vv^\ve)$ can be written in the following abstract form:
\begin{align}
\frac{\partial}{\partial t}Q^\ve &=\frac1\ve \FF(Q^\ve)+\GG(Q^\ve, \vv^\ve),\\
\frac{\partial}{\partial t}\vv^\ve&= P_\text{div}\nabla\cdot\Big( \frac1\ve\HH(Q^\ve)+\JJ(Q^\ve, \vv^\ve)\Big),
\end{align}
where $P_\text{div}$ is projection operator which projects a
vector field to its solenoidal part, and
\begin{align*}
\FF(Q)&=-6Q+2\alpha\big(\mathcal{M}_{Q}(Q)+\mathcal{M}_{Q}^T(Q)\big),\quad \HH(Q)=\frac{1-\gamma}{Re}\Big(3Q-2\alpha\mathcal{M}_{Q}(Q)\Big),\\
\GG(Q,\vv)&=-2\big(\mathcal{M}_{Q}(\CL (Q))+\mathcal{M}_{Q}^T(\CL (Q))\big)
+\big(\mathcal{M}_{Q}(\nabla\vv)+\mathcal{M}^T_Q(\nabla\vv)\big)-\vv\cdot\nabla Q\\
&\triangleq\GG_1(Q)+\GG_2(Q,\vv)+\GG_3(Q,\vv),\\
\JJ(Q,\vv)&=-\frac{2(1-\gamma)}{Re}\mathcal{M}_{Q}(\CL (Q))
+\frac{1-\gamma}{Re}\sigma^d(Q,Q)+\frac{1-\gamma}{2Re}\DD:M^{(4)}_{Q}-\vv\otimes\vv+\frac{2\gamma}{Re}\DD\\
 &\triangleq\JJ_1(Q)+\JJ_2(Q)+\JJ_3(Q,\vv)+\JJ_4(\vv)+\frac{2\gamma}{Re}\DD.
\end{align*}
Then we have
\begin{align*}
\frac{\partial}{\partial t}Q_R=&\frac{1}{\ve^4}\Big(\FF(Q^\ve)-\FF(\widetilde{Q})\Big)
+\frac1{\ve^3}\Big(\GG(Q^\ve, \vv^\ve)-\GG(\widetilde Q, \tilde\vv)\Big)
+\frac{1}{\ve^3}\Big(\frac1\ve\FF(\widetilde{Q})+\GG(\widetilde Q, \tilde\vv)-\frac{\partial}{\partial t}\widetilde Q\Big),\\
\frac{\partial}{\partial t}\vv_R=&P_\text{div}\Big(\frac{1}{\ve^4}\big(\HH(Q^\ve)-\HH(\widetilde{Q})\big)
+\frac1{\ve^3}\big(\JJ(Q^\ve, \vv^\ve)-\JJ(\widetilde Q,\tilde\vv)\big)\Big)\\
&+\frac{1}{\ve^3}\Big(P_\text{div}\big(\frac1\ve\HH(\widetilde{Q})+\JJ(\widetilde Q, \tilde\vv)\big)-\frac{\partial}{\partial t}\tilde\vv\Big).
\end{align*}
By the choices of $Q_i(0\le i\le 3), \vv_j(0\le j\le 2)$, we know that
\begin{align}
\Big\|\frac{1}{\ve^3}\Big(\frac1\ve\FF(\widetilde{Q})+\GG(\widetilde Q, \tilde\vv)-\frac{\partial}{\partial t}\widetilde Q\Big)\Big\|_{H^2},\quad
\Big\|\frac1{\ve^3}\big(\JJ(Q^\ve, \vv^\ve)-\JJ(\widetilde Q,\tilde\vv)\big)\Big\|_{H^2}
\end{align}
are bounded by a constant uniformly in $\ve$, then they are good terms.
\begin{lemma}\label{lem:remGJ}For the difference terms arising from $\GG$ and $\JJ$, we have
\begin{align}
\GG_1(Q^\ve)-\GG_1(\widetilde Q)&=-2\ve^3\big(\mathcal{M}_{Q_0}(\CL (Q_R))+\mathcal{M}_{Q_0}^T(\CL (Q_R))\big)+\ve^3\mathfrak{R},\label{rem:G1}\\
\GG_2(Q^\ve,\vv^\ve)-\GG_1(\widetilde Q, \tilde\vv)&=\ve^3\big(\mathcal{M}_{Q_0}(\nabla\vv_R)
+\mathcal{M}_{Q_0}^T(\nabla\vv_R)\big)+\ve^3\mathfrak{R},\label{rem:G2}\\
\GG_3(Q^\ve,\vv^\ve)-\GG_3(\widetilde Q, \tilde\vv)&=\ve^3\mathfrak{R},\label{rem:G3}\\
\JJ_1(Q^\ve)-\JJ_1(\widetilde Q)&=-\ve^3\frac{2(1-\gamma)}{Re}\mathcal{M}_{Q_0}(\CL (Q_R))+\ve^3\mathfrak{R},\label{rem:J1}\\
\JJ_2(Q^\ve)-\JJ_2(\widetilde Q)&=\ve^3\mathfrak{R},\label{rem:J2}\\
\JJ_3(Q^\ve,\vv^\ve)-\JJ_3(\widetilde Q,\tilde\vv)
&=\ve^3\frac{1-\gamma}{2Re}\DD_R:M^{(4)}_{Q_0}+\ve^3\mathfrak{R},\label{rem:J3}\\
\JJ_4(\vv^\ve)-\JJ_4(\tilde\vv)&=\ve^3\mathfrak{R}.\label{rem:J4}
\end{align}
\end{lemma}
\begin{proof}
First, by Lemma \ref{prop:diff-CM}, for $0\leq k\leq2$, we have
\begin{align*}
&\|\mathcal{M}_{Q^\ve}(\CL (Q^\ve))-\mathcal{M}_{\widetilde Q}(\CL (Q^\ve))\|_{H^k}\\
&\le \|\mathcal{M}_{Q^\ve}(\CL(\widetilde Q))-\mathcal{M}_{\widetilde Q}(\CL (\widetilde Q))\|_{H^k}
+\ve^3\|\mathcal{M}_{Q^\ve}(\CL (Q_R))-\mathcal{M}_{\widetilde Q}(\CL (Q_R))\|_{H^k}\\
&\le C\big(\|Q^\ve\|_{L^\infty}, \|\widetilde Q\|_{L^\infty}, \|Q^\ve-Q^*\|_{H^{k}},\|\widetilde Q-Q^*\|_{H^{k}}\big)\|\ve^3Q_R\|_{H^k}\|\CL (\widetilde Q)\|_{H^{k+2}}\\
&\quad+\ve^3 C\big(\|Q^\ve\|_{L^\infty}, \|\widetilde Q\|_{L^\infty}\big)\|\CL (Q_R)\|_{H^k}\|\ve^3Q_R\|_{L^\infty}\\
&\quad+\ve^3C\big(\|Q^\ve\|_{L^\infty}, \|\widetilde Q\|_{L^\infty},\|Q^\ve-Q^*\|_{H^{k}},\|\widetilde Q-Q^*\|_{H^{k}}\big)\|\CL (Q_R)\|_{L^\infty}\|\ve^3Q_R\|_{H^k},\\
&\ve^3\|\mathcal{M}_{\widetilde Q}(\CL (Q_R))-\mathcal{M}_{Q_0}(\CL (Q_R))\|_{H^k}\\
&\le \ve^3 C\big(\|Q_0\|_{L^\infty}, \|\widetilde Q\|_{L^\infty}, \|Q_0-Q^*\|_{H^{k+2}},\|\widetilde Q-Q^*\|_{H^{k+2}}\big)
\|\widetilde Q-Q_0\|_{H^{k+2}}\|\CL (Q_R)\|_{H^k}.
\end{align*}
Using $\|\widetilde Q-Q_0\|_{H^{k+2}}=\ve\|Q_1+\ve Q_2+\ve^2Q_3\|_{H^{k+2}}\le CE$,  we have
\begin{align*}
\ve^k\|\mathcal{M}_{Q^\ve}(\CL (Q^\ve))-\mathcal{M}_{\widetilde Q}(\CL (Q^\ve))\|_{H^k}
&\le C(\ve E)\ve^3 E+C(\ve E)\ve^4E(E+F),\\
\ve^{k+3}\|\mathcal{M}_{\widetilde Q}(\CL (Q_R))-\mathcal{M}_{Q_0}(\CL (Q_R))\|_{H^k}
&\le \ve^3 C(E+\ve F).
\end{align*}
Thus we obtain
\begin{align*}
\mathcal{M}_{Q^\ve}(\CL (Q^\ve))=\ve^3\mathcal{M}_{Q_0}(\CL (Q_R))+\ve^3\mathfrak{R}.
\end{align*}
This implies (\ref{rem:J1}). (\ref{rem:G1}) and (\ref{rem:G2}) can be proved in the same way.
Moreover, $M^{(4)}_{Q_1}-M^{(4)}_{Q_2}$ shares the same estimate with $\mathcal{M}_{Q_1}-\mathcal{M}_{Q_2}$, so (\ref{rem:J3}) is also true.

For (\ref{rem:J2}), we have
\begin{align*}
\ve^k\|\sigma^d (Q^\ve,Q^\ve)-\sigma^d(\widetilde Q,\widetilde Q)\|_{H^k}
&=\ve^{3+k}\|\sigma^d(Q^\ve,Q_R)+\sigma^d(Q_R,\widetilde Q)\|_{H^k}\\
&\le C \ve^3\|\ve^k\nabla Q_R\|_{H^k}(1+\|\ve^3 \nabla Q_R\|_{L^\infty})\\
&\le \ve^3C(1+\ve E)E.
\end{align*}
In the same way, (\ref{rem:G3}) and (\ref{rem:J4}) can be deduced.
\end{proof}

\begin{lemma}\label{lem:remHF}For the difference terms arising from $\HH$ and $\FF$, we have
\begin{align}\label{rem:H}
\HH(Q^\ve)-\HH(\widetilde Q)&=2\ve^3\frac{1-\gamma}{Re}\CM_{Q_0}\big(\CH_{Q_0}(Q_R)\big)+\ve^4\mathfrak{R},\\
\FF(Q^\ve)-\FF(\widetilde Q)& =2\ve^3\Big(\CM_{Q_0}\big(\CH_{Q_0}(Q_R)\big)+\CM^T_{Q_0}\big(\CH_{Q_0}(Q_R)\big)\Big)+\ve^4\mathfrak{R}.\label{rem:F}
\end{align}
\end{lemma}
\begin{proof}
First, we have
\begin{align*}
&3Q^\ve-2\alpha\mathcal{M}_{Q^\ve}(Q^\ve)-(3\widetilde Q-2\alpha\mathcal{M}_{\widetilde Q}(\widetilde Q))\\
&=2\mathcal{M}_{Q^\ve}(B_{Q^\ve}-\alpha Q^\ve)-2\mathcal{M}_{\widetilde Q}(B_{\widetilde Q}-\alpha\widetilde Q)\\
&=2\mathcal{M}_{Q^\ve}(B_{Q^\ve}-\alpha Q^\ve)-2\mathcal{M}_{\widetilde Q}(B_{Q^\ve}
-\alpha Q^\ve)+2\mathcal{M}_{\widetilde Q}(B_{Q^\ve}-\alpha Q^\ve)-2\mathcal{M}_{\widetilde Q}(B_{\widetilde Q}-\alpha\widetilde Q).
\end{align*}
Using Talylor expansion for $B_Q=B(Q)$, we get
\begin{align}
\|B_{Q^\ve}-B_{\widetilde Q}-\ve^3&\mathcal{Q}^{-1}_{\widetilde Q}(Q_R)\|_{H^k}
=\Big\|\ve^6\int_{0}^1s(Q_R:\nabla_Q)^2B_{(\widetilde Q+s\ve^3Q_R)}ds\Big\|_{H^k}\nonumber\\\label{eq:rem:M1}
&\le \ve^6 C(\ve^3\|Q_R\|_{H^2})\|Q_R\|_{H^k}\|Q_R\|_{H^2}\le\ve^5 C(\ve E)E \|Q_R\|_{H^k}.
\end{align}
Therefore, $B_{Q^\ve}-B_{\widetilde Q}-\ve^3\mathcal{Q}^{-1}_{\widetilde Q}(Q_R)=\ve^4\mathfrak{R},$ which implies
\begin{align*}
\mathcal{M}_{\widetilde Q}(B_{Q^\ve}-\alpha Q^\ve)-\mathcal{M}_{\widetilde Q}(B_{\widetilde Q}-\alpha\widetilde Q)
=&\ve^3\mathcal{M}_{\widetilde Q}(\mathcal{Q}^{-1}_{\widetilde Q}(Q_R)-\alpha Q_R)+\ve^4\mathfrak{R}\\
=&\ve^3\mathcal{M}_{Q_0}(\mathcal{Q}^{-1}_{Q_0}(Q_R)-\alpha Q_R)+\ve^4\mathfrak{R}.
\end{align*}
By Lemma \ref{prop:diff-CM}, we can obtain
\begin{align*}
&\|\mathcal{M}_{Q^\ve}(B_{Q^\ve}-\alpha Q^\ve)-\mathcal{M}_{\widetilde Q}(B_{Q^\ve}-\alpha Q^\ve)\|_{H^k}\\
&\le C(\|\ve^3Q_R\|_{H^2})\|\ve^3Q_R\|_{H^k}\|B_{Q^\ve}-B_{Q_0}-\alpha (Q^\ve-Q_0)\|_{H^2}\\
&\le C(\|\ve^3Q_R\|_{H^2})\ve^4\|Q_R\|_{H^k}(1+\|\ve^2Q_R\|_{H^2})\le \ve^4 C(\ve E)\|Q_R\|_{H^k}(1+\ve E).
\end{align*}
Thus
\begin{align}\label{eq:rem:M2}
\mathcal{M}_{Q^\ve}(B_{Q^\ve}-\alpha Q^\ve)-\mathcal{M}_{\widetilde Q}(B_{Q^\ve}-\alpha Q^\ve)\in\mathfrak{R}.
\end{align}
Combining (\ref{eq:rem:M1}) and (\ref{eq:rem:M2}) and recalling the
definition of $\CH_{Q_0}(Q_R):=\CQ^{-1}_{Q_0}(Q_R)-\alpha Q_R$, we get (\ref{rem:H}) and (\ref{rem:F}).
\end{proof}

Combining Lemma \ref{lem:remGJ} with Lemma \ref{lem:remHF}, we finally arrive at
\begin{align}
\frac{\partial Q_R}{\partial t}=&~-4\CJ_{\nn}
\Big(\frac{1}{\ve}\CH_{\nn}(Q_R)+\CL(Q_R)\Big) +2\CJ_{\nn}(\nabla\vv_R)+\mathfrak{R},\label{eq:QR}\\
\frac{\partial \vv_R}{\partial t}=&~-\nabla p_R +\frac{\gamma}{Re}\Delta\vv_R+\frac{1-\gamma}{2Re}
       \nabla\cdot(\DD_R:M^{(4)}_{\widetilde{Q}})\nonumber\\
       &~~+\frac{1-\gamma}{Re}\nabla\cdot\Big(2\CM_{Q_0}\Big(\frac1\ve\CH_{\nn}(Q_R)+\CL (Q_R)\Big)\Big)
         +\nabla\cdot\mathfrak{R} +\mathfrak{R},\label{eq:vR}\\
 \nabla\cdot \vv_R=&~0.\label{eq:imcomR}
\end{align}

\subsection{Uniform estimates for the remainder}

In order to obtain the uniform energy estimates, we introduce
\begin{align*}
\CH^{\ve}_{\nn}(Q_R)=\CH_{\nn}(Q_R)+\ve\CL (Q_R),
\end{align*}
and the following energy functional
\begin{align*}
\Ef(t)=&\frac12\int\Big(|\vv_R|^2+\CJ^{-1}_{\nn}(Q_R):Q_R+\frac{1-\gamma}{\ve Re}\CH^{\ve}_{\nn}(Q_R):Q_R\Big)
 +\ve^2\Big(|\nabla \vv_R|^2\\
 &+\frac{1-\gamma}{\ve Re}\CH^{\ve}_{\nn}(\nabla Q_R):\nabla Q_R\Big)
 +\ve^4\Big(|\Delta \vv_R|^2+\frac{1-\gamma}{\ve Re}\CH^{\ve}_{\nn}(\Delta Q_R):\Delta Q_R\Big)\ud\xx,\\
 \Ff(t)=&\int\Big(\frac{\gamma}{Re}|\nabla\vv_R|^2
    +\frac{4(1-\gamma)}{\ve^2Re}\CJ_{\nn}(\CH^{\ve}_{\nn}(Q_R)):\CH^{\ve}_{\nn}(Q_R)\Big)\\
 &+\ve^2\Big(\frac{\gamma}{Re}|\Delta\vv_R|^2
    +\frac{4(1-\gamma)}{\ve^2Re}\CJ_{\nn}(\CH^{\ve}_{\nn}(\nabla Q_R)):\CH^{\ve}_{\nn}(\nabla Q_R)\Big)\\
 &+\ve^4\Big(\frac{\gamma}{Re}|\nabla\Delta\vv_R|^2
    +\frac{4(1-\gamma)}{\ve^2Re}\CJ_{\nn}(\CH^{\ve}_{\nn}(\Delta Q_R)):\CH^{\ve}_{\nn}(\Delta Q_R)\Big)\ud\xx.
\end{align*}

\begin{lemma}\label{lem:energy}
There holds
\begin{align}
||Q_R||_{H^1}+||(\ve\nabla^2Q_R,\ve^2\nabla^3Q_R)||_{L^2}
  +||(\vv_R,\ve\nabla\vv_R,\ve^2\nabla^2\vv_R)||_{L^2}\leq &~C\Ef^{\frac12},\\
||(\frac{1}{\ve}\CH^{\ve}_{\nn}( Q_R),\nabla\CH^{\ve}_{\nn}(Q_R),
    \ve\Delta\CH^{\ve}_{\nn}(Q_R))||_{L^2}
    \leq&~ C(\Ef+\Ff)^{\frac12},\\
||\big(\ve\nabla\CL (Q_R),
    \ve^2\Delta\CL(Q_R)\big)||_{L^2}
    +||(\nabla\vv_R,\ve\nabla^2\vv_R,\ve^2\nabla^3\vv_R)||_{L^2}
    \leq&~ C(\Ef+\Ff)^{\frac12}.
\end{align}
\end{lemma}
\begin{proof} The first estimate follows from the non-negativity of $\CH_\nn$.
The second estimate can be deduced from the strict positivity of $\CJ_\nn$ and the following estimates for communicators
\begin{align*}
\|\partial_i\CH^{\ve}_{\nn}(Q_R)-\CH^{\ve}_{\nn}(\partial_iQ_R)\|_{L^2}\le C\|Q_R\|_{L^2},\\
\|\Delta\CH^{\ve}_{\nn}(Q_R)-\CH^{\ve}_{\nn}(\Delta Q_R)\|_{L^2}\le C\|Q_R\|_{H^1}.
\end{align*}
For the last one, we have
\begin{align*}
\|\ve\partial_i\CL (Q_R)\|_{L^2}&~=\|\CH^{\ve}_{\nn}(\partial_i Q_R)-\CH_\nn(\partial_iQ_R)\|_{L^2}\\
&~ \le \|\CH^{\ve}_{\nn}(\partial_iQ_R)\|_{L^2}+{C}\|\partial_i Q_R\|_{L^2}\le C(\Ef+\Ff)^{\frac12},\\
\|\ve^2\Delta\CL (Q_R)\|_{L^2} & ~\le \|\ve\CH^{\ve}_{\nn}(\Delta Q_R)\|_{L^2}+{C}\|\ve\Delta Q_R\|_{L^2}\le C(\Ef+\Ff)^{\frac12}.
\end{align*}
The estimates for $\vv_R$ is straightforward to prove. The proof is completed.
\end{proof}
\begin{corollary}\label{corol:EF}
$E\le C\Ef^{1/2}$, $F\le C(\Ef+\Ff)^{1/2}$.
\end{corollary}
Now, we state the a priori estimate for the remainder $(Q_R,\vv_R)$.
\begin{proposition}\label{prop:energy}
There exist two functions $C$ and $f$ depending on $Q_i, \vv_j$ and the parameters of
the system (except $\ve$), such that if $(\vv_R,Q_R)$ be a strong solution of the
system (\ref{eq:QR})--(\ref{eq:vR}) on $[0,T]$, then for any $t\in [0,T]$,  it holds that
\begin{align}
\frac{d }{d t}\Ef(t)+\Ff(t)\le
 C(\ve \Ef)\big(1+\Ef\big)+\ve f(\Ef)+C(\ve \Ef)\ve\Ff.\non
\end{align}
\end{proposition}
\begin{proof}
First, for $B\in \mathbb{Q}$, and $A\in\mathbb{R}^{3\times3}$, we have:
\begin{align}\label{eq:cancel}
\big\langle\CM_{Q_0}\big(\CH_{\nn}^\ve(B)\big),A\big\rangle
=\big\langle\CH_{\nn}^\ve(B),\CM_{Q_0}(A)\big\rangle
=\big\langle\CJ_\nn(A),\CH^{\ve}_{\nn}(B)\big\rangle.
\end{align}
This relation will be repeatedly used in the proof.

{\bf Step 1. $L^2$-estimate}\\
Using the equation (\ref{eq:QR}) and Lemma \ref{lem:energy}, we have
\begin{align}\label{estimate:Q0}
&\big\langle\frac{\partial Q_R}{\partial t},\CJ^{-1}_{\nn}(Q_R)\big\rangle+\frac{4}{\ve}
\big\langle\CH^{\ve}_{\nn}(Q_R),Q_R\big\rangle
 =\big\langle 2\nabla\vv_R,Q_R\big\rangle+\big\langle\mathfrak{R}, \CJ_\nn^{-1}(Q_R)\big\rangle\nonumber\\
&\qquad \le  C||Q_R||_{L^2}\big(\|\nabla\vv_R\|_{L^2}+\|\mathfrak{R}\|_{L^2}\big)
 \leq \delta_0 \Ff+C_{\delta_0}\Ef+C\|\mathfrak{R}\|_{L^2}^2.
\end{align}
We can also obtain
\begin{align*}
& \frac{Re}{1-\gamma}\big\langle\frac{\partial\vv_R}{\partial t},\vv_R\big\rangle+\big\langle\frac{\partial Q_R}{\partial t},
 \frac{1}{\ve}\CH^{\ve}_{\nn}(Q_R)\big\rangle\\
 =&-\frac{\gamma}{1\!-\!\gamma}\|\nabla\vv_R\|^2_{L^2}-\frac{1}{2}\big\langle\DD_R:M^{(4)}_{\widetilde{Q}},\DD_R\big\rangle
-\frac{2}{\ve}\big\langle\CM_{Q_0}(\CH_\nn^{\ve}(Q_R)),\nabla\vv_R\big\rangle
 + \frac{Re}{1-\gamma}\big\langle\nabla\cdot\mathfrak{R}+\mathfrak{R},\vv_R\big\rangle\\
 &-\frac{4}{\ve^2}\big\langle\CJ_{\nn}(\CH^{\ve}_{\nn}(Q_R)),\CH^{\ve}_{\nn}(Q_R)\big\rangle
 +\big\langle2\CJ_{\nn}(\nabla\vv_R),
  \frac{1}{\ve}\CH^{\ve}_{\nn}(Q_R)\big\rangle+\big\langle\mathfrak{R},\frac{1}{\ve}\CH^{\ve}_{\nn}(Q_R)\big\rangle.
\end{align*}
Using (\ref{eq:cancel}) with $(A,B)=(\nabla\vv_R, Q_R)$,
and the fact that $\big\langle\DD_R:M^{(4)}_{\widetilde{Q}},\DD_R\big\rangle\ge0$, we have
\begin{align}\label{estimate:H0}
& \frac{Re}{1-\gamma}\big\langle\frac{\partial\vv_R}{\partial t},\vv_R\big\rangle+\big\langle\frac{\partial Q_R}{\partial t},
\frac{1}{\ve}\CH^{\ve}_{\nn}(Q_R)\big\rangle+\frac{\gamma}{1\!-\!\gamma}\|\nabla\vv_R\|^2_{L^2}
+\frac{4}{\ve^2}\big\langle\CJ_{\nn}(\CH^{\ve}_{\nn}(Q_R)),\CH^{\ve}_{\nn}(Q_R)\big\rangle\nonumber\\
&\le\frac{Re}{1-\gamma}\big\langle\nabla\cdot\mathfrak{R}+\mathfrak{R},\vv_R\big\rangle
+\big\langle\mathfrak{R},\frac{1}{\ve}\CH^{\ve}_{\nn}(Q_R)\big\rangle \le \delta_0\Ff+C\Ef+C \|\mathfrak{R}\|_{L^2}^2.
\end{align}

{\bf Step 2. $H^1$-estimate}\\
Using (\ref{eq:QR})-(\ref{eq:imcomR}), we have
\begin{align*}
& \frac{Re}{1-\gamma}\big\langle\frac{\partial}{\partial t}\partial_i\vv_R,\partial_i\vv_R\big\rangle+\big\langle\frac{\partial}{\partial t}\partial_iQ_R,
 \frac{1}{\ve}\CH^{\ve}_{\nn}(\partial_iQ_R)\big\rangle\\
 =&-\frac{\gamma}{1\!-\!\gamma}\|\nabla\partial_i\vv_R\|^2_{L^2}-\frac{1}{2}\big\langle\partial_i\DD_R:M^{(4)}_{\widetilde{Q}},\partial_i\DD_R\big\rangle
-\frac{2}{\ve}\big\langle\CM_{Q_0}(\CH_\nn^{\ve}(\partial_iQ_R)),\nabla\partial_i\vv_R\big\rangle\\
&-\frac{1}{2}\big\langle\DD_R:\partial_iM^{(4)}_{\widetilde{Q}},\partial_i\DD_R\big\rangle
+\frac{2}{\ve}\big\langle[\CM_{Q_0}\CH_\nn^{\ve}, \partial_i]Q_R,\nabla\partial_i\vv_R\big\rangle
+\frac{Re}{1-\gamma}\big\langle\nabla\cdot\partial_i\mathfrak{R}+\partial_i\mathfrak{R},\partial_i\vv_R\big\rangle\\
&-\frac{4}{\ve^2}\big\langle\CJ_{\nn}(\CH^{\ve}_{\nn}(\partial_iQ_R)),\CH^{\ve}_{\nn}(\partial_iQ_R)\big\rangle
+\big\langle2\CJ_{\nn}(\nabla\partial_i\vv_R),\frac{1}{\ve}\CH^{\ve}_{\nn}(\partial_iQ_R)\big\rangle\\
&-\frac{4}{\ve^2}\big\langle[\CJ_{\nn}\CH^{\ve}_{\nn}, \partial_i]Q_R,\CH^{\ve}_{\nn}(\partial_iQ_R)\big\rangle
+\big\langle2[\CJ_{\nn},\partial_i]\nabla\vv_R,\frac{1}{\ve}\CH^{\ve}_{\nn}(\partial_iQ_R)\big\rangle
+\big\langle\partial_i\mathfrak{R},\frac{1}{\ve}\CH^{\ve}_{\nn}(\partial_iQ_R)\big\rangle.
\end{align*}
It is straightforward to obtain the estimates:
\begin{align*}
|\big\langle\DD_R:\partial_iM^{(4)}_{\widetilde{Q}},\partial_i\DD_R\big\rangle|
&\le C\|\DD_R\|_{L^2}\|\partial_i\DD_R\|_{L^2}\le \ve^{-2}(\delta_0\Ff+C\Ef),\\
\ve^{-1}\big\langle[\CM_{Q_0}\CH_\nn^{\ve}, \partial_i]Q_R,\nabla\partial_i\vv_R\big\rangle
&=\ve^{-1}\big\langle[\CM_{Q_0}\CH_\nn, \partial_i]Q_R
+\ve[\CM_{Q_0}\CL, \partial_i]Q_R, \nabla\partial_i\vv_R\big\rangle\\
&\le \ve^{-2}C\big(\|Q_R\|_{L^2}+\ve\|Q_R\|_{H^2}\big)\big\|\ve\nabla\partial_i\vv_R\big\|_{L^2}\\
&\le \ve^{-2}(\delta_0 \Ff+C\Ef),\\
-\ve^{-2}\big\langle[\CJ_{\nn}\CH^{\ve}_{\nn}, \partial_i]Q_R,\CH^{\ve}_{\nn}(\partial_iQ_R)\big\rangle
&\le \ve^{-2}C\big(\|Q_R\|_{L^2}+\ve\|Q_R\|_{H^2}\big)\big\|\CH^{\ve}_{\nn}(\partial_iQ_R)\big\|_{L^2}\\
&\le \ve^{-2}(\delta_0 \Ff+C\Ef),\\
\ve^{-1}\big\langle[\CJ_{\nn},\partial_i]\nabla\vv_R,\CH^{\ve}_{\nn}(\partial_iQ_R)\big\rangle
&\le \ve^{-1}\|\nabla\vv_R\|_{L^2}\|\CH^{\ve}_{\nn}(\partial_iQ_R)\|_{L^2}\le \ve^{-2}(\delta_0\Ff+C\Ef).
\end{align*}
Therefore, by the cancelation relation (\ref{eq:cancel}) with taking
$(A,B)=(\nabla\partial_i\vv_R, \partial_iQ_R)$,  we have
\begin{align}\label{estimate:H1}
& \frac{\ve^2Re}{1-\gamma}\big\langle\frac{\partial}{\partial t}\partial_i\vv_R,\partial_i\vv_R\big\rangle
+\ve\big\langle\frac{\partial}{\partial t}\partial_iQ_R, \CH^{\ve}_{\nn}(\partial_iQ_R)\big\rangle\nonumber\\
&+\frac{\ve^2\gamma}{1\!-\!\gamma}\|\nabla\partial_i\vv_R\|^2_{L^2}
+{4}\big\langle\CJ_{\nn}(\CH^{\ve}_{\nn}(\partial_iQ_R)),\CH^{\ve}_{\nn}(\partial_iQ_R)\big\rangle
\le \delta_0\Ff+C\Ef +C\|\ve\partial_i\mathfrak{R}\|_{L^2}^2.
\end{align}

{\bf Step 3. $H^2$-estimate}\\
Using (\ref{eq:QR})-(\ref{eq:imcomR}), we get
\begin{align*}
& \frac{Re}{1-\gamma}\big\langle\frac{\partial}{\partial t}\Delta\vv_R,\Delta\vv_R\big\rangle
+\big\langle\frac{\partial}{\partial t}\Delta Q_R,\frac{1}{\ve}\CH^{\ve}_{\nn}(\Delta Q_R)\big\rangle\\
=&-\frac{\gamma}{1\!-\!\gamma}\|\nabla\Delta\vv_R\|^2_{L^2}-\frac{1}{2}\big\langle\Delta\DD_R:M^{(4)}_{\widetilde{Q}},\Delta\DD_R\big\rangle
-\frac{2}{\ve}\big\langle\CM_{Q_0}(\CH_\nn^{\ve}(\Delta Q_R)),\nabla\Delta\vv_R\big\rangle\\
&+\frac{1}{2}\big\langle[\Delta, M^{(4)}_{\widetilde{Q}}:]\DD_R,\Delta\DD_R\big\rangle
+\frac{2}{\ve}\big\langle[\CM_{Q_0}\CH_\nn^{\ve}, \Delta]Q_R,\nabla\Delta\vv_R\big\rangle
+\frac{Re}{1-\gamma}\big\langle\nabla\cdot\Delta\mathfrak{R}+\Delta\mathfrak{R},\Delta\vv_R\big\rangle\\
&-\frac{4}{\ve^2}\big\langle\CJ_{\nn}(\CH^{\ve}_{\nn}(\Delta Q_R)),\CH^{\ve}_{\nn}(\Delta Q_R)\big\rangle
+\big\langle2\CJ_{\nn}(\nabla\Delta\vv_R),\frac{1}{\ve}\CH^{\ve}_{\nn}(\Delta Q_R)\big\rangle\\
&-\frac{4}{\ve^2}\big\langle[\CJ_{\nn}\CH^{\ve}_{\nn}, \Delta]Q_R,\CH^{\ve}_{\nn}(\partial_iQ_R)\big\rangle
+\big\langle2[\CJ_{\nn},\Delta]\nabla\vv_R,\frac{1}{\ve}\CH^{\ve}_{\nn}(\Delta Q_R)\big\rangle
+\big\langle\Delta\mathfrak{R},\frac{1}{\ve}\CH^{\ve}_{\nn}(\Delta Q_R)\big\rangle.
\end{align*}
Similar to Step 2, we can obtain
\begin{align*}
\big\langle[\Delta, M^{(4)}_{\widetilde{Q}}:]\DD_R,\Delta\DD_R\big\rangle
&\le C\ve^{-4}\|\ve^2\DD_R\|_{H^1}\|\ve^2\Delta \DD_R\|_{L^2}\le \ve^{-4}(\delta_0\Ff+C\Ef),\\
\ve^{-1}\big\langle[\CM_{Q_0}\CH_\nn^{\ve}, \Delta ]Q_R,\nabla\Delta \vv_R\big\rangle
&=\ve^{-1}\big\langle[\CM_{Q_0}\CH_\nn, \Delta ]Q_R
+\ve[\CM_{Q_0}\CL, \Delta ]Q_R, \nabla\Delta \vv_R\big\rangle\\
&\le \ve^{-4}C\big(\|\ve Q_R\|_{H^1}+\|\ve^2\CL (Q_R)\|_{H^1}\big)\big\|\ve^2\nabla\Delta \vv_R\big\|_{L^2}\\
&\le \ve^{-4}(\delta_0 \Ff+C\Ef),\\
-\ve^{-2}\big\langle[\CJ_{\nn}\CH^{\ve}_{\nn}, \Delta ]Q_R,\CH^{\ve}_{\nn}(\Delta Q_R)\big\rangle
&\le \ve^{-4}C\big(\|\ve Q_R\|_{H^1}+\|\ve^2\CL (Q_R)\|_{H^1}\big)\big\|\ve\CH^{\ve}_{\nn}(\Delta Q_R)\big\|_{L^2}\\
&\le \ve^{-4}(\delta_0 \Ff+C\Ef),\\
\ve^{-1}\big\langle[\CJ_{\nn},\Delta ]\nabla\vv_R,\CH^{\ve}_{\nn}(\Delta Q_R)\big\rangle
&\le \ve^{-4}\|\ve^2\nabla\vv_R\|_{H^1}\|\ve\CH^{\ve}_{\nn}(\Delta Q_R)\|_{L^2}\le \ve^{-4}(\delta_0\Ff+C\Ef).
\end{align*}
Using (\ref{eq:cancel}) again, we have
\begin{align}\label{estimate:H2}
& \frac{\ve^4Re}{1-\gamma}\big\langle\frac{\partial}{\partial t}\Delta \vv_R,\Delta \vv_R\big\rangle
+\ve^3\big\langle\frac{\partial}{\partial t}\Delta Q_R, \CH^{\ve}_{\nn}(\Delta Q_R)\big\rangle\nonumber\\
&+\frac{\ve^4\gamma}{1\!-\!\gamma}\|\nabla\Delta \vv_R\|^2_{L^2}
+{4}\ve^2\big\langle\CJ_{\nn}(\CH^{\ve}_{\nn}(\Delta Q_R)),\CH^{\ve}_{\nn}(\Delta Q_R)\big\rangle
\le \delta_0\Ff+C\Ef +C\|\ve^2\Delta \mathfrak{R}\|_{L^2}^2.
\end{align}

{\bf Step 4. The completion of energy estimate}\\
Recalling (\ref{eq:Hn}) we get
\begin{align*}
\CH^{\ve}_{\nn}(Q)=&\psi_1(\nn\nn-\frac13\II)(\nn\nn:Q)+\psi_2(-Q+\nn\nn\cdot Q+Q\cdot\nn\nn
-\frac23\II\nn\nn:Q)+\ve\CL (Q).
\end{align*}
With $Q_R:\II=\text{tr}~Q_R=0$, it yields
\begin{align*}
 \frac{1}{\ve}\frac{d}{dt}\big\langle Q_R,\CH^{\ve}_{\nn}(Q_R)\big\rangle=&\frac{2}{\ve}
 \Big\langle \frac{\partial}{\partial t}Q_R,\CH^{\ve}_{\nn}(Q_R)\Big\rangle+
 \frac{1}{\ve}\Big\langle Q_R,\psi_1(\nn\nn:Q_R)\partial_t(\nn\nn)
 +\psi_1(\partial_t(\nn\nn):Q_R)\nn\nn\\
 &+\psi_2\big(\partial_t(\nn\nn)\cdot Q_R+Q_R\cdot\partial_t(\nn\nn)\big)\Big\rangle\\
 =&\frac{2}{\ve}
 \langle \frac{\partial}{\partial t}Q_R,\CH^{\ve}_{\nn}(Q_R)\rangle
 +\frac{2}{\ve}\Big\langle Q_R,\psi_1(\partial_t(\nn\nn):Q_R)\nn\nn
 +\psi_2\partial_t(\nn\nn)\cdot Q_R\Big\rangle.
\end{align*}
Lemma \ref{keylemma} tell us that
\begin{align*}
&\frac{2}{\ve}\Big\langle Q_R,\psi_1(\partial_t(\nn\nn):Q_R)\nn\nn
 +\psi_2\partial_t(\nn\nn)\cdot Q_R\Big\rangle\\
&\leq\delta||\frac{1}{\ve}\CH^{\ve}_{\nn}(Q)||^2_{L^2}
 +C_{\delta}\big(\frac{1}{\ve}\langle\CH^{\ve}_{\nn}(Q),Q\rangle+||Q||^2_{L^2}\big).
\end{align*}
Thus we have
\begin{eqnarray*}
 \frac{1}{2\ve}\frac{d}{dt}\langle Q_R,\CH^{\ve}_{\nn}(Q_R)\rangle
 \leq\langle \frac{\partial}{\partial t}Q_R,\CH^{\ve}_{\nn}(Q_R)\rangle+\delta\Ff+C\Ef.
\end{eqnarray*}
Similarly, the following inequalities hold
\begin{eqnarray*}
 &&\frac{\ve}{2}\frac{d}{dt}\langle \partial_iQ_R,\CH^{\ve}_{\nn}(\partial_iQ_R)\rangle
 \leq\ve\langle \frac{\partial}{\partial t}\partial_iQ_R,\CH^{\ve}_{\nn}(\partial_iQ_R)\rangle+\delta\Ff+C\Ef,\\
 &&\frac{\ve^3}{2}\frac{d}{dt}\langle \Delta Q_R,\CH^{\ve}_{\nn}(\Delta Q_R)\rangle
 \leq\ve^3\langle \frac{\partial}{\partial t}\Delta Q_R,\CH^{\ve}_{\nn}(\Delta Q_R)\rangle+\delta\Ff+C\Ef.
\end{eqnarray*}
Together with (\ref{estimate:Q0})-(\ref{estimate:H2}), we arrive at
\begin{align*}
 \frac12\frac{d }{d t}\Ef(t)+\Ff(t)\le \delta\Ff+C_\delta\Ef+\|\mathfrak{R}\|_{L^2}^2
 +\|\ve\nabla\mathfrak{R}\|_{L^2}^2+\|\ve^2\Delta\mathfrak{R}\|_{L^2}^2.
\end{align*}
Recalling that $\mathfrak{R}$ denotes {\it good terms} with
\begin{align*}
\|\mathfrak{R}\|_{L^2} +\ve\|\nabla\mathfrak{R}\|_{L^2}+\ve^2\|\Delta\mathfrak{R}\|_{L^2}
\le C(\ve E)(1+E+\ve F)+\ve f(E),
\end{align*}
and Corollary (\ref{corol:EF}), we have
\begin{align*}
 \frac12\frac{d }{d t}\Ef(t)+\Ff(t)\le \delta\Ff+C_\delta\Ef+ C(\ve^2\Ef)(1+\Ef+\ve^2\Ff)+\ve^2f(\Ef).
\end{align*}

Taking $\delta$ enough small leads to the Proposition \ref{prop:energy}.
\end{proof}
\begin{lemma}\label{keylemma}
For any $\delta>0$, there exists a constant $C=C(\delta,||\nabla_{t,\xx}\nn||_{L^{\infty}},||\nabla\nn_{t}||_{L^{\infty}})$
such that for any $Q\in  \mathbb{R}^{3\times3}_{sym,0}$, it holds that
\begin{align*}
 \frac{1}{\ve}\langle\partial_t(\nn\nn)\cdot Q,Q\rangle
 \leq&\delta||\frac{1}{\ve}\CH^{\ve}_{\nn}(Q)||^2_{L^2}
 +C_{\delta}\big(\frac{1}{\ve}\langle\CH^{\ve}_{\nn}(Q),Q\rangle+||Q||^2_{L^2}\big),\\
 \frac{1}{\ve}\langle Q:\partial_t(\nn\nn),Q:\nn\nn\rangle
 \leq&\delta||\frac{1}{\ve}\CH^{\ve}_{\nn}(Q)||^2_{L^2}
 +C_{\delta}\big(\frac{1}{\ve}\langle\CH^{\ve}_{\nn}(Q),Q\rangle+||Q||^2_{L^2}\big).
\end{align*}
\end{lemma}
The proof of Lemma \ref{keylemma} can be found in \cite{WZZ4}.

\subsection{Proof of Theorem \ref{thm:main2}}
Given the initial data $(\vv_0^\ve,Q_0^\ve)\in H^2\times H^3$, thanks to Theorem \ref{thm:main1}
there exists a maximal time $T_\ve>0$ and a unique solution $(\vv^\ve,Q^\ve)$ of the system (\ref{eqepsiQ1})-(\ref{eqepsiQ3}) such that
\beno
\vv^\ve\in C([0,T_\ve);H^2)\cap L^2(0,T_\ve;H^3),\quad Q^\ve\in C([0,T_\ve);H^3).
\eeno
Now we prove that $T_\ve\ge T$. Suppose it is not. By Proposition \ref{prop:Hilbert}, the solution has the expansion
\begin{align*}
&\vv^\ve=\vv_0+\ve\vv_1+\ve^2\vv_2+\ve^3\vv_R^\ve,\\
&Q^\ve=Q_0+\ve Q_1+\ve^2Q_2+\ve^3Q_3+\ve^3Q_R^\ve.
\end{align*}
For the remainder $(\vv_R^\ve,Q_R^\ve)$, we infer from Proposition \ref{prop:energy} that
\begin{align}
\frac{d }{d t}\Ef(t)+\Ff(t)\le
 C(\ve \Ef)\big(1+\Ef\big)+\ve f(\Ef)+C(\ve \Ef)\ve\Ff,\non
\end{align}
for any $t\in [0,T_\ve]$. Thanks to the assumptions of Theorem \ref{thm:main2}, we know that
$$\Ef(0)\le C_1\Big(\|\vv_{I,R}^\ve\|_{H^2}+\|Q_{I,R}^\ve\|_{H^3}+\ve^{-1}\|\CP^{out}(Q^\ve_{I,R})\|_{L^2}\Big)\le C_1 E_0.$$
Let $E_1=(2+C_1E_0){e}^{T}-2>\Ef(0)$, and
$$T_1=\sup\{t\in[0,T_\ve]: \Ef(t)\le E_1\}.$$
Thus, if we take $\ve_0$ small enough such that
$$C(\ve_0 E_1)\le 1,\quad\ve_0f(E_1)\le 1,\quad\ve_0\le 1/2,$$
then for $t\le T_1$, it holds that
\begin{align}
\frac{d}{dt}\Ef(t)\le 2+\Ef.
\end{align}
If $T_\ve<T$, Gronwall's inequality gives that for $t\le T_1$,
\beno
\Ef(t)\le {e}^{t}(2+C_1E_0)-2 < E_1,
\eeno
which implies $T_1=T_\ve$ and at time $T_\ve$, $(\vv_\ve, Q_\ve)\in H^2\times H^3$, which contradict with our assumption.
Thus $T\le T_\ve$, and $\Ef(t)\le E_1$ for $t\in[0,T]$. Then Theorem \ref{thm:main2} follows.

\section{Appendix}
\subsection{Some basic estimates in Sobolev spaces}

The following product estimates and commutator estimates are  well-known, see \cite{Triebel} for example, and frequently used in this paper.

\begin{lemma}\label{lem:product}
Let $s\ge 0$. Then for any multi-index $\al, \be, \gamma, \delta$, there holds
\begin{align*}
\|\partial^\alpha f\partial^\beta g\|_{H^s}\le &~ C\big(\|f\|_{L^\infty}\|g\|_{H^{s+|\al|+|\be|}}
+\|g\|_{L^\infty}\|f\|_{H^{s+|\al|+|\be|}}\big);\\
\|\partial^\alpha f\partial^\beta g\|_{H^s}\le &~ C\|f\|_{H^{s+|\al|+|\gamma|}}\|g\|_{H^{s+|\be|+|\delta|}},
\quad\text{ if } s+|\gamma|+|\delta|\ge 2.
\end{align*}
In particular, we have
\begin{align*}
\|fg\|_{H^s}\le &~ C\big(\|f\|_{L^\infty}\|g\|_{H^{s}}+\|g\|_{L^\infty}\|f\|_{H^{s}}\big); \\
\|fg\|_{H^s}\le &~ C\|f\|_{H^s}\|g\|_{H^{s}},\quad\text{ if } s\ge 2;\\
\|fg\|_{H^k}\le &~ C\min\{ \|f\|_{H^k}\|g\|_{H^{2}}, \|f\|_{H^2}\|g\|_{H^{k}}\},\quad\text{ if } 0\le k\le 2.
\end{align*}
\end{lemma}
\begin{lemma}\label{lem:composition}
Let $s\ge 0$ and $F(\cdot)\in C^\infty(\mathbb{R}^d)$ with $F(0)=0$. Then
\beno
\|F(f)\|_{H^s}\le C(\|f\|_{L^\infty})\|f\|_{H^s}.
\eeno
\end{lemma}

\begin{lemma}\label{lem:commutator}
Let $a$ be a multiple index. There holds
\beno
\|\big[\partial^a, g\big]f\|_{L^2}\le C\big(\|\nabla g\|_{L^\infty}\|f\|_{H^{|a|-1}}
+\|\nabla g\|_{H^{|a|-1}}\|f\|_{L^\infty}\big).
\eeno
Moreover, if $|a|\ge 2$, it holds
\begin{align*}
&\|\big[\partial^a, g\big]f\|_{L^2}\le C\|g\|_{H^{|a|+1}}\|f\|_{H^{|a|-1}},\\
&\|\big[\partial^{a+1}, g\big]f\|_{L^2}\le C\|g\|_{H^{|a|+1}}\|f\|_{H^{|a|}}.
\end{align*}
\end{lemma}

\begin{lemma}\label{lem:difference}
Let $\Omega$ be a convex domain in $\mathbb{R}^d$ and $k\ge 0$ be an integer.
$F(\cdot)\in C^\infty(\Omega)$ and $k'=\max\{k,2\}$. Then
\beno
\|F(u)-F(v)\|_{H^k}\le C(\|u\|_{L^\infty},\|v\|_{L^\infty})(1+\|u\|_{H^{k'}}+\|v\|_{H^{k'}} )\|u-v\|_{H^k}.
\eeno
\end{lemma}
\begin{proof} We may assume that $F'(0)=0$, since if not, we can consider $G(u)=F(u)-u\cdot F'(0)$.
By the fact that
$$F(u)-F(v)=(u-v)\cdot\int_{0}^1F'(v+t(u-v))dt,$$
we have
\begin{align*}
\|F(u)-F(v)\|_{L^2}&~\le \|u-v\|_{L^2}\sup_{t\in[0,1]}\|F'(v+t(u-v))\|_{L^\infty}\\
&~ \le C(\|u\|_{L^\infty},\|v\|_{L^\infty})\|u-v\|_{L^2} ,\\
\|\nabla(F(u)-F(v))\|_{L^2}&~\le \|\nabla(u-v)\|_{L^2}\sup_{t\in[0,1]}\|F'(v+t(u-v))\|_{L^\infty}\\
&~\quad+\|u-v\|_{H^1}\sup_{t\in[0,1]}\|\nabla(F'(v+t(u-v)))\|_{H^1} \\
&~\le C(\|u\|_{L^\infty},\|v\|_{L^\infty})(\|u\|_{H^2}+\|v\|_{H^2})\|u-v\|_{H^1} ,\\
\end{align*}
and for $k\ge 2$,
\begin{align*}
\|F(u)-F(v)\|_{H^k}\le&~ C\Big(\|u-v\|_{L^\infty}\sup_{t\in[0,1]}\|F'(v+t(u-v))\|_{H^k}\\
&\qquad+\|u-v\|_{H^k}\sup_{t\in[0,1]}\|F'(v+t(u-v))\|_{L^\infty}\Big)\\
\le &~C(\|u\|_{L^\infty},\|v\|_{L^\infty})(1+\|u\|_{H^k}+\|v\|_{H^k})\|u-v\|_{H^k}.
\end{align*}
Here, we have used the following estimate which is induced by Lemma \ref{lem:composition}:
\begin{align*}
\|F'(v+t(u-v))\|_{H^k}&~\le C(\|v+t(u-v)\|_{L^\infty})\|v+t(u-v)\|_{H^k} \\
&~\le C(\|u\|_{L^\infty},\|v\|_{L^\infty})(\|u\|_{H^k}+\|v\|_{H^k}).
\end{align*}
This concludes the proof.
\end{proof}

\bigskip

\noindent {\bf Acknowledgments.} The authors would like to thank Professor Zhifei Zhang for his helpful discussion.
W. Wang is supported by China Postdoctoral Science Foundation under Grant 2013M540010 and 2014T70008.
P. Zhang is partly supported by NSF of China under Grant 21274005, 11421110001 and 11421101.


\begin{thebibliography}{10}

\bibitem{ADL1} {H. Abels, G. Dolzmann and Y. Liu}, {\em Well-posedness of a fully-coupled Navier-Stokes/$Q$-tensor system
with inhomogeneous boundary data}, SIAM J. Math. Anal., 46(2014), 3050-3077.

\bibitem{ADL2} {H. Abels, G. Dolzmann and Y. Liu}, {\em Strong solutions for the Beris-Edwards model for nematic
liquid crystals with homogeneous Dirichilet boundary conditions}, arXiv:1312.5988(2013).

\bibitem{BCD} {H. Bahouri, J-Y. Chemin and R. Danchin}, {\em Fourier Analysis and Nonlinear Partial Differential Equations.
Fundamental Principles of Mathematical Sciences}, vol. 343. Springer, Heidelberg, 2011.

 \bibitem{BM} { J. M. Ball and A. Majumdar}, {\em Nematic liquid crystals: from Maier-Saupe to a continuum
 theory}, Mol.Cryst.Liq.Cryst., 525(2010),1-11.

 \bibitem{BE} A. N. Beris and B. J. Edwards,{\it  Thermodynamics of flowing systems
with internal microstructure}, Oxford Engrg. Sci. Ser. 36, Oxford University Press, Oxford, New York, 1994.

\bibitem{DG} P. G. De Gennes, {\it The physics of liquid crystals}, Clarendon Press, Oxford, 1974.

\bibitem{EZ} W. E and P. Zhang, {\it A molecular kinetic theory of inhomogeneous liquid crystal flow
and the small Deborah number limit}, Methods and Applications of
Analysis, {13}(2006), 181-198.

\bibitem{E-61} J. Ericksen, {\it Conservation laws for liquid crystals},
Trans. Soc. Rheol. , {5}(1961), 22-34.

\bibitem{FS} I. Fatkullin and V. Slastikov, {\it Critical points of the Onsager functional on a sphere},
Nonlinearity, 18(2005), 2565-80.

\bibitem{Feng} J. Feng, C. V. Chaubal and L. G. Leal, {\it  Closure approximations for the Doi theory: Which
to use in simulating complex flows of liquid-crystalline polymers?},  Journal of Rheology,
42(1998), 1095-1109.

\bibitem{FLS} J. J. Feng, G. L. Leal and G. Sgalari, {\it A theory for flowing nematic polymers with
orientational distortion}, Journal of Rheology, 44(2000), 1085-1101.

\bibitem{WZZ3}J. Han, Y. Luo, W. Wang, P. Zhang and Z. Zhang {\it From microscopic theory to
macroscopic theory: systematic study on modeling for liquid crystals},  Arch. Ration. Mech. Anal., online, DOI: 10.1007/s00205-014-0792-3.

\bibitem{HD} J. Huang and S. Ding, {\it Global well-posedness for a coupled incompressible Navier-Stokes
and Q-tensor system}, arXiv:1405.1863.

\bibitem{HLW} J. Huang, F. H Lin and C. Wang, {\it Regularity and existence of global solutions to the Ericksen-Leslie system
in $\mathbb R^2$}, Commun. Math. Phys., 331(2014), 805-850.

\bibitem{KD} N. Kuzuu and M. Doi,
{\it Constitutive equation for nematic liquid crystals under weak
velocity gradient derived from a molecular kinetic equation,}
Journal of the Physical Society of Japan, 52(1983), 3486-3494.

\bibitem{Les} F. M. Leslie, {\it Some constitutive equations for liquid crystals}, Arch. Ration. Mech. Anal., 28
(1968), 265-283.

\bibitem{LL} F.-H. Lin and C. Liu, {\it Existence of solutions for the Ericksen-Leslie system},
Arch. Ration. Mech. Anal., 154(2000), 135-156.

\bibitem{LZZ} H. Liu, H. Zhang and P. Zhang, {\it Axial symmetry and classification of stationary solutions of
Doi-Onsager equation on the sphere with Maier-Saupe potential}, Comm. Math. Sci., 3(2005), 201-218.

\bibitem{MZ} A. Majumdar and  A. Zarnescu, {\it Landau-De Gennes theory of nematic liquid
crystals: the Oseen-Frank limit and beyond},  Arch. Ration. Mech. Anal., 196(2010), 227-280.

\bibitem{MN} N. J. Mottram and C. Newton, {\it Introduction to $Q$-tensor theory}. University of Strathclyde, Department of Mathematics, Research Report, 10(2004).

\bibitem{PZ1} M. Paicu and  A. Zarnescu, {\it Energy dissipation and regularity for a coupled Navier-Stokes and $Q$-tensor system},
 Arch. Ration. Mech. Anal., 203 (2012), 45--67.

\bibitem{PZ2} M. Paicu and  A. Zarnescu, {\it Global existence and regularity for the full coupled Navier-Stokes and $Q$-tensor system},
SIAM J. Math. Anal., 43 (2011), 2009--2049.

\bibitem{Parodi} O. Parodi, {\it Stress tensor for a nematic liquid crystal,}
Journal de Physique, 31 (1970), 581-584.


\bibitem{QS} T. Qian and P. Sheng, {\it Generalized hydrodynamic equations for nematic liquid crystals}, Phys. Rev. E, 58 (1998), 7475-7485.

\bibitem{Triebel} H. Triebel,
{\it Theory of function spaces}. Monographs in Mathematics,
 Birkh\"{a}user Verlag, Basel, Boston, 1983.


\bibitem{WW} M. Wang and W. Wang, {\it Global existence of weak solution for the 2-D Ericksen-Leslie system},
Calc. Var. Partial Differ. Equ., online, DOI: 10.1007/s00526-013-0700-y.

\bibitem{WZZ1} W. Wang, P. Zhang and Z. Zhang, {\it The small Deborah number limit of the Doi-Onsager equation  to the
Ericksen-Leslie equation}, accepted by Comm. Pure Appl. Math.

\bibitem{WZZ2} W. Wang, P. Zhang and Z. Zhang, {\it Well-posedness of the Ericksen-Leslie system},
Arch. Ration. Mech. Anal., 206(2012), 953-995.

\bibitem{WZZ4} {W. Wang, P. Zhang and Z. Zhang}, {\em Rigorous derivation from Landau-de Gennes theory
 to Ericksen-Leslie theory}, SIAM. Math. Anal., to appear.

\bibitem{WXL} H. Wu, X. Xu and C. Liu: {\it On the general Ericksen Leslie system: Parodi¡¯s relation, well-posedness and stability}. Arch. Ration. Mech. Anal., 208(2013), 59-107
    
\end{thebibliography}
\end{document}